\documentclass[11pt]{amsart}
\usepackage[utf8]{inputenc}
\usepackage[normalem]{ulem}

\usepackage[figuresright]{rotating}
\usepackage{amsmath, amssymb, amsthm,color, graphicx, enumitem, upgreek} 
\usepackage[unicode]{hyperref}
\usepackage[artemisia]{textgreek}
\usepackage{epsfig,color,mathrsfs,amsbsy,amsfonts,amsmath,amssymb,amsthm,fullpage}
\usepackage{caption}
\usepackage{subcaption}

\usepackage[all]{xy}
\usepackage{pb-diagram,epsfig,tikz}

\definecolor{purple}{rgb}{0.59, 0.44, 0.84}

\definecolor{turk}{rgb}{0.2, .60, .50}

\newtheorem{theorem}{Theorem}
\newtheorem{prop}{Proposition}
\newtheorem{cor}{Corollary}

\newtheorem{defn}{Definition}
\newtheorem{lemma}{Lemma}

\newcounter{example}
\newenvironment{example}[1][]{\refstepcounter{example}\par\medskip
   \noindent \textbf{Example~\theexample. #1} \rmfamily}{\medskip}
\newcounter{remark}
\newenvironment{remark}[1][]{\refstepcounter{remark}\par\medskip
   \noindent \textbf{Remark~\theremark. #1} \rmfamily}{\medskip}
\def \Frob{\mr{Frob}}

\newcommand{\mc}[1]{\mathcal{#1}}

\newcommand{\mfr}[1]{\mathfrak{#1}}
\newcommand{\mr}[1]{\mathrm{#1}}

\newcommand{\ZZ}{\mathbf{Z}}
\newcommand{\FF}{\mathbf{F}}
\newcommand{\RR}{\mathbf{R}}
\newcommand{\QQ}{\mathbf{Q}}
\newcommand{\CCC}{\mathbf{C}}

\def\F{\FF}
\def\C{\CCC}

\def \fH{\mathfrak H}

\def\P{\mathbb P}
\def\Q{\mathbb Q}

\def\Z{\mathbb Z}

\def\M#1#2#3#4{\begin{pmatrix}#1&#2\\#3&#4\end{pmatrix}}
\def\SM#1#2#3#4{\left(\begin{smallmatrix}#1&#2\\#3&#4\end{smallmatrix}	\right)}

\def \GL{\rm{GL}}
\def \SL{\rm{SL}}

\def\eps{\varepsilon}
\def \l {\lambda}
\def\ol{\overline}
\def\CC#1#2{\binom {#1}{#2}}
\def\({\left(}
\def\){\right)}
\def \f12{\frac12}
\def\G{\Gamma}

\def\PSL{\mbox{PSL}}
\def\SL{\mbox{SL}}
\def\SL2Z{\mr{SL}_(\ZZ)}
\def\SL2R{\mr{SL}_2(\mathbf{R} )}
\def\PSL2R{\mr{PSL}_2(\mathbf{R} )}
\def\PSL2Z{\mr{PSL}_2(\ZZ )}
\def \M {{\rm{M}}}
\newcommand*\HYPERskip{&}
\catcode`,\active
\newcommand*\pFq{
	\begingroup
	\catcode`\,\active
	\def ,{\HYPERskip}%
	\doHyper
}
\catcode`\,12
\def\doHyper#1#2#3#4#5{%
	\, _{#1}F_{#2}\left[\begin{matrix}#3 \smallskip \\  #4\end{matrix} \; ; \; #5\right]%
	\endgroup
}

\catcode`,\active

\catcode`\,12

\catcode`,\active

\catcode`\,12

\catcode`,\active

\catcode`\,12

\newcommand*\HYPERpp{&}
\catcode`,\active
\newcommand*\pPPq{
	\begingroup
	\catcode`\,\active
	\def ,{\HYPERpp}%
	\doHyperFpp
}
\catcode`\,12
\def\doHyperFpp#1#2#3#4#5{%
	\, _{#1}{\mathbb P}_{#2}\left[\begin{matrix}#3 \smallskip \\  #4\end{matrix} \; ; \; #5\right]%
	\endgroup
}

\def \f{\frac}

\def \bk {\color{black}}

\def \P{\mathbb P}

\def \Tr {\mr{Tr}}
\def \g {\mathfrak g}
\def \p {\mathfrak p}
\def \P {\mathfrak P}

\def \ker \text{Ker}

\def\k{\kappa}

\begin{document}

\title{Traces of Hecke operators via hypergeometric character sums}
\author{Jerome W. Hoffman, Wen-Ching Winnie Li, Ling Long, and Fang-Ting Tu}
\keywords{Frobenius traces, arithmetic triangle groups,   hypergeometric character sums, automorphic forms,  Hecke operators}
\subjclass{11F11, 11F80, 11L05, 33C20}

\begin{abstract}
  In this paper we obtain explicit formulas for the traces of Hecke operators on spaces of cusp forms in certain instances related to 
arithmetic triangle groups.  These expressions are in terms of hypergeometric character sums over finite fields, a theory 
developed largely by Greene \cite{Greene},  Katz \cite{Katz}, Beukers-Cohen-Mellit \cite{BCM}, and  
Fuselier-Long-Ramakrishna-Swisher-Tu \cite{Win3X}. Our approach, in contrast to the previous works, is uniform and more geometric, and it works equally well for forms on elliptic modular curves and Shimura curves. The same method can be applied to obtain eigenvalues of Hecke operators as well. 
\end{abstract}

\maketitle

\tableofcontents

\section{Introduction and main results} 
\subsection{Introduction}
\label{S:mot}
In this paper we obtain explicit formulas for the traces of Hecke operators on spaces of cusp forms in certain instances related to 
arithmetic triangle groups.  These expressions are in terms of hypergeometric character sums over finite fields, a theory 
developed largely by Greene \cite{Greene},  Katz \cite{Katz}, Beukers-Cohen-Mellit \cite{BCM}, and  
Fuselier-Long-Ramakrishna-Swisher-Tu \cite{Win3X}.

Earlier works in this direction include  \cite{AO00} by Ahlgren and Ono for $\Gamma_0(8)$, \cite{Ahlgren01} by Ahlgren for $\Gamma_0(4)$, \cite{FOP2} by Frechette, Ono and Papanikolas for newforms of $\Gamma_0(8)$, \cite{Fuselier} by Fuselier for $\Gamma_0(1)$ and \cite{lennon2, lennon1} by Lennon for $\Gamma_0(3)$ and $\Gamma_0(9)$. The method in these papers 
is to combine (i) counting formulas for elliptic curves over finite fields, due to Schoof \cite{schoof}, and (ii) 
the Selberg trace formula. 
That proof is a descendant of the paper of Ihara \cite{Ihara67} wherein he computed the 
traces of Hecke operators for $\mr{SL}(2, \ZZ)$. Ihara did not use character sums and his counting 
of elliptic curves was based on results of Deuring. The paper \cite{Ono-Saad} by Ono and Saad did it for $\G_0(2)$ and $\G_0(4)$, following Zagier's work \cite{zagier2000modular} for $\G_0(1)$. Their method uses the Rankin-Cohen brackets of Zagier's mock modular form.

In contrast to the works above, our approach, inspired by \cite{LLT2} by the last three authors and Scholl \cite{Sch88}, is more geometric. It is based on Eichler-Shimura theory  \cite{Shimura-introduction} as interpreted by  Deligne  \cite{Deligne-mf} for elliptic modular groups and by Kuga-Shimura \cite{KS2}, Shimura \cite{Shimura68} and Ohta \cite{Ohta82,Ohta83,Ohta81} for the quaternion case.  Let $\G$ be an  arithmetic triangle group.  
Shimura's theory of canonical models implies that there is a smooth and projective algebraic curve $X_\G$  defined over 
an number field attached to it.  For the groups considered in this paper, the number field is $\QQ$. The key point is that the trace of the Hecke operator $T_p$ acting on $S_{k+2}(\G)$, the space of cusp forms \footnote {By a cusp form we mean a holomorphic modular form which vanishes at all cusps of $\G$. In particular, if $\G$ is cocompact, then a cusp form is just a modular form.}
of weight $k+2$, is equal to the trace of Frobenius $\Frob_p$ acting 
on the $H^1$ of this curve $X_\G$ with coefficients in a suitable $\ell$-adic sheaf, denoted by $V^k(\G)_\ell$; see \S \ref{S:autosh}  for the setup. 
The Grothendieck -Lefschetz theorem implies
\begin{equation}
\label{E:mot1}
-\mr{Tr} (T_p\mid S_{k+2} (\Gamma)) = -\mr{Tr} (\mr{Frob}_p\mid  H^1 (X_{\Gamma} , V^k (\Gamma)_{\ell} ) ) =
\sum_{\l \in X_{\Gamma}(\FF_p)} \mr{Tr}(\mr{Frob}_\l \mid (V^k(\Gamma)_{\ell})_{\bar{\lambda}} ), 
\end{equation}
and the main results of this paper express the terms on the  right-hand side in terms of hypergeometric character sums attached to a hypergeometric datum $HD(\G)$. 
\medskip

The novelty of this approach is twofold:
\begin{itemize}
\item[1.] We give a unified treatment for $\G$ elliptic modular and $\Gamma$ arising from quaternion algebras, specifically the quaternion algebra over $\QQ$ of discriminant 6. The latter continues the study of 4-dimensional Galois representations admitting quaternionic multiplications as 
initiated in \cite{ALLL,LLL}. We make use of a certain family of abelian varieties with quaternion structures proposed in \cite{Win3a} and  we extend some of the results of \cite{LLT2}. Moreover, if $\G'$ is a congruence subgroup of a group $\G$ studied by our approach such that $X_{\G'}$ is defined over $\QQ$ and the canonical covering map $X_{\G'} \to X_\G$ is explicit, then our method also expresses the traces of the Hecke operators on $S_{k+2}(\G')$ in terms of hypergeometric character sums attached to $HD(\G)$. In particular, we obtain the Hecke traces for the groups cited above this way. 

\item[2.] Suppose $S_{k+2}(\Gamma)$ is $m$-dimensional. Our method used to compute the right-hand side of (\ref{E:mot1}) can be easily extended to obtain the traces of $(\mr{Frob}_p)^r$, $1 \le r\le m$, which in turn give rise to the eigenvalues of $T_p$ on $S_{k+2}(\G)$. Modular forms for a congruence subgroup arising from a quaternion algebra over $\QQ$ do not admit the usual Fourier expansion. This has been a handicap to understanding the arithmetic of modular forms on Shimura curves, as showcased in Yang's paper \cite{Yang-Schwarzian}, where examples of Hecke eigenvalues for small primes are computed. Our method opens up a door to explore this territory.
\end{itemize}

\subsection{Main results}\label{SS:mainresults} The arithmetic triangle groups $$(e_1,e_2,e_3)=\langle g_1,g_2 \mid g_1^{e_1} = g_2^{e_2}=(g_1g_2)^{e_3}=id \rangle$$ have been classified by Takeuchi, 
\cite{Takeuchi-triangle}, \cite{Takeuchi-classify}, and each can be realized as a discrete subgroup of $\text{PSL}_2(\RR)$ acting on the complex upper half plane $\fH$.  The groups we consider are of two types: (1) those with noncompact fundamental regions, and therefore cusps, are commensurable with  $\mr{SL}_2( \ZZ)$, and (2) those with compact fundamental regions, and therefore no cusps, are commensurable with $ O^1_{B_6}$, the group of norm-1 elements of a  maximal order (which is unique up to conjugation) in the quaternion algebra $B_6$ over $\QQ$ of discriminant 6. We also consider related groups gotten by adding Atkin-Lehner involutions. 
Note that modular forms on these quaternion groups are more difficult to understand because of the lack of $q$-expansions. In those cases, the 
Jacquet-Langlands (JL) correspondence allows one to relate these modular forms to cusp forms on congruence subgroups of  $\mr{SL}_2( \ZZ)$, and 
we have utilized the JL correspondence to verify the calculations done in this paper related to the quaternions. Diagrams of the triangle groups and 
their interrelationships appear in Figures 1 and 2. In all cases, the curves $X_{\G}$ are of genus 0 and defined over $\QQ$.

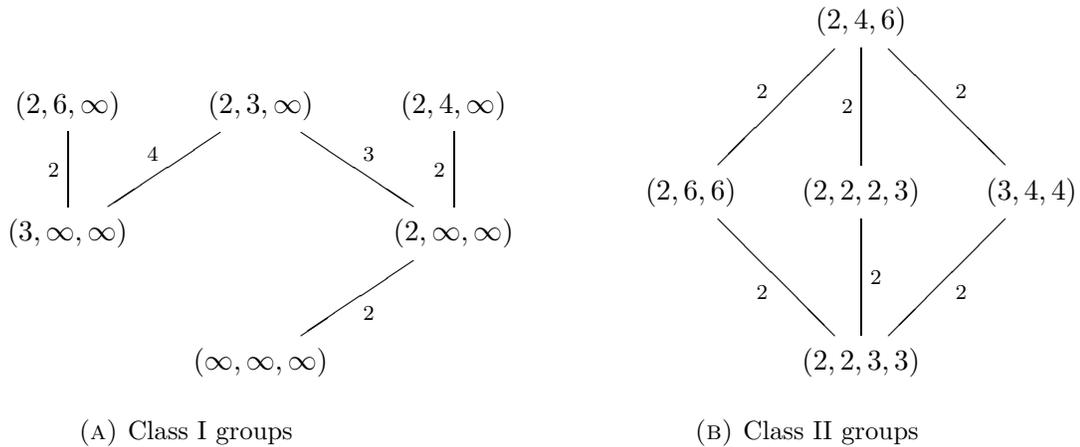
\begin{figure}[h]
\centering
\begin{subfigure}{.5\textwidth}
  \centering
 $$
\begin{diagram}
  \node[2]{(2,6,\infty)}    \arrow{s,l,-}{2} 
  \node[1]{ (2,3,\infty)}   \arrow{sw,l,-}{4} \arrow{se,l,-}{3} 
   \node[1]{(2,4,\infty)}   \arrow{s,l,-}{2}  \\ 
   \node[2]  { (3,\infty,\infty)}  
   \node[2]   { (2,\infty,\infty)} \arrow{sw,r,-}2 \\ 
   \node[3]{ (\infty,\infty,\infty)}
  \end{diagram}
$$
     \caption{Class I groups}
     \label{fig:class-I}
\end{subfigure}%
\begin{subfigure}{.5\textwidth}
  \centering
  $$
\begin{diagram}
\node[3]{(2,4,6)} \arrow{sw,l,-}{2} \arrow{s,l,-}{2} \arrow{se,l,-}{2} \\
\node[2]   {(2,6,6)} \arrow{se,r,-}{2} 
\node[1]   {(2,2,2,3)} \arrow{s,r,-}{2}
\node[1]{(3,4,4)} \arrow{sw,r,-}{2} \\
\node[3] {(2,2,3,3)}
\end{diagram}
$$
    \caption{Class II groups}
    \label{fig:class-II}
\end{subfigure}

\label{fig:test}
\end{figure}

On the right, the group (2,4,6) arises from the indefinite quaternion algebra $B_6=\left(\frac{-1,3}\QQ\right)$ with  a maximal order $O_{B_6}=\ZZ+\ZZ I+\ZZ J+\ZZ\frac{1+I+J+IJ}2$, where $I^2=-1,J^2=3, IJ=-JI$. 
Its elliptic points of order $2$, $4$, $6$ are the fixed points of the Atkin-Lehner involutions $w_6$, $w_2$, $w_3$, induced from the elements $3I+IJ$, $1+I$, $(3+3I+J + IJ)/2$ of $O_{B_6}$, respectively.  As a group, (2,4,6)  is generated by the reduced norm 1 elements of $O_{B_6}$ together with $w_2,w_3$ and $w_6$ modulo the center.

To all but two  triangle groups $\Gamma$ on the list, we associate hypergeometric data $HD(\G)=\{\alpha(\G),\beta(\G)\}$ defined over $\QQ$, for which Beukers-Cohen-Mellit \cite{BCM} introduced hypergeometric character sums $H_p(HD(\G), \lambda)$ for $\lambda \in \FF_p^\times$ (cf. \S \ref{ss:HG-FF}). We shall express the trace of $T_p$ on $S_{k+2}(\G)$ in terms of these hypergeometric character sums. Our main result is as follows. For integers $m \ge 1$, let $F_m(S,T)$ be the degree-$m$ polynomial in $S$ and $T$ defined by the recursive relation
 \begin{eqnarray}\label{eq:F(S,T)}
  F_{m+1}(S, T) = (S-T)F_m(S,T) - T^2F_{m-1}(S,T), \quad F_0(S,T)=1, \quad F_1(S,T) = S.
\end{eqnarray}
 Then with $S=u^2+uv+v^2$ and $T=uv$, we have
$$    F_m(u^2+uv+v^2,uv)= \sum_{i=0}^{2m} u^iv^{2m-i}.$$

\begin{theorem}\label{thm:traceformula} For $\G = ~(2,4,6),~(2,\infty,\infty),(2,3,\infty),(2,4,\infty),(2,6,\infty)$, the table below describes the hypergeometric datum $HD(\G)=\{\alpha(\G),\beta(\G)\}$ and the choice of a generator $\l=\l(\G)$ of the field of $\QQ$-rational functions on $X_\G$ 
by its values at each elliptic point of given order and each cusp: 
$$
\begin{tabular}{|c|c|c|c|c|c|c|c|c|c|c|c|}
\hline
$\G$&$(2,\infty,\infty)$&$(2,3,\infty)$&$(2,4,\infty)$&$(2,6,\infty)$&$(2,4,6)$\\
\hline
$\l$&$(1,0,\infty)$&$(1,\infty, 0)$& $(1,\infty,0)$& $(1,\infty,0)$& $(1,\infty,0)$\\ \hline
$\alpha(\G)$    & $\{\f12,\f12,\f12\}$ &$\{\f12,\frac16,\frac56\}$&$\{\f12,\frac14,\frac34\}$&$\{\f12,\frac13,\frac23\}$&$\{\f12,\frac14,\frac34\}$\\
\hline
$\beta(\G)$&$\{1,1,1\}$&$\{1,1,1\}$&$\{1,1,1\}$&$\{1,1,1\}$&$\{1,\frac56,\frac76\}$\\ \hline
\end{tabular}
$$

Given an even integer $k\ge 2$ and a fixed prime $\ell$, the terms on the right-hand side of (\ref{E:mot1}) for almost all primes $p\ne \ell$ where $X_\G$ has good reduction can be explicitly expressed as follows. 
 
For $\l$ not corresponding to an elliptic point or a cusp, 
\begin{equation}\label{eq:2}
    \mr{Tr}(\mr{Frob}_\l \mid (V^k(\Gamma)_{\ell})_{\bar{\lambda}} )=F_{k/2}(a_\G(\l,p), p \bk), 
\end{equation}where 
\begin{equation}\label{eq:a_G}
a_\G(\l,p)
=\begin{cases}\left(\frac{1-1/\l}{p}\right)H_p(HD(\Gamma), 1/\lambda) & \text{ if $\G \ne (2,4,6)$};\\
 \left(\frac{-3(1-1/\l)}{p}\right)pH_p(HD(\G),  1/\l)& \text { if $\G=(2,4,6)$},
\end{cases} 
\end{equation} {with $\left ( \frac{\cdot}p\right)$ denoting the Legendre symbol.}
The contribution of $\l$ corresponding to a cusp
is  $1$. Let $\l=\l(z)\in X_\G(\FF_p)$ correspond to an elliptic point $z$ of $X_\G$ of order $N_z$, which is a CM-point with complex multiplication (CM) by
an imaginary quadratic field   $K_z=\QQ(\sqrt{d_z})$  
listed below:  
\begin{table}[h]
  $$
\begin{tabular}{|c|c|c|c|c|c|c|c|c|c|c|c|}
\hline
$\G$& $(2,\infty,\infty)$&$(2,3,\infty)$&$(2,4,\infty)$&$(2,6,\infty)$&$(2,4,6)$\\
\hline
$\sqrt{d_{z}}$&$ \sqrt{-4}$,\,-,\,-&$\sqrt{-4}, \sqrt{-3}$,\,-& $\sqrt{-8}, \sqrt{-4},$\,-& $ \sqrt{-3}$, $\sqrt{-3}$,\,- & $\sqrt{-24}$, $\sqrt{-4}$, $\sqrt{-3}$ \\\hline
\end{tabular}$$
\caption{CM fields of the elliptic points}
\label{tab:CM-fields}
\end{table}

When $p$ is unramified in $K_z$, the contribution from $\l(z)$ in (\ref{E:mot1}) is
\begin{equation}\label{eq:(5)}
 \begin{cases}(-p)^{k/2} & \text{ if $p$ is inert in $K_z$};\\
\sum_{-\frac{k}{2N_z} \le i\le \frac{k}{2N_z}} p^{k/2} (\alpha_{z,p}^2/p)^{iN_z} &\text{if $p$ splits in $K_z$}.
\end{cases}   
\end{equation} In the latter case, upon picking any prime ideal $\wp$ of the ring of integers of $K_z$ above $p$,  $\alpha_{z,p}$ can be chosen as any generator of the principal ideal $\wp$  when $N_z>2$. 
When $N_z=2$, $\alpha_{z,p}^2$ can be taken as any root of  $T^2-\left(\frac{-3}p\right)^{u}p^uH_p(HD(\G);1)T+p^2=0$,  where $u=1$ if $\G=(2,4,6)$, and $u=0$ else. 
\end{theorem} 

In fact, we may choose $\alpha_{z,p}$ to be the Jacobi sum $J_\omega\left (\frac13,\frac13\right)$ (resp. $J_\omega\left (\frac14,\frac14\right)$) if $N_z \in \{3, 6\}$ (resp. $N_z = 4$). Here $\omega$ is any generator of the group of characters of $\FF_p^\times$, and a Jacobi sum $J_\omega\left (a,b\right)$, defined as $\sum_{x\in\F_p} \omega^{(p-1)a}(x)\omega^{(p-1)b}(1-x)$, is  a fundamental building block for hypergeometric character sums. See \S \ref{ss:HG-FF}-\S\ref{ss:HG-BCM} for more detail.

Here is an example illustrating our main result.  
 For $\G=(2,4,6)$ the lowest $k$ with nontrivial $S_{k+2}(2,4,6)$ is  $k=6$, in which case $S_8(2,4,6)=\langle h_4^2 \rangle$ is $1$-dimensional, where $h_4$ generates $S_4(2,2,3,3)$ (cf. \cite[\S3.1]{Baba-Granath-g2}). By Jacquet-Langlands correspondence \cite{JL, Hida,  Greenberg-Voight, Yang-Schwarzian}, $h_4^2$ corresponds to the normalized weight-$8$ level $6$ cuspidal newform $f_{6.8.a.a}$ in the  L-functions and Modular Forms Database (LMFDB) label notation,  which we will adopt throughout this paper.  For primes $p >5$, denote by $a_p(h_4^2)$ the eigenvalue of $T_p$ on $h_4^2$, which is equal to the $p$th Fourier coefficient $a_p(f_{6.8.a.a})$ of $f_{6.8.a.a}$. The theorem above gives

\begin{align*}
- a_p(h_4^2) = -a_p(f_{6.8.a.a})&=\sum_{\l \in \F_p,\l \neq 0,1} \left(a_{\G}(\l,p)^3-2pa_{\G}(\l,p)^2-p^2a_\G(\l,p)+p^3\right) \\    &+p((pH_p(HD(\G);1))^2-p^2) +\left(\left(\frac{-1}p\right)+\left(\frac{-3}p\right)+\left(\frac{-6}p\right)\right)p^3.
\end{align*}

More generally, under the Jacquet-Langlands correspondence, the space $S_{k+2}(2,4,6)$ is sent to $S_{k+2}^{new} (\G_0(6),-,-)$, the subspace of newforms on $\G_0(6)$ which are eigenfunctions with eigenvalues $-1$ of the Atkin-Lehner operators $\omega_2$ and $\omega_3$, hence we obtain a formula for the traces of Hecke operators on the latter space.

For an index-2 subgroup $\G \in \{(2,6,6), (2,2,2,3), (3,4,4)\}$ of $(2,4,6)$, using the pullback of the hypergeometric sheaf on $X_{(2,4,6)}$ along the explicit $\Q$-rational covering map $X_{\G} \to X_{(2,4,6)}$, we obtain the Hecke trace formulae for $S_{k+2}(\G)$, which in turn gives rise to those for $S_{k+2}(2,2,3,3)$ by inclusion-exclusion. Similar to $S_{k+2}(2,4,6)$, the Jacquet-Langlands correspondence maps $S_{k+2}(\G)$ to the subspace of $S_{k+2}^{new} (\G_0(6))$ consisting of common eigenfunctions of $\omega_2$ and $\omega_3$ with certain prescribed eigenvalues. See \S\ref{sec: other groups} for details. In particular, as a corollary of 
Theorem \ref{thm:266}, 
the trace of $T_p$ on $S_4(2,2,2,3)$ 
for any prime $p>5$ yields the following character sum identity 
\begin{equation}\label{eq:2223-wt4}
 pH_p\left (\{\frac12,\frac12,\frac14,\frac34\},\{1,1,\frac16,\frac56\};1\right)=\left(\frac{-1}p\right)a_p(f_{6.4.a.a})+\left(\frac{-2}p\right) p.   
\end{equation} 

\medskip

More can be said when the triangle group $\G$ 
is a matrix group 
not containing $-I_2$, such as $\G = (3,\infty, \infty)  \cong \G_1(3)$ and $\G= (\infty, \infty, \infty)  \cong \G_1(4)$,  see \S\ref{ss:groups} for more details.   In these cases we obtain explicit Hecke trace formulas for almost all $p$ and all $k \ge 2$. To describe our result, 
recall the identity  
(see for example \cite{FOP})  
\begin{equation}\label{eq:FOP}
    \sum_{i=0}^k u^i v^{k-i}=\sum_{j=0}^{\lfloor \frac k2 \rfloor}(-1)^j\CC{k-j}{j}(uv)^j\cdot (u+v)^{k-2j}.
\end{equation}

\begin{theorem}\label{thm:traceformula-2F1} 
For $\G = ~(e_\infty,\infty,\infty)$ with $e_\infty\in\{3,\infty\}$, {let $HD(3, \infty,\infty) = \{\{\frac13, \frac23\}, \{1,1\}\}$ and $HD(\infty, \infty, \infty) = \{\{\frac 12, \frac 12\}, \{1,1\}\}$}. 
Denote by $N_\G$ the level of $\G$ as a subgroup of $SL_2(\ZZ)$, which is $3$  (resp. $4$) for $e_\infty = 3$ (resp. $\infty$).  
Choose the generator $\l=\l(\G)$ of $\QQ(X_\G)$ so that $\l$ takes values $\infty,  1, 0$ at the vertices of $X_\G$ of order $e_\infty$, $\infty$, $\infty$, respectively.  Then for all  integers $k\ge 1$,  the contributions of the right hand side of \eqref{E:mot1} at $\l \in\FF_p\setminus\{0,1\}$ are given by \begin{equation}\label{eq:ap-Hecke} 
\sum_{j=0}^{k/2}(-1)^j\CC{k-j}{j}p^j\cdot H_p\left (HD(\G); 1/\lambda\right)^{k-2j}, 
\end{equation}
The contributions at $\l = 0,1,\infty$ are distinguished in two cases:
\begin{enumerate}
  \item For $\G_1(4)$,  they are $1, \left(\frac{-1}p\right)^k,  \frac{1+(-1)^k}2$, respectively. 
  \item For $\G_1(3) $, those from  $\l=0$ and $1$   are $1$ and $\left(\frac{-3}p\right)^k$, and from $\l=\infty$ it is  $$\begin{cases}
      0& \text{ if } p\equiv -1 \mod 3, \quad k\equiv 1 \mod2;\\
       (-p)^{k/2}& \text{ if } p\equiv -1 \mod 3, \quad k\equiv 0 \mod2;\\
       (-1)^k \displaystyle \sum_{\overset{0\le i\le k}{ k\equiv 2i \mod 3}}p^i \cdot J_\omega\left (\frac13,\frac13\right)^{k-2i} & \text{ if } p\equiv 1 \mod 3, 
  \end{cases}  
  $$
  where $\widehat{\FF_p^\times}= \langle \omega\rangle$ and  $J_\omega(a,b)=\sum_{x\in \F_p} \omega^{(p-1)a}(x)\omega^{(p-1)b}(1-x) $.
 \end{enumerate}
 In particular, when $p\equiv -1\mod N_\G$ and $k$ is odd, $\mr{Tr} (T_p\mid S_{k+2} (\Gamma))=0$.
\end{theorem}

Explicit Hecke trace formulae have many applications.  For instance,  knowing Hecke trace can lead to special values of hypergeometric functions. As an example, it follows from the reasoning in Yang \cite{Yang-Schwarzian} that the Hecke trace $a_7(h_4^2)$ above yields the identity
$$
  \pFq32{\frac 12&\frac 14&\frac34}{&\frac 56&\frac 76}{\frac{2^{10}\cdot 3^3\cdot 5^6\cdot 7}{11^4\cdot 23^4}}=\frac{11\cdot 23}{140\sqrt 3}\frac{2^{1/3}(4+2\sqrt 2)}{7^{7/6}}\frac{\G(7/6)\G(13/24)\G(19/24)}{\G(5/6)\G(17/24)\G(23/24)}. 
$$ 
 The recent paper \cite{HMM1} provides an explicit hypergeometric-modularity method to express  certain types of hypergeometric character sums in terms of Hecke eigenvalues of modular forms based on the ideas in this paper. Its sequel \cite{HMM2} explores some  applications of the method. The trace formulae in this paper are used by Grove in \cite{Grove-ST} to obtain the vertical Sato-Tate distribution, as $p \to \infty$, of the normalized $H_p(HD(\Gamma); \l), \l \in \FF_p^\times$, for $\Gamma$ in Theorem \ref{thm:traceformula-2F1}. The distribution of that for $\Gamma$ in Theorem \ref{thm:traceformula} can be obtained by a similar idea. The paper \cite{Grove-Saad} studies the distribution of certain normalized hypergeometric character sums attached to data of length 3 and 4. 

We end this section with an outline of the proofs of Theorems  \ref{thm:traceformula} and \ref{thm:traceformula-2F1}, and the roles of the remaining sections.  The trace formulae, and automorphic sheaves $V^k(\G)$ over both the complex and $\ell$-adic fields are discussed in \S \ref{S:autosh}. In particular, Proposition \ref{P:ALLL2} and Remark \ref{rem:3} reduce the computation at generic points to $k=2$ for Theorem \ref{thm:traceformula} and $k=1$ for Theorem \ref{thm:traceformula-2F1}, respectively. 
For certain arithmetic triangle groups $\Gamma$, it follows from  the Rigidity Theorem \ref{thm: Katz-Rigidity}  and the Comparison Theorem \ref{thm:compare} in \S\ref{ss:groups} that the automorphic sheaves $V^2(\G)_\ell$ or $V^1(\G)_\ell$, up to twists by rank-1 sheaves,  
could be replaced by hypergeometric sheaves for which the Galois action on a stalk has Frobenius traces explicitly expressed by hypergeometric character sums. In \S \ref{ss:3.3}   
we identify such $\G$ whose associated hypergeometric  data are defined over $\Q$; they are the groups in Theorems \ref{thm:traceformula} and \ref{thm:traceformula-2F1}. 
 We recall in \S\ref{ss:char} the hypergeometric sheaves and derive their basic properties, following the works of Katz \cite{Katz} and Beukers-Cohen-Mellit \cite{BCM}. Using the information obtained, we  determine, for each $\Gamma$, the required twist of the corresponding $\ell$-adic hypergeometic sheaf so that the resulting sheaf is isomorphic to the automorphic sheaf $V^i(\Gamma)_\ell$ for $i=2$ or 1, as appropriate. This gives the contributions from generic points. The contributions from singular points are the theme of  \S\ref{S:singular}.    
All elliptic points are CM points. We first determine the CM structure at each elliptic point, with results summarized in Table \ref{tab:8} in \S\ref{SS:singular1}. Then in \S\ref{SS:singular2} we compute the contribution at $p$ from each elliptic point, which mainly depends on the order of the point and the behavior of $p$ in the CM field. The contribution from each cusp is computed in \S\ref{ss:5.3}. For $k$ even it is 1, as done in Scholl \cite{Sch88}, while for $k$ odd, it is $\pm 1$ or 0, depending on the type of  reduction at $p$ of the degenerate elliptic curve at the singular point. This completes the proofs of Theorems \ref{thm:traceformula} and \ref{thm:traceformula-2F1}. In \S\ref{S:6} we provide an independent verification of the Hecke trace formula for $S_{k+2}(\G)$ with smallest $k \ge 1$ by proving the hypergeometric  character sum identities directly. The two novelties mentioned in \S\ref{S:mot} are addressed in \S\ref{ss:other}. The Appendix at the end clarifies the $G$-Sheaves and the definition of the pushforward sheaves appearing in this paper.

\section*{Acknowledgements} The authors would like to thank Jean-Pierre Serre for his helpful comments on  the paper. Li is partially supported by Simons Foundation grant \#355798,   Long is supported in part by the Simons Foundation grant  \#MP-TSM-00002492 and the LSU Michael F. and Roberta Nesbit McDonald Professorship, Tu is supported by the NSF grant DMS \#2302531. While working on this project, Li, Long, and Tu were hosted by the Institute of Mathematics, Academia Sinica, Taiwan in the summer of 2023, Tu was also hosted at Oregon State University during the fall of 2023, and Li visited the University of Hong Kong in summer 2024.  They would like to thank these Institutions for their hospitality and support.

\section{Automorphic sheaves}
\label{S:autosh} 
\subsection{ 
Complex automorphic sheaves}
\label{SS:autosh1} 
\subsubsection{} 
For a Fuchsian subgroup $\Gamma$ of $\SL2R$ of first kind,
let $ X_{\Gamma}:=X_{\Gamma} ^{an} = \Gamma \backslash \mathfrak{H}^*$, 
where $\mathfrak{H}^*=\mathfrak{H}$ or the union of $\mathfrak{H}$ and the cusps of $\Gamma$ according to whether $\Gamma$ is cocompact or not.  
 Let $p: \mathfrak{H}^* \to X_{\Gamma}$ be the projection. 
Note that $\mathfrak{H}^*$ carries a canonical topology, and $p$ is a continuous map.
 In this paper, two kinds of $\G$ are considered:
 \begin{itemize}
 \item[1.] $\Gamma$ is commensurable with  $\mr{SL}_2( \ZZ)$. Then the cusps of $\Gamma$ is the set 
of rational numbers and the point $\infty$. 
\item[2.] $\Gamma$ is commensurable to the set  $O_B ^1$ of norm 1 elements in a maximal order $O_B$ of an indefinite quaternion algebra $B$
over $\QQ$. That is, we choose once and for all an embedding $B \subset M_2 (\RR)$, which induces an embedding 
$\theta : O_B ^1 \subset \SL2R$. In this case, the set of cusps is empty, so  $\mathfrak{H}^* = \mathfrak{H}$. 
 \end{itemize}
 We also have the set of elliptic points of $\Gamma$. We have a finite set
 $S = S_{c} \cup S_{e} \subset X_{\Gamma} $ of points which are cusps or elliptic points, 
 that is,  their preimages in $\mathfrak{H}^*$ are cusps or elliptic points of $\Gamma$, respectively. Let
 \[
 X_{\Gamma} ^{\circ} = X_{\Gamma} - (S_{c} \cup S_{e} ), \quad
  \quad
  Y_{\Gamma} = X_{\Gamma} - S_{c}  =  \Gamma \backslash \mathfrak{H}, 
 \]
 the complement of the set of cusps and elliptic points, and the complement of the 
 set of cusps, respectively. Note that 
 $p: \mathfrak{H} -\{  \text{elliptic points of \ }  \Gamma  \}\to  X_{\Gamma} ^{\circ} $
 is a covering space in the sense of topology. In particular, if $\Gamma$ has no elliptic points,
 then $p:  \mathfrak{H}  \to   Y_{\Gamma} $ is the universal covering of $ Y_{\Gamma}$.
 The action of $\Gamma$ on  $\mathfrak{H}^*$  is via the quotient 
 $\bar{\Gamma}$ in  $\mr{PSL}_2( \RR)$. For the groups in this paper,
 $\bar{\Gamma} = \Gamma /\Gamma \cap \{ \pm I\}$.  Therefore, if 
 $\Gamma$ is torsion-free, then 
  $p:  \mathfrak{H}  \to   Y_{\Gamma} $ is the universal covering of $ Y_{\Gamma}$, and
  the fundamental group at any base-point $x$
  \[
  \pi _1 (Y_{\Gamma}, x) \cong \Gamma, 
  \]
 an isomorphism unique up to inner automorphism. In the quaternion cases 
 this is $\pi _1 (X_{\Gamma}, x) \cong \Gamma$ for $\Gamma$ without torsion.
 In general, there is an epimorphism 
 \[
   \pi _1 (X^{\circ}_{\Gamma}, x)\to \bar{\Gamma } \cong \mr{Aut} (( \mathfrak{H} -\{  \text{elliptic points of \ }  \Gamma  \})/ X _{\Gamma}^{\circ}   ).
 \]
 We consider local systems of $\CCC$-vector spaces of finite dimension on $X^{\circ}_{\Gamma}$ and their extensions to constructible 
 sheaves on $X_{\Gamma}$. By the usual dictionary, a local system on  $X^{\circ}_{\Gamma}$ is equivalent to a finite-dimensional representation 
 of $\pi _1 (X^{\circ}_{\Gamma}, x)$. In this paper, all the local systems will be defined by representations of $\bar{\Gamma}$.
\subsubsection{} 
There are constructible sheaves of $\CCC$-vector spaces $V^k (\Gamma)_{\CCC}$
on $X_{\Gamma}  $ for integers $k\ge 1$ with the following properties: 
\begin{itemize}
\item[1.] Over the open subset $X_{\Gamma} ^{\circ}$ 
these automorphic sheaves are local systems of rank 
$k+1$. They are thus solution sheaves to a system of regular singular differential equations
on  $X_{\Gamma} ^{\circ}$. 
For $\Gamma$ considered in this paper, these will be hypergeometric  local systems when $k=1$ or 2.

\item[2.] The sheaves $V^k (\Gamma)_{\CCC}$ on $X_{\Gamma} $ are extensions of the sheaves on $X_{\Gamma} ^{\circ}$ defined above: 
 \[
V^k (\Gamma) _{\CCC}  = \iota_*   (V^k (\Gamma) _{\CCC} \mid X_{\Gamma} ^{\circ}) \quad \text{for~the~inclusion~map}~
\iota: X_{\Gamma} ^{\circ}  \to X_{\Gamma}.
\] 
\item[3.] Fundamental theorems of Eichler, Shimura and Kuga give:
\begin{equation*}
\label{E:mot4}
H^1 _{par} (\Gamma , \mr{Sym}^k (\CCC ^2))  \cong 
H^1 (X_{\Gamma} , V^k (\Gamma)_{\CCC} )  \cong
S_{k+2} (\Gamma) \oplus \overline{ S_{k+2} (\Gamma)}, 
\end{equation*}
which is a Hodge decomposition of type $(k+1, 0), (0, k+1)$.  
The left-hand side is parabolic cohomology. 
Moreover
\[
H^1 (X_{\Gamma} , V^k (\Gamma)_{\CCC} ) \cong \mr{Im}(
H^1 _c (X_{\Gamma}^{\circ} , V^k (\Gamma)_{\CCC} )\to 
H^1  (X_{\Gamma}^{\circ} , V^k (\Gamma)_{\CCC} 
))
\]
the image of the compactly supported cohomology.
The Hecke operators act on the spaces as geometric correspondences and these
isomorphisms are equivariant for the Hecke actions. 
\end{itemize}
When $\G\subset \SL2R$ is an arbitrary Fuchsian subgroup of the first kind, these sheaves 
are discussed in detail in \cite{BayerNeu81} by Bayer and Neukirch. The group $\Gamma$ acts on 
$\mathfrak{H}^*\times \mr{Sym}^k (\CCC^2)$, on the first factor by fractional linear transformations, 
on the second factor via $\varrho _k := \mr{Sym}^k (\varrho_1)$, 
where $\varrho_1$
arises from the inclusion $\G\subset \SL2R$ and the canonical action of 
$\SL2R$ on $\CCC^2$. Then 
\begin{equation}\label{eq:V-CC}
V^k (\Gamma)_{\CCC}:= p_* ^{\G} (\underline{\mr{Sym}^k (\CCC ^2)}) = \G\backslash 
\mathfrak{H}^*\times \mr{Sym}^k (\CCC ^2), 
\end{equation}
where the middle expression uses the $\G$-invariant push-forward functor.  See the Appendix  on $G$-sheaves  for the theory of the functor $p_\ast^G$. Note that 
if $-I\in \G$ these sheaves are identically zero if $k$ is odd. When $k$ is even, the 
action factors through the projective group $\bar{\Gamma} = \Gamma/\{\pm I\}$, and this action is discontinuous.
When $\G$ is torsion-free, Bayer and Neukirch establish the Eichler-Shimura isomorphisms
with cusp forms for all $k$, even or odd. When $\G$ has torsion and $k$ is even, they show that by going to a torsion-free
normal subgroup $\G'$ of finite index 
and using the projection $q: X_{\G'}\to X_{\G}$, 
\[
  V^k (\Gamma )_{\CCC}  =    q_{*} ^{\G/\G'} \left ( V^k (\Gamma ')_{\CCC}\right )
\] 
gives the Eichler-Shimura isomorphism. When $x\in X_{\G}$ is not a cusp or elliptic point, 
the stalk $(V^k (\Gamma )_{\CCC}) _x$ is isomorphic to $\mr{Sym}^k (\CCC ^2)$. For $x$ a cusp 
or elliptic point, 
\[
(V^k (\Gamma )_{\CCC}) _x = \left (
\mr{Sym}^k (\CCC ^2)\right ) ^{\Gamma _{\tau}} 
\]
where $\tau \in \mathfrak{H}^*$ is any point with $p(\tau) = x$ and $\G_\tau$ is the stabilizer of $\tau$ in $\G$. 
The same holds for $k$ odd and $\G$ not containing $-I$. 

\subsection{$\ell$-adic automorphic sheaves} \label{SS:autosh2} 
\subsubsection{}\label{ss:2.2.1}
For $\G$ as before, in this section we view the quotient $X_{\Gamma}$ as the collection of complex points of a curve over a number field.
 Recall that $Y_\Gamma = \Gamma\backslash \mathfrak H$, which is $X_\G$ for $\G$ cocompact, and $X_\Gamma$ with cusps removed for $\G$ non-cocompact. 
Let $G$ be the algebraic group $\textsf{GL}_2$ or $\textsf B^\times$ according as $\G$ non-cocompact or cocompact arising from an indefinite quaternion algebra $B$ defined over $\QQ$. 
 In both cases, $G(\RR) =$ GL$_2(\RR)$  which has two connected components $G(\RR)_{\pm}$ consisting of elements 
with positive and negative determinants, respectively.  Moreover, $G(\RR)_+$ acts transitively on the upper half plane $\mathfrak H$ via fractional linear transformations so that we may identify $\mathfrak H$  with $G(\RR)_+/\rm{SO}_2 Z(\RR)$, where $Z$ denotes the center of $G$.  Thus  $G(\RR)/\rm{SO}_2Z(\RR)$ may be identified with $\mathfrak H^{\pm}$, the disjoint union of upper and lower half plane.  Let $G(\QQ)_+ =G(\QQ)\cap G(\RR)_+$.  Write $\mathbb A$ for the ring of adeles over $\QQ$, and $\mathbb A_f$ the subring of finite adeles. The group $G(\QQ)$ is diagonally embedded in $G(\mathbb A)$.  Denote by $G(\QQ)_{+,f}$ the diagonal imbedding of $G(\QQ)_+$ in $G(\mathbb A_f)$. 

 Let $U$ be a compact-open subgroup of $G(\mathbb A_f)$.  
Shimura's theory of canonical models gives a curve $Y_U$ defined over $\QQ$ whose complex points can be expressed as 
\begin{eqnarray*}
 Y_U := G(\QQ)\backslash \mathfrak H^{\pm} \times G(\mathbb A_f)/U 
              = G(\QQ)_+\backslash \mathfrak H \times G(\mathbb A_f)/U. 
\end{eqnarray*}
 It has finitely many connected components,  parametrized by the double cosets $$G(\QQ)_{+,f}\backslash G(\mathbb A_f)/U = \coprod_{ i =1}^h G(\QQ)_{+,f}\,x_iU$$ with  $h=|\widehat{\ZZ}^\times/\det U|$.  Note that an element $\gamma \in G(\QQ)_{+,f}$ satisfies $\gamma x_i U = x_i U$ if and only if $\gamma$ lies in $\G_{i,U} := G(\QQ)_{+,f} \cap x_iUx_i^{-1}$,  a congruence subgroup of SL$_2(\RR)$. Thus we get
 \begin{eqnarray}\label{eq:YUdisjointunion}
  Y_U = \coprod_{i=1}^h \G_{i,U}\backslash \mathfrak H = \coprod_{i=1}^h Y_{\G_{i,U}},
\end{eqnarray}
showing that each connected component of $Y_U$ is the curve of a congruence subgroup of  SL$_2(\RR)$.  Conversely, all congruence subgroups $\G$ of SL$_2(\RR)$ are known to arise as $\G = G(\QQ)_{+,f} \cap U$ for some compact open subgroup $U$ of $G(\mathbb A_f)$. We call $U$ the {\it (adelic) group corresponding to $\G$.} 
\medskip

\begin{defn}\label{defn:neat}
 A compact open subgroup $U$ of $G(\mathbb A_f)$ is called {\it neat} if all $\G_{i,U}$ defined above are torsion-free. 
\end{defn}
\medskip

Note that, when $h>1$, the curve $Y_U$ while defined over $\QQ$ is not absolutely irreducible over $\QQ$.  
Its connected components 
$Y_{\G _{i, U}}$ are absolutely irreducible over an abelian extension of $\QQ$ given by 
Shimura's theory of canonical models.

Adelically, we may regard 
$$ Y_U = G(\QQ)\backslash G(\RR)G(\mathbb A_f)/{\rm{SO_2Z}}(\RR)U
             = G(\QQ)\backslash G(\mathbb A)/{\rm{SO_2Z}}(\RR)U.$$
Denote by $S_k(U)$ the space of automorphic forms on $G(\mathbb A)$, cuspidal if $G=\text{GL}_2$,  right invariant under $U{\rm Z}(\RR)$, whose components at the place $\infty$ lie in the space of discrete series representation of $G(\RR)$ of weight $k$.
Note that it decomposes into the direct sum $\bigoplus  _{i=1}^hS_k(\Gamma _{i,U})$.

In the elliptic modular cases there are canonical compactifications $X_U$ of $Y_U$ by adding cusps of $\G_{i,U}$. For the 
quaternion cases, we set $X_U = Y_U$. 

Finally, there are constructible $\QQ_{\ell}$ sheaves $V^{k}(U)_{\ell} $ on $X_U$ to be elaborated in the next subsection.  
 For both cases Eichler-Shimura theory takes the shape:
\begin{theorem}[Deligne \cite{Deligne-mf}, Ohta \cite{Ohta82}] \label{T:autoshv1} 
For integers $k \ge 0$, the action of the Galois group $G_\QQ$ of $\QQ$ on $H^1 (X_U \otimes_\QQ \bar \QQ,  V^{k}(U)_{\ell}  )$ is unramified at the primes $p \ne \ell$ where $X_U$  has good reduction. Moreover, we have 
\[
\mr{Tr} (\mr{Frob}_p \mid H^1 (X_U \otimes_\QQ \bar \QQ,  V^{k}(U)_{\ell}  )) = 
\mr{Tr} (T_p \mid    S_{k+2} (U)). 
\]
 \end{theorem}

Combined with the Grothendieck-Lefschetz trace formula, we can express the trace of $T_p$ on $S_{k+2}(U)$ at primes $p$ where $X_U$ has good reduction as a sum over the $\FF_p$-points on the reduced curve: 
\begin{eqnarray}\label{e:automorphictrace}
-\mr{Tr} (T_p\mid S_{k+2} (U)) = 
\sum_{\l \in X_{U}(\FF_p)} \mr{Tr}(\mr{Frob}_\l \mid (V^{k} (U)_{\ell})_{\bar{\lambda}} ).
\end{eqnarray} 
Write $V^k(\G)_\ell$ for the restriction of the sheaf $V^k(U)_{\ell}$ on $X_\G$. In case that $X_\G$ is defined over $\QQ$, we have $X_U = X_\G$ and $S_{k+2}(U) = S_{k+2}(\G)$. Then the above theorem gives a geometric description of the trace of the Hecke operator $T_p$ on $S_{k+2}(\G)$. 

Now we explain the notation $(V^{k} (U)_{\ell})_{\bar{\lambda}}$ on the right side of (\ref{e:automorphictrace}). Let $X_U ^{\circ}$ be the complement in  $X_U $ of the cusps and elliptic points. Then our 
sheaves $V^{k} (U)_{\ell}$ are lisse on the  open set $X_U ^{\circ}$, equivalent therefore 
to a representation of $\pi _1 (X_U ^{\circ}, \bar{\eta})$ where $\bar{\eta}$ is 
the generic geometric point 
$\bar{\eta}= \mr{Spec} (\overline{\QQ (X_U)}) $  of $X_U ^{\circ}$. 
Thus there is a canonical epimorphism
\[
G_{\eta}:= 
\mr{Gal}(\overline{\QQ (X_U)}/\QQ (X_U))
\to \pi _1 (X_U ^{\circ}, \bar{\eta})
\]
so that we may regard all lisse sheaves on $X_U ^{\circ}$
as modules for the Galois group $G_{\eta}$. 
Every algebraic point $t \in X_U$ gives a discrete valuation of the function field $\QQ (X_U)$ with residue field $\kappa(t)$, so we may 
speak of the decomposition $D_{t}$ and inertia $I_{t}$ subgroups (well-defined up to conjugation) and their
images in $\pi _1 (X_U ^{\circ}, \bar{\eta})$. 
As $t :\mr{Spec} ( \kappa (t )) \to X_U$ gives a morphism of \'etale sites, we get 
$  t ^{*}  V^{k} (U)_{\ell} $, a sheaf for the \'etale site on $\mr{Spec} ( \kappa (t ))$.
This is equivalent to a module over $G_{\kappa(t)}:= \mr{Gal} (\overline{\kappa (t )}/\kappa (t )) \cong D_t/I_t$ for an algebraic closure  $\overline{\kappa (t )}$ of $\kappa (t )$. In other words, the stalk 
($V^{k} (U)_{\ell})_{\bar{t}} $ is the $\QQ_{\ell}$-vector space $\bar{t }^*V^{k} (U)_{\ell}$ with $G_{\kappa(t)}$ action. Here $\bar{t }: 
\mr{Spec} ( \overline{\kappa (t )})  \to \mr{Spec} ( \kappa (t ))\to X_U$. 

In particular, if $t$ is in the open set $X_U ^{\circ}$, then the inertia $I_{t} $ acts trivially. Also, 
by choosing an embedding $\kappa(\bar{t})\subset \kappa ({\bar{\eta}})$, we get a homomorphism
\[
G_{\kappa(t)} = \pi _1 (\mr{Spec}(\kappa (t)), \bar {t})\to 
\pi _1 (X_U ^{\circ}, \bar{\eta})
\]
and any representation of the right-hand side gives a representation of the left hand side. This 
describes
the stalk $(V^{k} (U)_{\ell})_{\bar{t}} $ for 
$t \in X_U ^{\circ}$. 

Moreover
if $\iota : X_U ^{\circ} \hookrightarrow X_U $ is the inclusion, then our sheaves also have the property that 
\[
V^{k} (U)_{\ell} =  \iota _{*} (V^{k} (U)_{\ell} \mid  X_U ^{\circ}).
\]
This implies that for every
$t \in X_U$, the stalk $(V^{k} (U)_\ell)_{\bar{t }}$ is the $\QQ_\ell$-space $(V^{k} (U)_\ell) _{\bar{\eta}}  ^{I _{t}}$ with $G_{\kappa(t)}$ action.

The structure of these stalks will be examined in more detail in \S\ref{ss: 2.2.4}.

The curves $X_U$ are defined over $\QQ$. In fact, they have canonical 
reductions modulo all but finitely many primes $p$. This follows because of the interpretation of these curves 
as moduli spaces. That is, we may regard $f: X_U \to S$ where $S = \mr{Spec}(\ZZ [1/N])$ for an integer $N \ge 1$, 
where $f$ is smooth and projective. Moreover, the sheaves $V^k (U)_{\ell}$ extend to these integral models, and
for each $i\ge 0$, the sheaves $R^i f_* \mathcal{F}_{\ell} $ are lisse on $S$ and compatible with base-change. This implies 
that for each prime $p \nmid N$ we have an isomorphism
\[
H^1  (X_U \otimes \bar{\FF} _p,     V^k (U)_{\ell}   \otimes \bar{\FF} _p  )  \cong 
 H^1 (X_U \otimes \bar{\QQ},      V^k (U)_{\ell}        \otimes \bar{\QQ}  ), 
\]
respecting Galois actions. Namely, a Frobenius $\mr{Frob}_p \in G_{\QQ}$ acts on the left-hand side as the 
(inverse of the) $p$th power endomorphism. 
The Grothendieck-Lefschetz 
trace formula expresses the Frobenius trace as a sum of local traces in the fibers of the points 
$X_U (\FF _p)$. 

For $\lambda \in X_U(\FF_p)$ in (10), the stalk $(V^k(U)_\ell)_{\bar \lambda}$ refers to the stalk of $V^k(U)_\ell$ at any algebraic point $t$ of  $X_U$ which reduces to $\lambda$ modulo $p$. 

\subsubsection{}\label{ss:2.2.2}
As in the previous subsection, let $G$ be \textsf{GL}$_2$ or $\textsf{B}^\times$ 
for an indefinite quaternion algebra $B$ defined over $\QQ$. Let $U$ be an open compact subgroup of $G(\mathbb A_f)$. 
In this subsection, we discuss the $\ell$-adic sheaves $V^k(U)_\ell$ on $X_U$, following the construction by Ohta in \cite{Ohta82}.  Since it is crucial to this paper, we sketch the key idea below. 

Denote by $\varrho_k:\textsf{GL}_2 \to \textsf{ GL}_{k+1} $ the morphism of algebraic groups over $\ZZ$
 which is  $k$th symmetric power of $\varrho_1:\textsf{ GL}_2 \to \textsf{ GL}_{2} $, 
the identity map. So for any commutative ring $R$, $\varrho_k (R) : \rm {GL}_2 (R) \to \rm {GL}_{k+1} (R)$
sends the matrix  $A \in \rm {GL}_2 (R)$ to the $R$-linear transformation  
$ R^{k+1} \to R^{k+1}$ given by the matrix which is 
the $k$th symmetric power of $A$.

When $G=\textsf{B}^\times$, choose a real quadratic field $F$ which splits $B$, namely $B\otimes_\QQ F = \rm{M}_2(F)$, and fix an embedding $\theta: B \to \M_2(F)$. Since $B$ is an inner form of  ${\M}_2(F)$, there is an element $a \in \rm {GL}_2(F)$ so that the embedded image can be described as  
$$\theta(B) = \{ x \in \M_2(F) : a \,x \,a^{-1} = x^\tau \}, $$
where $\tau$ is the nontrivial automorphism of $F$ over $\QQ$. 
Hence every $x \in \theta(B^\times)$ satisfies $${\varrho_k}(a){\varrho_k}(x){\varrho_k}(a)^{-1} = {\varrho_k}(x)^\tau, \quad {\rm for~all}~ k\ge 1.$$ 
Let
$$Y_k=\{ x \in \M_{k+1}(F) : {\varrho_k}(a)\, x\, {\varrho_k}(a)^{-1} = x^\tau \}\subset \M_{k+1}(F).$$ 
It is a central simple algebra over $\QQ$ with an involution. By construction
\[
Y_k \otimes _{\QQ}F \cong \M_{k+1}(F).
\]
The invertible elements of $Y_k$  define
an algebraic group over $\QQ$, denoted $\textsf{Y}_k ^\times$. Note that 
$\textsf{Y}_k^\times(\QQ)= Y_k ^\times $ and $\textsf{Y}_k^\times(F)= (Y_k \otimes_{\QQ}F )^\times \cong {\rm GL}_{k+1} (F)$. 
Let $j: Y_k ^{\times}\hookrightarrow  \M_{k+1}(F)^{\times}  = \mr{GL} _{k+1}(F)$ be the natural inclusion.
There is a morphism of algebraic groups
$\varrho _k ': \textsf{B}^\times \to \textsf{Y}_k ^\times$ with the property that 

\[
j \circ \varrho_k'(\QQ): \textsf{B}^{\times} (\QQ)= B^{\times}
\to \textsf{Y} _k ^{\times}(\QQ) = Y_k ^{\times} \hookrightarrow {\rm GL}_{k+1}(F)
\]
is $\theta$ restricted to $B^\times$ followed by $\varrho_k(F)$, and $\varrho_k'(F)$ is $\varrho_k(F)$.  Thus if we identify $B^\times$ with $\theta(B^\times)\subset {\rm GL}_2(F)$, then we may regard $\varrho'_k(\QQ)$ as the restriction of $\varrho'_k(F)$ to $B^\times$. 

As such, there are two possibilities for $Y_k$:

(I) $Y_k = \M_{k+1}(\QQ)$,

(II) $Y_k = {\M}_{(k+1)/2}(D)$ for a (division) quaternion algebra $D$ over $\QQ$.
 
For $k\ge 2$ even, $Y_k$ falls in case (I) and $\varrho_k'$ is a $(k+1)$-dimensional representation of 
the algebraic group $\textsf{B}^\times$. 

In the second case, we can use the regular representation $r: D^\times \to \mr{GL}_4$ and 
define a representation $\textsf{B}^\times \to \textsf{GL}_{2(k+1)}$ by composing 
$\varrho _k '$  with the embedding 
\[
{\M}_{(k+1)/2}(D) \to \M_{(k+1)/2}(\M_4 (\QQ)) = \M_{2(k+1)}(\QQ). 
\]
Note that $Y_1 = \theta (B) = D$ falls in case (II). So $\varrho_1'$  is 4-dimensional. 
\medskip

Fix a positive integer $k$. Let $g$ be an element in $G(\QQ)$ such that $U \subset g \prod_p G(\ZZ_p) g^{-1}$. Then $\G = G(\QQ)_{+}\cap U$ is contained in $gG(\ZZ) g^{-1}$. In the representation space of $\varrho_k$ (resp. $\varrho_k'$) if $G= \textsf{GL}_2$ (resp. $G=\textsf B^\times$), choose a lattice $L_k$ stable under $\varrho_k (gRg^{-1}\cap G(\QQ))$ (resp. $\varrho_k' (gRg^{-1}\cap G(\QQ))$), where $R=\M_2(\ZZ)$ for $G= \textsf {GL}_2$ and $R$ is a maximal order $O_B$ for $G=\textsf B^\times$. 

Fix a prime $\ell$.  
Define a tower of curves
\[
... \rightarrow Y_{U_3}\rightarrow  Y_{U_2}  \rightarrow Y_{U_1}\rightarrow Y_{U_0} = Y_U
\]
where  $$U_n = U \cap g(\{x \in G(\ZZ_\ell): x \equiv id \mod \ell^n\} \prod _{p \neq \ell} G(\ZZ_p))g^{-1}.$$ 

Then each $Y_{U_n}$ is a finite Galois covering of $Y_{U}$ with Galois group 
\[
\mr{Aut} (Y_{U_n}/Y_{U}) = U/({\rm Z}(\QQ)\cap U)U_n.
\]
As a compact subgroup of the discrete group ${\rm Z}(\QQ)$, ${\rm Z}(\QQ)\cap U$ is contained in $\{\pm \mr{id}\}$ for $U$ in general, and is trivial for $U$ neat. It follows from our choice of $g$ that the finite group $U/({\rm Z}(\QQ)\cap U)U_n$ can be represented by elements in $gRg^{-1}\cap G(\QQ)$. 
The constant group-scheme $Y_{U_n}\times L_k/\ell ^n L_k$ has an action 
of $U/({\rm Z}(\QQ)\cap U)U_n$. On the first factor this is the Galois action resulting from the theory of canonical models; the action on the second factor is from $\varrho_k$ or $\varrho_k'$ as appropriate. 

Assume $U$ is neat as in Definition \ref{defn:neat}. The quotient 
\[
V_n (U, L_k):= 
(U/({\rm Z}(\QQ)\cap U)U_n )\backslash (Y_{U_n}\times L_k/\ell ^n L_k)
\] is 
a finite \'etale group-scheme over $Y_U$. 
These form a projective system with $n$. Taking projective limit yields a lisse  $\ZZ _{\ell}$-sheaf $V(U, L_k)$ on $Y_U$.  Tensoring this 
 with $\QQ _{\ell}$, we obtain a lisse $\QQ _{\ell}$-sheaf on $Y_U$ independent of the choice of $L_k$. The sheaf $V^k(U)_\ell$ on $X_U$ is the pushforward of this sheaf on $Y_U$ via the inclusion map $\iota: Y_U \hookrightarrow X_U$. Recall that $\iota$ is surjective for $G=\textsf B^\times$, and for $G= \textsf{GL}_2$, the complement $X_U \setminus Y_U$ is the set of cusps.  Restricting $V^k(U)_\ell$ to $X_\G$ gives the sheaf $V^k(\G)_\ell$. 

When the group $U$ is neat, the curve $Y_U$ admits a moduli interpretation, namely there is a family of abelian varieties $f: A_U \to Y_U$ defined over $\QQ$ such that each fiber is an elliptic curve for $G = \textsf{GL}_2$, and a polarized abelian surface for $G=\textsf B^\times$. 
The sheaves defined above can be expressed geometrically as well. As shown in  \cite[Theorem 3.3.3]{Ohta83},  
\[ (\varprojlim _n V_n (U, L_1)) \otimes_{\ZZ_\ell} \QQ_\ell \simeq R^1 f_* \QQ _{\ell}    
\] and hence we can make the identification  
$$ V^1(U)_\ell = \iota_* R^1 f_* \QQ _{\ell}, \quad \iota : Y_U \hookrightarrow X_U. $$
Therefore $V^1(U)_\ell$ is a rank-2 or 4 sheaf according as $G=\textsf {GL}_2$ or $G= \textsf B^\times$. In the elliptic modular case, since $\varrho_k =$Sym$^k(\varrho_1)$, we have 
\[
V^{k}(U) _{\ell} = \iota_* \mr{Sym}^k R^1 f_*\QQ _{\ell}, 
\]
as given by Deligne \cite{Deligne-mf}. Thus $V^k(U)_\ell$ has rank $k+1$. For the quaternion case, as discussed before, $V^k(U)_\ell$ has rank $k+1$ for $k$ even, but 
the relation between $V^k(U)_\ell$ and $R^1 f_* \QQ _{\ell}$ is less clear. Proposition (4.3.1) of \cite{Ohta83} shows that $V^k(U)_\ell$ occurs as a component of the sheaf $R^{w} f_{*} ^{n} \QQ_\ell$ for suitable choices of $w$ and $n$, whose Galois module structure
can be expressed as 

\[
R^{w} f_{*} ^{n} \QQ_\ell = \bigwedge ^w R^{1} f_{*} ^n \QQ_\ell  = 
\bigwedge ^w (R^{1} f_{*} \QQ_\ell )^n,
\]
in which the first equality  follows from the structure of the cohomology of abelian varieties and the second from the K\"unneth formula.
From this one sees that all the sheaves $R^{w} f_{*} ^{n} \QQ_\ell$ are calculated from tensor operations from the 
fundamental $V^1(U)_\ell$. We will see an expression for $V^{2k}(U)_\ell$ similar to the elliptic modular case in Proposition \ref{P:ALLL2}. 
\medskip

The sheaf $V^k(U)_\ell$ defined above is for $U$ neat. For a general $U$, there is a neat compact open normal subgroup $U'$ of $U$ (such as $U_n$ above for $n$ large) for which $V^k(U')_\ell$ is defined. It follows from the definition of $V^k(U')_\ell$ that its stalk at any algebraic point of $Y_{U'}$ is $(\varprojlim _n L_k/\ell^n L_k )\otimes_{\ZZ_\ell} \QQ_\ell $, on which $U/({\rm Z}(\QQ)\cap U)U'$ acts since by construction the lattice $L_k$ is invariant under the action of $\varrho_k$ or $\varrho_k'$ on $gRg^{-1}\cap G(\QQ)$ and the quotient $U/({\rm Z}(\QQ)\cap U)U'$ can be represented by elements in $gRg^{-1}\cap G(\QQ)$. We say that $V^k(U')_\ell$ admits $U/({\rm Z}(\QQ)\cap U)U'$-structure. This allows us to define $V^k(U)_\ell$ by pushforward from $V^k(U')_\ell$ as described below. Here $k$ is even if ${\rm Z}(\QQ)\cap U = \{\pm {\rm id}\}$ so that $\varrho_k(-{\rm id}) = (-{\rm id})^k={\rm id}$. 
\medskip

As before, let $X_U^\circ$ be the open subset of $X_U$ with cusps and elliptic points removed. Denote by $T_{U'}$ the preimage of $X_U^\circ$ under the covering map $p: X_{U'} \to X_U$. {Note that $T_{U'} \subset X_{U'}^\circ$.} Let $q: T_{U'} \to X_U^\circ$ be the restriction of $p$ to $T_{U'}$. Then $q$ is finite \'etale. When $U$ is neat, we have $X_U^\circ = Y_U$ and $T_{U'} = Y_{U'}=X_{U'}^\circ$. The sheaf $V^k(U)_\ell$ restricted to $Y_U$ defined before is the pushforward of $V^k(U')$ restricted to $Y_{U'}$ via the covering map $q: Y_{U'} \to Y_U$. Hence, as a generalization from neat to general $U$, we pushforward the restriction of $V^k(U')_\ell$ to $T_{U'}$ via $q$ to get a lisse sheaf $\mathcal F_\ell$ on $X_U^\circ$, and define 
the sheaf $V^k(U)_\ell$ on $X_U$ as the pushforward of $\mathcal F_\ell$ via the inclusion $\iota: X_U^\circ \to X_U$.  In other words, $V^k(U)_\ell = \iota _* q_*^{U/({\rm Z}(\QQ)\cap U)U'}  (V^k(U')_\ell \mid T_{U'})$. Note that the sheaf $\mathcal F_\ell$ on $X_U^\circ$ is independent of the choice of $U'$, hence so is $V^k(U)_\ell$.

\begin{lemma}\label{l:pushford} The sheaf  $V^k(U)_\ell$ on $X_U$ defined above agrees with the pushforward $p_*^{U/({\rm Z}(\QQ)\cap U)U'}V^k(U')_\ell$ via the covering map $p: X_{U'} \to X_U$. Here $k$ is even if ${\rm Z}(\QQ)\cap U = \{\pm {\rm id}\}$. 
\end{lemma}

\begin{proof} 
By construction, this is true at all points of $X_U ^{\circ}$. Let $t$ be a point
of $X_U \setminus X_U ^{\circ}$ and $t'$ a point of $X_{U'}$ lying over it. 
Because $U'$ is a normal subgroup of $U$ with finite index, $p$ is a Galois cover with Galois group 
$U/({\rm Z}(\QQ)\cap U)U'$, which acts transitively on each fiber of $p$.  
Denote by $\bar \eta$ and $\bar {\eta'}$ the generic geometric points of $X_U$ and $X_{U'}$ respectively.
Note that 
\[
U/({\rm Z}(\QQ)\cap U)U'\cong \mr{Aut}(X_{U'}/X_U). 
\]
Also, because we defined $V^k(U)_\ell$ via $\iota$,
and $V^k(U)_\ell$ is lisse on $X^{\circ}_U$, 
we have 
\[
(V^k(U)_\ell)_{\bar t} =  (V^k(U)_\ell)_{\bar \eta}^{I_t}.
\]
Similarly, we also have
\[
(V^k(U')_\ell)_{\bar t'} =  (V^k(U')_\ell)_{\bar \eta'}^{I_{t'}}.
\]
Here $I_t$ and $I_{t'}$ are the inertia subgroups of the decomposition groups $D_t$ and $D_{t'}$ at $t$ and $t'$, respectively. The stabilizer of $t'$ in  $U/({\rm Z}(\QQ)\cap U)U'$ is $D_t/D_{t'}$ with the inertia subgroup $I(t'/t) = I_t/I_{t'}$. Since the stalk at $t$ of $p_*^{U/({\rm Z}(\QQ)\cap U)U'}V^k(U')_\ell$ is $((V^k(U')_\ell)_{\bar {t'}})^{I(t'/t)}$ the result now follows from 
\[((V^k(U')_\ell)_{\bar {\eta'}}^{I_{t'}})^{I_t/I_{t'}} = ((V^k(U')_\ell)_{\bar {\eta'}}^{I_{t}}=(V^k(U)_\ell)_{\bar \eta}^{I_t}.
\]
\end{proof}

Observe that the sheaf $V^k(U)_\ell$ defined above is compatible with taking the first cohomology.

\begin{prop}\label{P:V(U)} With the above notation and assumption on the parity of $k$ we have 
$$H^1(X_{U'}\otimes_\QQ \bar \QQ, V^k(U')_\ell)^{U/({\rm Z}(\QQ)\cap U)U'} =  H^1(X_{U}\otimes_\QQ \bar \QQ, p_*^{U/({\rm Z}(\QQ)\cap U)U'}V^k(U')_\ell) = H^1(X_{U}\otimes_\QQ \bar \QQ, V^k(U)_\ell).$$
\end{prop}

\begin{proof} Because $\QQ _{\ell}$ is flat over $\ZZ _{\ell}$ it follows from the definition of $V^k(U')_\ell$ that 
\[
H^1 (X_{U'}\otimes_\QQ \bar{\QQ}, V^k(U')_{\ell}) ^{U/({\rm Z}(\QQ)\cap U)U'} =
  H^1 (X_{U'}\otimes_\QQ \bar{\QQ}, V(U', L_k))^{U/({\rm Z}(\QQ)\cap U)U'}  \otimes _{\ZZ _{\ell}}\QQ _{\ell}.
\]
An elementary argument shows that 
\begin{align*}
  H^1 (X_{U'}\otimes_\QQ \bar{\QQ} , V (U', L_k)) ^{U/({\rm Z}(\QQ)\cap U)U'} &=
  (\varprojlim _n
H^1 (X_{U'}\otimes_\QQ \bar{\QQ}, V_n (U', L_k)))^{U/({\rm Z}(\QQ)\cap U)U'}\\
&=  \varprojlim _n
H^1 (X_{U'}\otimes_\QQ \bar{\QQ}, V_n (U', L_k))^{U/({\rm Z}(\QQ)\cap U)U'}, 
\end{align*}
so the problem is reduced to prove that 
\[
H^1 (X_{U'} \otimes _{\QQ} \bar{\QQ}, V_n (U', L_k))^{U/({\rm Z}(\QQ)\cap U)U'}\cong
H^1 (X_U \otimes _{\QQ} \bar{\QQ}, p_* ^{U/({\rm Z}(\QQ)\cap U)U'} V_n (U', L_k)). 
\] 
For $\ell$ prime to the order of $U/({\rm Z}(\QQ)\cap U)U'$, the higher 
direct images $R^i p_*^{U/({\rm Z}(\QQ)\cap U)U'}  (\mathcal F) = 0$ for $i > 0 $ and any sheaf $\mathcal F$ of $\ZZ/\ell ^n$-modules. A spectral 
sequence argument then gives the result (in the case of ordinary sheaf cohomology, 
see \cite[ch. V, cor. of prop. 5.2.3, and th.  5.3.1] {groth57}).
\end{proof}

In \cite{Ohta82} Ohta proved Theorem \ref{T:autoshv1} for $U$ neat using the sheaf $V^k(U)_\ell$. For a general $U$, he chose a neat compact open normal subgroup $U'$ 
of $U$ for which Theorem \ref{T:autoshv1} holds, then took the $U/({\rm Z}(\QQ)\cap U)U'$ invariant part of both sides to show that the trace of the Hecke operator $T_p$ on $S_{k+2}(U)$ is equal to the trace of Frob$_p$ on $H^1(X_{U'}\otimes_\QQ \bar \QQ, V^k(U')_\ell)^{U/({\rm Z}(\QQ)\cap U)U'}$. Proposition \ref{P:V(U)} above implies that 
$H^1(X_{U'}\otimes_\QQ \bar \QQ, V^k(U')_\ell)^{U/({\rm Z}(\QQ)\cap U)U'}$ indeed is equal to $H^1(X_{U}\otimes_\QQ \bar \QQ, V^k(U)_\ell).$

\begin{remark}
When the group $\Gamma$ is a finite index subgroup of $O_B^1$ for a split or nonsplit indefinite quaternion algebra $B$ over $\QQ$, the group of Atkin-Lehner involutions is defined as the quotient $\Gamma^*/(\{\pm id\}\Gamma )$, where $\Gamma^*$ is the normalizer of $\Gamma$ in the group of elements in $O_B$ with positive reduced norm. It is a finite elementary 2-group. A split $B$ corresponds to $G={\textsf {GL}}_2$. When $\Gamma$ is $\Gamma_0(N)$ or $\Gamma_1(N)$, these are the familiar Atkin-Lehner operators $W_Q$ in the literature (cf. \cite{AL70, AL78}), where $Q$ is the highest power of a prime $p$ dividing $N$. A nonsplit $B$ corresponds to $G = {\textsf B}^\times$. When $\Gamma = O_B^1$, the group $\Gamma^*/(\{\pm id\}\Gamma)$ is generated by $r$ elements, which can be chosen to be any element in $O_B$ with reduced norm $p$ at the $r$ places $p$ where $B$ ramifies (see \S28.9 of \cite{Voight-book}). More discussions are given in \S5.1. We shall regard the Atkin-Lehner involutions as elements in $G(\mathbb A_f)$.

Let $U$ be the open compact subgroup of $G(\mathbb A_f)$ corresponding to $\G$ and let $U''$ be the group generated by $U$ and some Atkin-Lehner operators of $\Gamma$. Let $U'$ be a neat normal open compact subgroup of $U$. Then the same argument as before shows that the sheaf $V^k(U')_\ell$ admits $U''/({\rm Z}(\QQ)\cap U'')U'$-structure so that the sheaf $V^k(U'')_\ell$ on $X_{U''}$ can be defined as the pushforward $q_*^{U''/({\rm Z(\QQ)}\cap U'')U'} V^k(U')_\ell$ via the projection $q : X_{U'} \to X_{U''}$. 
\end{remark}

For $U''$ defined above, when the curve $X_{U''}$ is defined over $\QQ$, Theorem \ref{T:autoshv1} also holds for $U''$. 

\subsubsection{}
\label{ss: 2.2.4}
The goal of this subsection is to investigate the Galois structure of the stalks of $V^{k}(U) _\ell$. We begin by proving a general statement for the quaternion case inspired by Theorem 3.1.2 of \cite{ALLL}.  
{\begin{prop}
\label{P:quatrep}
Let $B$ be a quaternion algebra over $\QQ$ and $K$ an extension of $\QQ$ such that $B_K = B \otimes _{\QQ}K \cong {\rm M}_2(K)$. Let $V_K$ be a 4-dimensional $K$-vector space which is a principal $B_K$-module. Suppose a group $H$ acts on $V_K$ via the representation $\rho_K$ which commutes with the $B_K$-action.  
Then, with the standard action of ${\M}_2(K)$ on the $2$-dimensional $K$-vector space $K^2$, we have

(1) $$W_K := \mr{Hom} _{\mr{M}_2(K)}  (K^2, V_K) = \mr{Hom} _{B_K}  (K^2, V_K)$$ 
is a 2-dimensional $K$-vector space endowed with the $H$-action via the 
representation $\sigma_K$ induced from the $H$-action on $V_K$; 

(2) $\rho _K \cong \sigma _K \oplus \sigma _K$ as $H$-modules, {and $\rho_K \cong \sigma_K \otimes \varrho_1 (K)$ as $H \times \mr{GL}_2 (K)$-modules.} 
\end{prop}
\begin{proof}
The induced $H$-action on $W_K=\mr{Hom} _{\mr{M}_2(K)}  (K^2, V_K) $ is given by the rule
$$(\sigma_K(h) f)(z) = \rho _K (h)(f(z)) \quad {\rm for}~ h \in H, f \in W_K, ~{\rm and}~  
z \in K^2.$$ We verify that $\sigma_K(h)f$ lies in $W_K$. Indeed, for $m \in M_2(K)$ and $z \in K^2$, we have
$$(\sigma_K(h)f)(mz) = \rho_K(h)(f(mz)) = \rho_K(h)( m f(z)) = m (\rho_K(h)(f(z))) = m (\sigma_K(h) f)(z)$$ due to the assumption that $\rho_K(h)$ preserves the $B_K$-action. This makes the canonical evaluation map
\[
(*) \qquad K^2 \otimes _K  \mr{Hom} _{\mr{M}_2(K)}  (K^2, V_K)\to   V_K 
\quad {\rm given~by}~~
(z, f) \mapsto f(z)
\]
$H$-equivariant. 

Next we show that $W_K = \mr{Hom} _{\mr{M}_2(K)}  (K^2, V_K)$ is a 2-dimensional $K$-vector space. Since $V_K$ is a principal M$_2(K)$-module, denote it by $V_K = {\rm M}_2(K) v$  for some nonzero $v\in V_K$ with the ${\rm M}_2(K)$-action given by left multiplication. 
Write the elements in $K^2$ as a column vector so that M$_2(K)$ acts on $K^2$ by the matrix product and $K^2 = {\rm M}_2(K)\begin{pmatrix}1 \\0\end{pmatrix}$.  
Hence an element $f \in W_K$ is determined by $f\left (\begin{pmatrix}1 \\0\end{pmatrix}\right)= \begin{pmatrix} a & b\\c & d \end{pmatrix} v$ for some $a,b,c,d \in K$. As the stabilizer of $\begin{pmatrix}1 \\0\end{pmatrix}$ is the set  $ \left \{\begin{pmatrix} 1 & x\\0 &y \end{pmatrix}~ :~ x, y \in K \right\}$, it follows that $\begin{pmatrix} 1 & x\\0 &y \end{pmatrix}\begin{pmatrix} a & b\\c & d \end{pmatrix} v = f\left (\begin{pmatrix} 1 & x\\0 &y \end{pmatrix}\begin{pmatrix}1 \\0\end{pmatrix}\right ) = f\left (\begin{pmatrix}1 \\0\end{pmatrix}\right ) =\begin{pmatrix} a & b\\c & d \end{pmatrix} v$ for all $x,y \in K$, which in turn implies $c=d=0$. Therefore $f \mapsto \begin{pmatrix} a & b\\0 & 0 \end{pmatrix}$ is a $K$-linear bijection between $W_K$ and $K^2$. We record this correspondence by writing $f_{(a,b)}$ for $f$. This proves that $W_K$ is a $2$-dimensional $K$-vector space on which $H$ acts via $\sigma_K$, which is assertion (1). Notice that $f_{(a,b)}\left (\begin{pmatrix}x & y\\z & w \end{pmatrix}\begin{pmatrix}1 \\0\end{pmatrix}\right) = \begin{pmatrix}ax & bx\\az & bz \end{pmatrix}v$. 

Observe that $K^2 \otimes_K W_K \cong W_K \oplus W_K$. Finally we prove the surjectivity of the evaluation map $(*)$, from which the first assertion in (2) will follow. To see this, take any element $\begin{pmatrix}x & y\\z & w \end{pmatrix}v \in V_K$ and note that  
$$\begin{pmatrix}x & y\\z & w \end{pmatrix}v =\begin{pmatrix}x & 0\\z & 0 \end{pmatrix}v + \begin{pmatrix}0 & y\\0 & w \end{pmatrix}v= f_{(1,0)}(\begin{pmatrix}x & y\\z & w \end{pmatrix}\begin{pmatrix}1 \\0\end{pmatrix}) + f_{(0,1)}(\begin{pmatrix}y & x\\w & z \end{pmatrix} \begin{pmatrix}1 \\0\end{pmatrix}).$$ 
For the second assertion, observe that $\sigma_K \oplus \sigma_K = \sigma_K \otimes (1_K \oplus 1_K) = \sigma_K \otimes \varrho_1 (K)$ as an $H\times \mr{GL}_2 (K)$-module. 
\end{proof}
The main applications of the above proposition are to the following settings in which $\G$ arises from an indefinite quaternion algebra $B$ defined over $\QQ$:
\begin{itemize}
\item[(a)] $K=\CCC$ and $\G$ cocompact and torsion-free. Denote by $A_z$ the fiber at $z$ of the universal family of  2-dimensional abelian varieties 
over a modular curve $X_{\Gamma}$. Then we have a representation of 
$H = \pi _1 (X_{\Gamma}, z)$ on $A_z$. It commutes with the action of the quaternions 
$B_K$ on $V_K = H^1 (  A_z , \CCC)$. It is known that $V_K $ is a principal $B_K$-module. This 
gives 
\end{itemize}
\begin{cor}
\label{P:autoshv1}
Assume that $\Gamma$ is cocompact and torsion-free. Then 
\begin{itemize}
\item[1.]
  $R^1 f_{*}\CCC$ is, in many ways, isomorphic to a direct sum
$\sigma _{\Gamma, \CCC}\oplus \sigma _{\Gamma, \CCC}$  for a rank 2 local 
system $\sigma _{\Gamma, \CCC}$ on $X_{\Gamma}^\circ$. 
\item[2.] 
\[
V^{k}(\G) _{\CCC} = \mr{Sym}^{k} ( \sigma _{\Gamma, \CCC}  ).
\]
\end{itemize}
\end{cor}
\begin{itemize}
\item[(b)] Same as above, except that we consider the universal abelian variety over $X_U$ where $U$ is the compact open subgroup of $G(\mathbb A_f)$ corresponding to $\G$, $z$ is a point of $X_U$, $H$ is Grothendieck's 
$ \pi _1 (X_{U}, \bar{z})$, the field $K$ is an extension of $\QQ _{\ell}$ which splits $B$, and $V_K = H^1 (  A_{\bar{z}} , K)$.   
Then $\sigma _K$ defines a local system of $K$-vector spaces of dimension 2 on $X_U$ for the \'etale topology.
\item[(c)] For the same $K$ as in (b), this time we consider  $V_K= H^1 (  A_{\bar{z}}, K)$ as an $ H =  \mr{Gal} (\bar{\QQ}/F)$-module
for a number field $F$ over which the fiber $A_{\bar{z}}$ is defined, $F\otimes \QQ_\ell$ contains $K$, and the actions of $B_K$ and $H$ on $V_K$ commute. This is the representation $\rho_K$ in Proposition \ref{P:quatrep}. See more details below. An explicit description of the 2-dimensional representation $\sigma_K$ occurring in $\rho_K = \sigma_K \oplus \sigma_K$ for the groups $\G$ in Theorem \ref{thm:traceformula} is provided in Proposition \ref{prop:eta}. 

\end{itemize}
}

\bigskip

Now we examine the structure of the stalks of $V^{k}(U) _\ell$ at algebraic points in $X_U^\circ$. 
To ease our notation, we shall write $V^k(U)_{\bar \l, \ell}$ for the stalk $(V^k(U)_\ell)_{\bar \l}$ and denote by $\rho^k(U)_{\l,\ell}$ the Galois representation on this stalk.

We first study the case $G = \textsf{B}^\times$ and $U$ not containing $-{\rm id}$.  Let $\lambda \in X_{U}^\circ (L)$ for a number field $L$. Since $-{\rm id}$ is not in $U$, $V^1(U)_{\ell} $ is defined as the pushforward of the sheaf $V^1(U')_\ell$ on $X_{U'}$ for a neat normal open compact subgroup $U'$ of $U$ via the covering map $X_{U'} \to X_U$. Thus the stalk of $V^1(U)_{\ell}$ at $\l \in X_{U}^\circ(L)$ is the same $\QQ_\ell$-space as the stalk of $V^1(U')_{\ell}$ at any point $\l' \in X_{U'}^\circ$ above $\l$. Here $\l'$ is rational over a finite extension $L'$ of $L$. As remarked before, Ohta proved that $V^1(U')_\ell =R^1 f_* \QQ_\ell$ for a family of abelian surfaces $f: A_{U'} \to X_{U'}$. Moreover, the action of $B$ on each fiber commutes with the Galois action. Hence $B$ acts on $V^1(U)_{\bar{\l}, \ell}$ and commutes with the action of a finite index subgroup $G_{L'}$ of $G_L$. 

Let $E$ be a real quadratic extension of $\QQ$ which splits $B$. Fix a prime 
$\mathfrak{l}$ of $E$ above $\ell$. Then Proposition 
\ref{P:quatrep} applied to $K = E_{\mathfrak{l}}$, the completion of $E$ at $\mathfrak{l}$, $H = G_{L'}$, and $\rho_K = \rho^1(U)_{\l, \ell} \otimes_{\QQ_\ell} E_{\mathfrak{l}}$ shows that, as a $G_{L'}$-module, $\rho^1(U)_{\l, \ell}\otimes E_{\mathfrak{l}}$ decomposes as 
$$\rho^1(U)_{\l, \ell} \otimes _{\QQ_\ell} E_{\mathfrak{l}} = \sigma _{\lambda, \mathfrak l}\oplus \sigma _{\lambda, \mathfrak l} = \sigma _{\lambda, \mathfrak l} \otimes (1_{G_{L'}} \oplus 1_{G_{L'}}).$$
for a canonical representation $\sigma _{\lambda, \mathfrak l}$  of $ G_{L'}$ on $E_{\mathfrak{l}} ^2$. 

We proceed to obtain a tensor product decomposition over $G_F$.
When $\rho^1(U)_{\l,\ell}\otimes E_{\mathfrak{l}}$ is irreducible, using Clifford theory in the form of Theorem 2.3 of \cite{LLL} we know that there are two 2-dimensional $\mathfrak{l}$-adic representations $\eta_{\l,\mathfrak l}$ and $\gamma_{\l,\mathfrak l}$ of $G_F$, where $\eta_{\l,\mathfrak l}$, up to a finite order character, is an extension of  
$\sigma_{\l,\mathfrak l}$ and $\gamma_{\l,\mathfrak l}$ has finite image, such that 
\begin{equation}\label{eq:Clifford}
    \rho^1(U)_{\l,\ell}\otimes E_{\mathfrak{l}}\cong \eta_{\l,\mathfrak l} \otimes\gamma_{\l,\mathfrak l} \quad {\rm as}~G_F-{\rm modules}.
\end{equation}

This also holds when $\rho^1(U)_{\l,\ell}\otimes E_{\mathfrak{l}}$ is reducible, for in this case, $\rho^1(U)_{\l,\ell}\otimes E_{\mathfrak{l}}$ is the sum of two degree-$2$ representations of $G_F$ whose restrictions to $G_{L'}$ are equal to $\sigma_{\l,\mathfrak l}$. Thus they differ by a finite character $\xi$ of $G_F$ trivial on $G_{L'}$. We may take one component to be $\eta_{\l,\mathfrak l}$ and $\gamma_{\l,\mathfrak l} = 1 \oplus \xi$. 

\begin{remark} As an extension of $\sigma_{\l,\mathfrak l}$, the representation $\eta_{\l,\mathfrak l}$ in (\ref{eq:Clifford}) is unique up to twisting by a finite order character $\mu_{\mathfrak l}$ of $G_F$, in which case $\gamma_{\l,\mathfrak l}$ is twisted by $\mu_{\mathfrak l}^{-1}$. 
\end{remark}

It follows from (\ref{eq:Clifford}) that \footnote{ Assume $V$ and $W$ both 2-dimensional $G$-modules with a basis $\{v_1,v_2\},\{w_1,w_2\}$ respectively. Then $V\otimes W$ has a basis $\{v_i\otimes w_j\}_{1\le i,j\le 2}$ so that  $\bigwedge^2 (V\otimes W)$ has a basis $\{(v_{i_1}\otimes w_{j_1} )\wedge (v_{i_2}\otimes w_{j_2} )\mid i_1\neq i_2  \text{ or } j_1\neq j_2\}$. The vectors   $(v_i\otimes w_1 )\wedge (v_i\otimes w_2 )$ with $i=1,2$ together with $(v_1\otimes w_1 )\wedge (v_2\otimes w_2 )+(v_2\otimes w_1 )\wedge (v_1\otimes w_2 )$ 
span a sub-$G$-module, which is isomorphic to Sym$^2V\otimes \det W$. Switching the roles of $V$ and $W$ will give the other one.  It is independent of the choice of $\eta_{\l,\mathfrak l}$ in (\ref{eq:Clifford}) by the remark above.}
\begin{equation}\label{eq:2ndWedge}
    \bigwedge^2 \rho^1(U)_{\l,\ell}\otimes E_{\mathfrak{l}}\cong \left(\text{Sym}^2 \eta_{\l,\mathfrak l} \otimes \text{det}\gamma_{\l,\mathfrak l} \right) \oplus \left( \text{det}\eta_{\l,\mathfrak l} \otimes \text{Sym}^2 \gamma_{\l,\mathfrak l}\right)
\end{equation}  as a sum of two degree-3 $G_F$-modules.  In particular, $\left(\text{Sym}^2 \eta_{\l,\mathfrak l} \otimes \text{det}\gamma_{\l,\mathfrak l} \right)|_{G_L'}$ is a subrepresentation of $\sigma _{\lambda, \mathfrak l}\otimes \sigma _{\lambda, \mathfrak l}.$ \bk
By Proposition \ref{P:quatrep} (2), $\rho^1(U)_{\l,\ell}\otimes E_{\mathfrak{l}}$ is also a $\mr{GL}_2(E_{\mathfrak{l}})$-module. To see the $\mr{GL}_2(E_{\mathfrak{l}})$-module structure of $\bigwedge^2 \rho^1(U)_{\l,\ell}\otimes E_{\mathfrak{l}}$, we regard (\ref{eq:2ndWedge}) as $G_{L'}$-modules. As such, on the right hand side, the first factor is  $\text{Sym}^2 (\sigma_{\l,\mathfrak l}\mid G_{L'}) \otimes 1_{E_{\mathfrak{l}}}$, which  consists of three copies of a 1-dimensional $\mr{GL}_2(E_{\mathfrak{l}})$-module, while the second factor reads $\det(\sigma_{\l, \mathfrak l}\mid G_{L'}) \otimes \text{Sym}^2(\varrho_1(E_{\mathfrak{l}}))$, which is a degree-3 irreducible $\mr{GL}_2(E_{\mathfrak{l}})$-module.

On the other hand, by Proposition (4.3.1) and the proof of Proposition (4.4.2) of \cite{Ohta83}, {
$$\bigwedge^2 \rho^1(U)_{\l,\ell}\otimes E_{\mathfrak{l}} =\bigwedge^2 \rho^1(U')_{\l',\ell}\otimes E_{\mathfrak{l}}  \cong (R^2 f_{*}\QQ _{\ell})_{\bar {\l'}} \otimes E_{\mathfrak{l}} \cong  (V^2(U')_{\bar{\l'},\ell} \otimes E_{\mathfrak{l}}) \oplus E_{\mathfrak{l}}(-1)^3 $$ as $G_{L'}$-modules,
where $E_{\mathfrak{l}}(-1)$ denotes the Tate twist of the constant sheaf $E_{\mathfrak{l}}$ by $-1$. As $\mr{GL}_2(E_{\mathfrak{l}})$-modules, on the right side, the first factor is  of degree-3 irreducible  generically, while the second factor is three copies of a 1-dimensional module.} This shows that $V^2(U)_{\bar \l,\ell} \otimes E_{\mathfrak{l}}$ and $\text{Sym}^2 \eta_{\l,\mathfrak l} \otimes \text{det}\gamma_{\l,\mathfrak l}$ agree as $G_{L'}$-modules, and hence as $G_F$-modules, they agree up to a (finite order) character of $G_F$ trivial on $G_{L'}$. Thus there is a finite order character $\chi_{\l,\mathfrak l}$ of $G_F$ such that the representation $\rho^2_{\l,\ell}\otimes E_{\mathfrak{l}}$ of $G_F$ on $V^2(U)_{\bar{\l},\ell}\otimes E_{\mathfrak{l}}$ is 
\begin{eqnarray}\label{eq:k=1}
 \rho^2(U)_{\l, \ell} \otimes E_{\mathfrak{l}} = \chi_{\l, \mathfrak l} \otimes \text{Sym}^2 \eta_{\l,\mathfrak l}.
 \end{eqnarray}

We show that this is a general phenomenon. 
\begin{prop}\label{P:ALLL2} {Let $U$ be a group generated by an open compact subgroup of $G(\mathbb A_f)$ and a subset of the Atkin-Lehner operators. Suppose $-{\rm id}\notin U$. Let $\l$ be an algebraic point of $X_{U}^\circ$.}

(1) Suppose $G = \textsf{B}^\times$ for an indefinite quaternion algebra $B$ defined over $\QQ$. 
Let $E$ be a real quadratic field that splits $B$. Fix a prime $\ell$ and a prime $\mathfrak l$ of $E$ above $\ell$. 
Suppose the 4-dimensional $\ell$-adic representation $\rho^1(U)_{\l, \ell}$ 
on the stalk $V^1(U)_{\bar{\l}, \ell}$ is a $G_F$-module for some number field $F$. 
Then there is a $2$-dimensional 
$\mathfrak l$-adic representation $\eta_{\l,\mathfrak l}$ and a finite order character $\chi_{\l,\mathfrak l}$ of $G_F$ such that for all integers $k \ge 1$, the representation $\rho^{2k}(U)_{\l,\ell}\otimes_{\QQ_\ell} E_{\mathfrak l}$ on the fiber $V^{2k}(U)_{\bar{\l},\ell}\otimes_{\QQ_\ell} E_{\mathfrak l}$ is given by 
\begin{eqnarray}\label{eq:sym2k}
\rho ^{2k}(U) _{\lambda, \ell}\otimes_{\QQ_\ell} E_{\mathfrak l}  = \chi_{\l ,\mathfrak l} ^k \otimes \mr{Sym}^{2k} (\eta_{\l,\mathfrak l})
\end{eqnarray}
as $G_F$-modules. Moreover, $\chi_{\lambda, \mathfrak l}\det (\eta_{\l, \mathfrak l}) = \epsilon_{\mathfrak l}$, where $\epsilon_{\mathfrak l}$ denotes the $\mathfrak l$-adic cyclotomic character.

(2) Suppose $G= \textsf{GL}_2$. 
Then  for all integers $k \ge 1$, 
\begin{eqnarray}\label{eq:symk}
 \rho^k(U)_{\l,\ell} = \mr{Sym}^{k} (\rho^1(U)_{\l, \ell}),
\end{eqnarray} 
 and $\det \rho^1(U)_{\l,\ell}$ is the cyclotomic character $\epsilon_\ell$. 

\noindent In both cases, the determinant of the representation is the cyclotomic character raised to the degree of the representation.
\end{prop}

\begin{proof} 
Let $\rho$ be a degree-2 characteristic zero representation of a group. We have, for $m \ge 4$, 
\begin{eqnarray}\label{eq:symm}
\qquad {\rm Sym}^{m-2}(\rho) \otimes {\rm Sym}^{2}(\rho) \cong {\rm Sym}^{m}(\rho) \oplus (\det(\rho) \otimes {\rm Sym}^{m-2}(\rho)) \oplus (\det (\rho)^2 \otimes {\rm Sym}^{m-4}(\rho))
\end{eqnarray} by checking the characters of the representations or the Clebsch Gordan coefficient formula. 

Case (1) $G = \textsf{B}^\times$. Applying (\ref{eq:symm}) to the representation $\rho=\varrho_1(E)$ of GL$_2(E)$, we obtain the canonical decomposition 
\[
    \varrho _{m-2}(E)\otimes \varrho _2(E) \cong   \varrho _{m}(E)\oplus (\det \varrho_1(E)\otimes\varrho _{m-2}(E))\oplus ((\det \varrho_1(E))^2\otimes \varrho _{m-4}(E)).\]
As discussed in \S \ref{ss:2.2.2}, the representations $\varrho'_m(\QQ)$ of $B^\times$ for $m\ge 2$ even and $m=1$ may be regarded as the restriction of $\varrho_m(E)$ to $B^\times$ under the identification $B^\times =\theta(B^\times) \subset {\rm GL}_2(E)$. Thus the above identity holds with $\varrho'_j(E)$ replacing $\varrho_j(\QQ)$, which in turn gives rise to the following identity on sheaves
\begin{eqnarray}\label{eq:V2k}
     V^{2k-2}(U) _{\ell} \otimes V^2 (U)_{\ell} \cong   V^{2k}(U) _{\ell} \oplus  (\QQ_\ell(-1)\otimes V^{2k-2}(U)_\ell) \oplus (\QQ_\ell(-2) \otimes V^{2k-4}(U)_\ell)   
\end{eqnarray}
by the proof of Proposition (4.4.2) in \cite{Ohta83} since the $\ell$-adic rank-1 sheaf arising from $\det(\varrho'_1(\QQ))$ is the Tate twist by $-1$ of the constant sheaf $\QQ_\ell$ on which the Galois group acts by the cyclotomic character $\epsilon_\ell$. Also, $V^0(U)_\ell$ is the constant sheaf $\QQ_\ell$ on $X_U$. The same holds when each sheaf is tensored with $E_{\mathfrak l}$.

We now prove (\ref{eq:sym2k}) by induction on $k$. When $k=1$, this is (\ref{eq:k=1}) proved above. Assume the statement holds for $2k-2$, that is, 
as $G_L$-modules, \[
 V^{2k-2}(U) _{\bar{\lambda}, \ell}\otimes E_{\mathfrak l} \cong \chi_{\lambda, \mathfrak l}^{k-1}\otimes \mr{Sym}^{2k-2}(\eta _{\lambda, \mathfrak l}). 
 \]

\noindent The two surjective $G_F$-homomorphisms 
 $$(V^{2k-2}(U) _{\bar{\l},\ell}\otimes E_{\mathfrak l}) \otimes (V^2 (U)_{\bar{\l},\ell}\otimes E_{\mathfrak l}) \to   V^{2k}(U) _{\bar{\l},\ell}\otimes E_{\mathfrak l}$$
from (\ref{eq:V2k}) tensored with $E_{\mathfrak l}$ and 
$$(\chi_{\lambda, \mathfrak l}^{k-1}\otimes{\rm Sym}^{2k-2}(\eta_{\l,\mathfrak l})) \otimes (\chi_{\lambda, \mathfrak l}\otimes{\rm Sym}^{2}(\eta_{\l,\mathfrak l})) \to \chi_{\lambda, \mathfrak l}^{k}\otimes{\rm Sym}^{2k}(\eta_{\l,\mathfrak l}) $$
from (\ref{eq:symm}) with $\rho=\eta_{\l, \mathfrak l}$ and $m=2k$ both arise from the symmetrization map. As the left hand sides are isomorphic $G_F$-modules by induction hypothesis, so are the right hand sides. This proves (\ref{eq:sym2k}).

Similarly, replacing the right hand sides of the above two surjective homomorphisms by the middle term from the right side of (\ref{eq:V2k}) and (\ref{eq:symm}) respectively, we obtain two isomorphic $G_F$-modules $\det(\eta_{\l, \mathfrak l})\otimes\chi_{\lambda, \mathfrak l}^{k} \otimes {\rm Sym}^{2k-2}(\eta_{\l, \mathfrak l})$ and $E_{\mathfrak l}(-1) \otimes V^{2k-2}(U) _{\bar{\l},\ell}\otimes E_{\mathfrak l}$, which implies the identity $\chi_{\lambda, \mathfrak l}\det (\eta_{\l, \mathfrak l}) = \epsilon_{\mathfrak l}$ since the $G_F$ action on $E_{\mathfrak l}(-1)$ is $\epsilon_{\mathfrak l}$.

Case (2) $G = \textsf {GL}_2$. This follows from Deligne \cite{Deligne-mf}. 
\end{proof}

\begin{remark}\label{rem:3} For the groups $\G$ considered in this paper, the curves $X_\G$ are defined over $\QQ$, hence $X_\G = X_U$ for the group $U$ corresponding to $\G$. Formulae (\ref{eq:sym2k}) and (\ref{eq:symk}) indicate that, when $-{\rm id} \notin U$, the Frobenius traces of $\rho^{2k}(U)_{\l, \ell}$ and $\rho^k(U)_{\l, \ell}$ can be computed from those of $\rho^2(U)_{\l, \ell}$ and $\rho^1(U)_{\l, \ell}$, using the relation given by \eqref{eq:F(S,T)} and \eqref{eq:FOP}, respectively.  When $U$ 
contains $-{\rm id}$, by definition, $\rho^{2k}(U)_{\l, \ell}$ is an extension of $\rho^{2k}(U')_{\l', \ell}=\chi_{\l',\ell}^k \otimes {\rm Sym}^{2k}(\eta_{\l', \ell})$ for some neat compact open normal subgroup $U'$ of $U$, resulting from the sheaf $V^{2k}(U)_\ell$ being a push-forward of $V^{2k}(U')_\ell$. While $\rho^{2}(U)_{\l, \ell}$ may no longer be a symmetric square up to twist by a character, the recursive relation \eqref{eq:V2k} holds over $X_U^\circ$. In particular, this implies that if $\l$ is in $X_U^\circ(\QQ)$, then at a good prime $p$, the trace of $\rho^{2k}(U)_{\l, \ell}$ at Frob$_p$ is equal to $F_k(S, T)$ defined by \eqref{eq:F(S,T)} with $S$ being the trace of $\rho^2(U)_{\l, \ell}({\rm Frob}_p)$ and $T = p$, the same expression as the case $U$ not containing $-{\rm id}$. 
This will be used in the proofs of the theorems listed in \S\ref{S:mot}. Thus our work is reduced to determining the Frobenius traces of $\rho^2(U)_{\l, \ell}$ or $\rho^1(U)_{\l, \ell}$ as appropriate. A key contribution of this paper is to showcase certain groups $U$ (or their corresponding $\G$ as discussed in \S\ref{S:mot}) for which we can identify suitable hypergeometric sheaves with explicit Galois actions which are isomorphic to the automorphic sheaves in question.    
\end{remark}

We end this section by discussing the integrality of Frobenius traces of the Galois action on the stalk $V^k(U)_{\bar{\lambda}, \ell}$ at an algebraic point $\lambda$ in $X_U^\circ$. First assume the group $G = \textsf{B}^\times$ for a definite quaternion division algebra $B$ defined over $\QQ$. 
{ If $A$ is a 2-dimensional  abelian variety defined over a number field $L$ with multiplication by $B$, then the action of $\mr{Gal}(\bar{L}/L)$ on the first \'etale cohomology
group factors as
\[
\rho _{\ell} : \mr{Gal}(\bar{L}/L) \to (B\otimes _{\QQ}\QQ _{\ell})^{\times}  =  {\textsf B}^\times (\QQ _{\ell})
\subset \mr{Aut} (H^1 (A\otimes _L \bar{L},\QQ_{\ell} )) \sim
\mr{GL}_4 (\QQ _{\ell}).
\]
This is because the commutant of $B$ in 
$\mr{End} (H^1 (A\otimes _L \bar{L},\QQ_{\ell} ))\simeq \mr{M}_4 (\QQ_{\ell}) $ is 
$B_{\ell} = B\otimes _{\QQ}\QQ _{\ell}$.  It follows from the construction of the $\ell$-adic sheaf in \S \ref{ss:2.2.2} that, 
the stalks 
 $V^k (U) _{\bar{\lambda}, \ell}$ in an $L$-rational point $\lambda \in X_U^\circ (L)$, as a Galois module, are in the shape
\[
\varrho_k'(\QQ_\ell) \circ \rho_{\ell} : \mr{Gal}(\bar{L}/L) \to {\textsf B}^\times (\QQ _{\ell})
\to \textsf{ GL}_m (\QQ _{\ell})
\]
where $m= k+1$ for $k$ even and $m=2(k+1)$ for $k$ odd, as explained in \S \ref{ss:2.2.2}.  
From this, one can prove that the coefficients of the characteristic
polynomials in unramified primes $\mathfrak{p}$ of $L$
\begin{eqnarray}\label{eq:traceZ}
\det \left (  1 - T\, \mr{Frob}_{\mathfrak{p}} \mid  V^k (U) _{\bar{\lambda}, \ell} \right )\in \ZZ[T]
\end{eqnarray}
have integer coefficients. We will not prove this here. The above reasoning assumes that 
$U$ is neat. To do a general $U$, 
which may include some Atkin-Lehner operators as above, note that
the Frobenius traces can be computed, as in the theory of Artin $L$-functions, by averaging 
the Frobenius traces in the fibers over $\lambda$ in a covering 
$X_{U' }\to X_U$ for a neat $U'$. 

A similar argument yields the same conclusion for $G = \textsf {GL}_2$. In this case $A$ is an elliptic curve, $\mr{End} (H^1 (A\otimes _L \bar{L},\QQ_{\ell} ))\simeq \mr{M}_2 (\QQ_{\ell}) $ and the Galois action on the stalk $V^k (U) _{\bar{\lambda}, \ell}$ has the shape $\mr{Gal}(\bar{L}/L) \to \textsf{GL}_{k+1} (\QQ _{\ell})$. 
}

\section{Groups}\label{ss:groups}
\subsection{Hypergeometric functions over $\CCC$}\label{SS:hypergeometricoverC} We first recall some relevant notation in hypergeometric functions. Let $\G(x)$ denote the Gamma function. For  $k \in \ZZ$ and $a\in \CCC$, define the Pochhammer symbol $(a)_k:=\G(a+k)/\G(a)=a(a+1)\cdots(a+k-1)$.  
A hypergeometric datum is a pair of multi-sets  $\alpha=\{a_1,\cdots,a_n\},\beta=\{b_1=1,b_2,\cdots, b_n\}$ with $a_i,b_j\in \QQ$  to which we associate the Hypergeometric function in variable $t \in \CCC$: \begin{equation}\label{eq:F}
\displaystyle F(\alpha,\beta;t) = \,_nF_{n-1}\left[\begin{matrix}a_1&a_2& \cdots & a_n\\ &b_2& \cdots &b_n \end{matrix} \; ; \; t \right] :=\sum_{k\ge 0} \frac{(a_1)_k\cdots(a_n)_k}{(b_1)_k\cdots(b_n)_k}t^k.    
\end{equation}It satisfies the Fuchsian ordinary differential equation below with three regular singular points at $0, 1, \infty$: 

	$$
	\left[\theta\left(\theta+b_2-1\right)\cdots \left(\theta+b_n-1\right) - t\left(\theta+a_1\right)\cdots \left(\theta+a_n\right) \right]F=0, \quad \text{where}~ \theta:=t\frac{d}{dt}.
	$$
 The local solutions of this equation form a rank-$n$ hypergeometric local system. 
\begin{theorem}\label{thm:localexponent}
			The local exponents of the above hypergeometric differential equation   are 
			\begin{equation}
			\begin{split}\label{eq:indicial}
			0,1-b_2,\cdots,1-b_n  & \quad \text{ at } t=0,\\
			a_1,\,a_2,\,\cdots,\,a_n &  \quad \text{ at } t=\infty,\\
			0,1,2,\cdots,n-2,\gamma& \quad \text{ at } t=1,
			\end{split}
			\end{equation} respectively, where \begin{equation}\label{eq:gamma}
			\gamma=-1+\sum_{j=1}^n b_j - \sum_{j=1}^n a_j. 			\end{equation}  
\end{theorem} See \cite[\S 2]{Beukers-Heckman} by Beukers and Heckman for more details. We say that a local monodromy matrix in $\GL_n(\RR)$ has exponents $(r_1, \cdots, r_n)  \in (\QQ/\ZZ)^n$ if its eigenvalues are  $e^{2\pi i r_1}, \cdots, e^{2\pi i r_n}$.
In particular, the  above theorem 
says the local monodromy at 1 with monodromy matrix $M_1$ is a pseudoreflection (cf. \cite{ Beukers-Heckman,Katz}), namely  the rank of $M_1-I_n$ is 1, where $I_n$ stands for the rank-$n$ identity matrix. \bk
 
 Next we recall a couple useful transformation formulae. Combined with the Euler transformation formula (cf. \cite[Theorem 2.2.5]{AAR}) 
 \begin{equation}\label{eq:Euler}
  \pFq21{a&b}{&c}{t}=(1-t)^{c-a-b}   \pFq21{c-a&c-b}{&c}{t},
 \end{equation}
 a version of the Clausen formula (\cite{AAR} or \cite[Eqn. (39)]{LLT2}) says that for $a,b,t\in \CCC$,   
 \begin{equation}\label{eq:clausen1}
    {(1-t)^{-\f12}} \pFq32{\f12&a-b+\f12&b-a+\f12}{&a+b+\f12&\frac32-a-b}{t}=  \pFq21{a&b}{&a+b+\f12}{t}\pFq21{1-a&1-b}{&\frac32-a-b}{t}
\end{equation}when both sides are convergent.

\subsection{A motivating example}\label{ss:3.2}
We demonstrate the main motivation  of this section using the following classical example. Let  $\G=\PSL2Z\cong (2,3,\infty)$ whose modular curve $X_\G= X(1)$ is parameterized by the modular $j$-function. The open subset $X(1)^\circ:=X(1)\setminus\{0,1728,\infty\}$ admits the universal elliptic curve 
\begin{equation}\label{eq:Universal-E}
    \mathcal E_j:\quad  y^2+xy=x^3-\frac{36x+1}{j-1728}.
\end{equation} 
Through the fiber map $f:\mathcal E_j \mapsto j$, the local constant sheaf $\QQ_\ell$ on $\mathcal E_j$ gives rise to 
$\mathcal F_\ell =R^1f_* \QQ_\ell$ on $X(1)^\circ$.  The Picard-Fuchs equation of this elliptic fibration, using an algorithm in \cite[Theorem 1.5]{Stienstra-Beukers} by Stienstra and Beukers, is  hypergeometric with one solution given by  $(1-t)^{ -1/4}\pFq21{\frac1{12}&\frac5{12}}{&1}{t}$ with $t=1728/j$. Note that the hypergeometric datum $\{\{\frac1{12},\frac5{12}\},\{1,1\}\}$ is not defined over $\QQ$.  However  \begin{multline*}
    \left ((1-t)^{ -1/4}\pFq21{\frac1{12}&\frac5{12}}{&1}{t}\right)^2\overset{\eqref{eq:Euler}}=\pFq21{\frac1{12}&\frac5{12}}{&1}{t}\pFq21{\frac{11}{12}&\frac7{12}}{&1}{t}\\  \overset{\eqref{eq:clausen1}} =(1-t)^{- 1/2}\pFq32{\frac1{2}&\frac1{6}&\frac56}{&1&1}{t}, 
\end{multline*} in which the hypergeometric datum $\{\{\frac1{2},\frac1{6},\frac56\},\{1,1,1\}\}$ in the last expression is defined over $\QQ$.

\noindent Here a datum $HD=\{\alpha=\{a_1,\cdots,a_n\}, \beta=\{b_1=1, b_2, \cdots, b_n\}\}$ with $a_i, b_j\in\QQ^\times$ is said to be defined over $\QQ$ if the set of column vectors $\{\begin{pmatrix}a_1\\b_1\end{pmatrix}, \cdots, \begin{pmatrix}a_n\\b_n\end{pmatrix}\}$ mod $\ZZ$ is invariant  under multiplication by all $r \in (\ZZ/M\ZZ)^\times$, where $M = M(HD)$, called the level of $HD$, is the least positive common denominators of $a_1,..., a_n, b_1,..., b_n$. 
See Definition 1 in \cite[\S2.2]{LLT2}. Also, in this paper we mainly consider primitive hypergeometric data $HD$, namely $a_i -b_j \notin \ZZ$ for all $i, j$, such that the corresponding local systems are irreducible.

\subsection{The groups to be considered}\label{ss:3.3}
We now restrict our discussion from congruence subgroups of $\SL2R$ to a class of arithmetic triangle groups which have been classified by Takeuchi \cite{Takeuchi-triangle, Takeuchi-classify}. Petersson \cite{Petersson} (see also \cite{Clark-Voight} by Clark and Voight) provides an algorithm to express the generators as  matrices over a totally real subfield of a cyclotomic field. 
Among all arithmetic triangle groups, we 
consider those  
whose corresponding rank-2 local systems $V$ are  hypergeometric with  datum $\alpha_2=\{a,b\},\beta_2=\{1,c\}$ such that either 

a) The datum $\{\alpha_2,\beta_2\}$ is defined over $\QQ$; or, 

b) Sym$^2V$ is also hypergeometric,  which is very special due to the pseudoreflection requirement, and its corresponding datum is defined over $\QQ$ (as the example in \S \ref{ss:3.2}). 
\medskip

The above criteria give rise to the following groups, classified according to the conditions they satisfy: \medskip

a.I) $(m,\infty,\infty)$ with $m=3,\infty$,

b.I) $(2,m,\infty)$ with $m=3,4,6,\infty$, 

b.II) $(2,4,6),(2,6,6)$. 

All groups are isomorphic to a subgroup of SL$_2(\RR)$ mod $\pm I_2$, while those in a.I) are the only ones isomorphic to a subgroup of SL$_2(\RR)$ (in fact $\mr{SL}_2(\ZZ)$) not containing $- I_2$.

We now describe how  we arrive at the lists.
Given a triangle group $\G=(e_0,e_1,e_\infty)$,  following Theorem  9  of  \cite{Yang-Schwarzian} by Y. Yang we introduce
the following  hypergeometric parameters: $$a=\frac12(1-\frac 1{e_1}-\frac1{e_0}-\frac1{e_\infty}),\quad b=\frac12(1-\frac 1{e_1}-\frac1{e_0}+\frac1{e_\infty}), \quad c=1-\frac 1{e_0}, $$
and
$$\tilde a=\frac12(1-\frac 1{e_1}+\frac1{e_0}-\frac1{e_\infty}),\quad \tilde b=\frac12(1-\frac 1{e_1}+\frac1{e_0}+\frac1{e_\infty}), \quad \tilde c=1+\frac 1{e_0}. $$ Using these Yang wrote down an explicit basis for $S_k(\G)$ in terms of $\pFq21{a&b}{&c}{t}$ and $\pFq21{\tilde a&\tilde b}{&\tilde c}{t}$ which satisfy ordinary differential equations that are projectively equivalent.  When $e_0=e_1=\infty$, the corresponding arithmetic triangle group $(e_\infty,\infty,\infty)$ can be identified with an index-2 subgroup of a Hecke group up to conjugation in PSL$_2(\RR)$. We further check whether they have a realization in SL$_2(\RR)$ not containing $-I_2$. This process leads to the list a.I) above with $HD(\G)=\{\alpha=\{\frac12(1-\frac1{e_\infty}),\frac12(1+\frac1{e_\infty})\}, \beta=\{1,1\}\}$, in which the Hauptmodul is chosen such that the two cusps are located at $t=0$ and $t=1$. 

If, in particular, $e_1=2$, then
  the Clausen formula recalled in  \eqref{eq:clausen1} implies
   \begin{equation}\label{eq:Clausen1}
       \pFq21{a&b}{&c}{t} \pFq21{\tilde a&\tilde b}{&\tilde c}{t}=\pFq32{\frac 12&\frac 12-\frac 1{e_\infty}&\frac 12+\frac 1{e_\infty}}{& 1-\frac 1{e_0}&1+\frac 1{e_0}}t.
   \end{equation} 
    
\noindent    The hypergeometric datum corresponding to the right hand side of \eqref{eq:Clausen1} is 
\begin{eqnarray}\label{eq:HD(Gamma)}
HD:=HD( \G):=\left \{\alpha=\{\frac 12,\frac 12-\frac 1{e_\infty},\frac 12+\frac 1{e_\infty}\},\beta=\{1, 1-\frac 1{e_0},1+\frac 1{e_0}\}\right\}
\end{eqnarray}
in which $\alpha$ and $\beta$ are self-dual. 
 The datum is defined over $\QQ$ if $e_0,e_\infty\in \{3,4,6,\infty\}$. This leads to the groups listed in b.I) for $\G$ non-cocompact and b.II) for $\G$ cocompact. 

As explained in \S \ref{SS:hypergeometricoverC}, a hypergeometric sheaf is defined on $\mathbf P^1$ with singularities at $0, 1, \infty$. When we associate a hypergeometric sheaf to a triangle group $\G$, the three singularities of the sheaf are set at the three vertices of the curve $X_\G$ (as a compact Riemann surface). For this to happen, the model for $X_\G$ should be a genus zero curve with the three singular points that are $\QQ$-rational.

The models we choose for the non-cocompact groups are as follows
$$
\begin{array}{cccc}
   (2,\infty,\infty)&(2,3,\infty)&(2,4,\infty)&(2,6,\infty) \\
  \G_0(2)/\{\pm I\}& {\rm PSL}_2(\ZZ)& \langle \G_0(2), w_2 \rangle/\{\pm I\}& \langle \G_0(3), w_3 \rangle /\{\pm I\}
\end{array}
$$
where $w_2$ and $w_3$ are Atkin-Lehner involutions, so that $X_\G$ are projective curves over $\QQ$ and all elliptic points and cusps are $\QQ$-rational. Hence we may choose $\l(\G)$ as described in the table in Theorem \ref{thm:traceformula}. For the cocompact group $(2,4,6)$, more explanation is needed. 
The quotient of the norm-$1$ group of $O^1_{B_6}$ by its center $\pm 1$ gives the group $(2,2,3,3)$ in the class II group \ref{fig:class-II}. It is normalized by the Atkin-Lehner involutions $w_2$, $w_3$, and $w_6 = w_2w_3$. The extended groups $\langle(2,2,3,3), w_i\rangle$ for $i=3, 2, 6$ respectively are the groups $(2,6,6)$, $(3,4,4)$ and $(2,2,2,3)$, and $(2,4,6) = \langle(2,2,3,3), w_2, w_3\rangle$. The curve
$X_{(2,2,3,3)}$ is a genus 0 curve defined over $\QQ$ with the canonical model given by $x^2 + 3y^2 + z^2=0$  (cf. \cite{Kurihara}). As described in Baba-Granath \cite{Baba-Granath-g2}, it is a 2-fold cover of the curves $X_{(2,6,6)}$, $X_{(3,4,4)}$, $X_{(2,2,2,3)}$, all projective lines over $\QQ$, via the covering maps sending $[x:y:z] \in X_{(2,2,3,3)}$ to $[y:z]$, $[x:z]$, $[x:y]$, respectively. The curve $X_{(2,4,6)}$ is also a projective line over $\QQ$; it is covered by $X_{(2,6,6)}$, $X_{(3,4,4)}$, $X_{(2,2,2,3)}$ with respective covering maps \begin{equation}\label{eq:pi}
    \pi_3: [y:z] \mapsto \left[-3y^2-z^2: y^2\right], \quad \pi_2: [x:z] \mapsto \left[x^2:\frac{-x^2-z^2}{3}\right], \quad \pi_6: [x:y] \mapsto \left [x^2: y^2\right].
\end{equation}These are the models we choose. As such, the three elliptic points on $X_{(2,4,6)}$ are $[0: 1], [1: 0], [-3: 1]$, all $\QQ$-rational;  for the sake of convenience we may choose the Hauptmodul $\l$ for the curve $X_{(2,4,6)}$ as given in  Theorem \ref{thm:traceformula}. 
But $X_{(2,6,6)}$, $X_{(3,4,4)}$ and $X_{(2,2,2,3)}$ all have elliptic points rational only over a quadratic extension of $\QQ$.  
For this reason, we did not  include $(2,6,6)$ in Theorem \ref{thm:traceformula}. Since $X_{(2,6,6)}$ is a 2-fold cover of $X_{(2,4,6)}$ with an explicit covering map, we will use the  pullback of the hypergeometric sheaf on $X_{(2,4,6)}$ to deal with (2,6,6), see Theorem \ref{thm:266} in Section \ref{sec: other groups}. 

The rank-2 hypergeometric local systems $\mathcal H^1(\G)_\CCC = \mathcal H(HD(\G))_\CCC$ for groups in  a.I) are  listed in the Table \ref{tab:a-list} below. 
\begin{center}
\begin{table}[h]
{\small
\begin{tabular}{|c|c|c|c|c|c|c|c|}
\hline
 $\G=(e_0,e_1,e_\infty)$&Generators& Exponents& {$HD(\G)$} & $\overset{\mathcal H^1(\G)_\CCC { = \mathcal H(HD(\G))_\CCC}}{\text{a particular function}} $\\
\hline
$(\infty,\infty,3)$     & $T:=\begin{pmatrix}1&1\\0&1\end{pmatrix}$ & (1,1) & $\alpha=\{\frac{1}{3}, \frac{2}{3}\}$ &$\mathcal H(\{\frac1{3},\frac23\},\{1,1\})_\CCC$\\  $\cong \G_1(3)$   & $S:=\begin{pmatrix}-2&{3}\\-3&{4}\end{pmatrix}$ &(1,1)&$\beta=\{1,1\}$&   $\pFq21{\frac1{3}&\frac23}{&1}t$\\       & $(ST)^{-1}$ &$(\frac13,\frac23)$& &\\    \hline
$(\infty,\infty,\infty)$ & $T:=\begin{pmatrix}1&1\\0&1\end{pmatrix}$ & (1,1) &$\alpha=\{\frac12, \frac12\}$ &$\mathcal H(\{\frac1{2},\frac12\},\{1,1\})_\CCC$\\
 $\cong \G_1(4)$   &$(ST)^{-1}$ &$(1,1)$&$\beta=\{1,1\}$ & $\pFq21{\frac1{2}&\frac12}{&1}t$  \\
    & $ S:=\begin{pmatrix}1&{-1}\\4&  {-3}\end{pmatrix}$ &$(\frac12,\frac12)$& &\\
      \hline
\end{tabular}} 
\caption{Groups in  a.I) list}
    \label{tab:a-list}
\end{table}
\end{center} 
Notice that the exponents of the generator $S$ of $\G=(\infty,\infty,\infty) \simeq \Gamma_1(4)$ are $(\f12,\f12)$. This means that the point fixed by $S$ is an \emph{irregular} cusp, 
so the contribution of this cusp to $\Tr(T_p\mid S_{k+2}(\G))$ will depend on the parity of $k$ as in the statement of Theorem \ref{thm:traceformula-2F1}. 

For groups in b.I), we also write down the rank-2 local hypergeometric system $\mathcal H^1(\G)_\CCC$. 
The groups $\Gamma$ in b.II) arise from indefinite quaternion algebras defined over $\QQ$. 
As explained in \S\ref{SS:autosh1}, $V^2(\G)_\CCC$ is a 3-dimensional irreducible local system.  
For each case, we also identify a rank-2 hypergeometric local system $\mathcal H^1(\G)_\CCC$ listed in the second to the last column of Table \ref{tab:2},  in which branches for $n$th roots are chosen accordingly to make the local exponents on the third column as claimed.  The  symmetric square of $\mathcal H^1(\G)_\CCC$, denoted by $\mathcal H^2(\G)_\CCC$, is  of the form $(1-t)^{-\f12} F(HD(\G);t)$, where $F(HD(\G);t)$, defined by \eqref{eq:F}, is as shown in the last column of Table \ref{tab:2}. By construction, for each $\G$ in the b) list, the  hypergeometric  datum $HD(\G)$ as given in \eqref{eq:HD(Gamma)} is
defined over $\QQ$. 
The data for  these groups 
are listed in Table \ref{tab:2}.  We will omit (2,6,6) for reasons explained earlier.  \medskip  

\begin{sidewaystable}
\vskip160mm
   \centering
\begin{tabular}{|c|c|c|c|c|c|}
\hline
 $(e_0,e_1,e_\infty)$&Generators&Expo. &M& $\overset{\mathcal H^1(\G)_\CCC}{\text{a particular function}}$ & $ \overset{\mathcal H^2(\G)_\CCC {= \mathcal L(\G)_\CCC\otimes\mathcal H(HD(\Gamma))_\CCC}}{\text{a particular function}}$\\
\hline
$(\infty,2,\infty)$    & $T:=\begin{pmatrix}1&1\\0&1\end{pmatrix}$ & $(1,\,1)$ &&  $\mathcal H(\{\frac14\},\{1\})_\CCC\otimes\mathcal H(\{\frac1{4},\frac14\},\{1,1\})_\CCC$&  $\mathcal H(\{\frac12\},\{1\})_\CCC\otimes\mathcal H(\{\frac1{2},\frac12,\f12\},\{1,1,1\})_\CCC$ \\
    & $S:=\begin{pmatrix}{-1}&1\\{-2}&{1}\end{pmatrix}$ &$(\frac14, \frac34)$&4& $ ~(1-t)^{-\frac14}\pFq21{\frac1{4}&\frac14}{&1}{t}$&$  ~(1-t)^{-\frac12}\pFq32{\frac12&\frac1{2}&\frac12}{&1&1}{t}$ \\
      & $(ST)^{-1}$ &$(\f12,\, \f12)$&&&\\
    \hline
      $(\infty,2,3)$    & $T:=\begin{pmatrix}1&1\\0&1\end{pmatrix}$ & (1,\,1) &&$\mathcal H(\{\frac14\},\{1\})_\CCC\otimes\mathcal H(\{\frac1{12},\f5{12}\},\{1,1\})_\CCC$& $\mathcal H(\{\frac12\},\{1\})_\CCC\otimes\mathcal H(\{\frac1{2},\frac1{16},\f5{6}\},\{1,1,1\})_\CCC$ \\
     & $S:=\begin{pmatrix}0&1\\-1&0\end{pmatrix}$ &$(\frac14, \frac34)$& 12& $(1-t)^{-\frac14}\pFq21{\frac1{12}&\frac5{12}}{&1}{t}$&$(1-t)^{-\frac12}\pFq32{\frac12&\frac1{6}&\frac5{6}}{&1&1}{t}$ \\
      & $(ST)^{-1}$ &$(\frac13, \frac23)$&&&\\
     \hline
$(\infty,2,4)$& $T:=\begin{pmatrix}1&1\\0&1\end{pmatrix}$ & (1,1)& &$\mathcal H(\{\frac14\},\{1\})_\CCC\otimes\mathcal H(\{\frac1{8},\f3{8}\},\{1,1\})_\CCC$&$\mathcal H(\{\frac12\},\{1\})_\CCC\otimes\mathcal H(\{\frac1{2},\frac1{4},\frac3{4}\},\{1,1,1\})_\CCC$ \\
     &$ S:=\frac1{\sqrt 2}\begin{pmatrix}0&1\\{-2}&0\end{pmatrix}$ &$(\frac14, \frac34)$ &8  &$(1-t)^{-\frac14}\pFq21{\frac1{8}&\frac3{8}}{&1}{t}$& $(1-t)^{-\frac12}\pFq32{\frac12&\frac1{4}&\frac34}{&1&1}{t}$   \\
     & $(ST)^{-1}$ &$ (\frac38, \frac58)$&&&\\
      \hline
      $(\infty,2,6)$& $T:=\begin{pmatrix}1&1\\0&1\end{pmatrix}$ & (1,1) & &$\mathcal H(\{\frac14\},\{1\})_\CCC\otimes \mathcal H(\{\frac1{6},\frac1{3}\},\{1,1\})_\CCC$&$\mathcal H(\{\frac12\},\{1\})_\CCC\otimes\mathcal H(\{\frac1{2},\frac1{3},\frac2{3}\},\{1,1,1\})_\CCC$ \\
     &$ S:=\frac1{\sqrt 3}\begin{pmatrix}0&1\\{-3}&0\end{pmatrix}$  &$(\frac14, \frac34)$ &6 &$(1-t)^{-\frac14}\pFq21{\frac1{6}&\frac13}{&1}{t}$& $(1-t)^{-\frac12}\pFq32{\frac12&\frac1{3}&\frac23}{&1&1}{t}$   \\
     & $(ST)^{-1}$ &$ (\frac5{12},\, \frac{7}{12})$&&&\\
      \hline
$(6,\,2,\,4)$&  $s_6$ & $(\frac1{12},\, \frac{11}{12})$& &$\mathcal K_\CCC \otimes\mathcal H(\{\frac14\},\{1\})_\CCC\otimes\mathcal H(\{\frac1{24},\f7{24}\},\{1,\frac56\})_\CCC$&$\mathcal H(\{\frac12\},\{1\})_\CCC\otimes\mathcal H(\{\frac1{2},\frac1{4},\frac3{4}\},\{1,\frac56,\frac76\})_\CCC$\\
&$s_2$ & $(\frac14, \frac34)$&24& &  \\ 
&$s_4$ & $ (\frac38, \frac58)$&&$t^{-\frac1{12}}(1-t)^{-\frac 14}\pFq21{\frac1{24}&\frac{7}{24}}{&\frac56}{t}$&$(1-t)^{-\f12}\pFq32{\frac12&\frac1{4}&\frac3{4}}{&\frac56&\frac76}{t}$\\ \hline
\end{tabular}
    \caption{Groups in b) list  except (2,6,6)}\label{tab:2} where $(\mathcal K_\CCC)_{\bar t}=t^{-\frac1{12}}$ and $(\mathcal H(\{a\}, \{1\})_\CCC)_{\bar t} = (1-t)^{-a}$. 
\end{sidewaystable}

We take the representation of 
$B_6$ in $M_2(\RR)$ by matrices of size 2 over $\QQ (\sqrt{3})$ and choose 
\begin{equation}\label{eq:B6-embed}
    s_4=\frac 1{\sqrt 2}\begin{pmatrix} {-1}&{1}\\{-1}&{-1}\end{pmatrix}, \quad s_6=\frac 1{2\sqrt 3}\begin{pmatrix}{3+\sqrt 3}&{-3+\sqrt 3}\\{3+\sqrt 3}&{3-\sqrt 3}\end{pmatrix}, \quad s_2=(s_4s_6)^{-1}
\end{equation}
 with $s_4^4=s_6^6=s_2^2=-I_2$ as matrices. See \cite{Bayer-Travesa, Elkies} for more detail. They generate a group whose quotient mod $\pm I_2$ is isomorphic to $(2,4,6)$.

\bigskip

\subsection{Isomorphisms using rigidity}\label{ss: triangle gp and rigidity}  
For arithmetic triangle groups, the rigidity theorems  may be applied
to compare the automorphic sheaves described in \S\ref{S:autosh} and hypergeometric local systems arising from local solutions of hypergeometric differential equations. More details are discussed in \cite{Katz}, where $\ell$-adic hypergeometric sheaves on the algebraic group $\mathbb G_m$ are also introduced.

\begin{theorem}(Katz \cite[Rigidity Theorem 3.5.4]{Katz})\label{thm: Katz-Rigidity}
Let $\mathcal F$ and $\mathcal G$ be two irreducible local systems on $(\mathbb G_m-\{s_0\})^{\mbox{an}}$ of the same rank $n\ge 1$, {where $s_0$ is a fixed nonzero value,} and satisfying

\begin{itemize}
    \item[1)] The local monodromies at $s_0$ of both $\mathcal F$ and $\mathcal G$ are pseudoreflections;
    \item[2)]  At both 0 and $\infty$, $\mathcal F$ and $\mathcal G$ have the same characteristic polynomials of local monodromies respectively.  
\end{itemize}
Then $\mathcal F$ and $\mathcal G$ are isomorphic.
\end{theorem}

 For the groups in  a) list, the generator $T$ listed in Table \ref{tab:a-list} is a pseudoreflection. From Theorems \ref{thm:localexponent} and \ref{thm: Katz-Rigidity}, and the information in Table \ref{tab:a-list},  we have the following conclusion.
 \begin{theorem}\label{thm:alist}
     The local system $V^1(\G)_\CCC$ is isomorphic to  $\mathcal H^1(\G)_\CCC = \mathcal H(HD(\G))_\CCC$ 
 given in the last column of Table \ref{tab:a-list}. 
 \end{theorem}

We now illustrate how to use Katz's rigidity theorem to relate $V^{2}(\G)_\CCC$ 
and hypergeometric local system directly  for  b) groups (which also works for $(2,6,6)$).  For both local systems, $s_0=1$. 
\begin{theorem}
\label{thm:blist}
    For each group $\G$ in Table \ref{tab:2}, let $V^2(\G)_\CCC$ be the rank-3 sheaf as in \S \ref{SS:autosh1}.   \bk Then $V^2(\G)_\CCC$ is isomorphic to $ \mathcal L(\G)_\CCC\otimes \mathcal H(HD(\G))_\CCC {=\mathcal H^2(\G)_\CCC}$, where {$\mathcal H(HD(\G))_\CCC$ and $\mathcal H^2(\G)_\CCC$ are the hypergeometric sheaves  in Table \ref{tab:2}.  Moreover $V^2(\G)_\CCC$ is {isomorphic to} the symmetric square of $\mathcal H^1(\G)_\CCC$ listed in the second to the last column in Table \ref{tab:2}.} 
\end{theorem}

\begin{proof} 
By Clausen formula, $\rm{Sym}^2 \mathcal H^1(\G)_\CCC$ is isomorphic to $\mathcal H^2(\G)_\CCC$. It remains to prove the first statement.

The groups $\G=(2, m,\infty)$ in b.I) with $m=3$, $4$, $6$ are isomorphic to the projective groups PSL$_2(\ZZ)$, $\G_0^+(2)/\{\pm I_2\}$, and $\G_0^+(3)/\{\pm I_2\}$ respectively. They are generated by the matrices $$
  T:=\begin{pmatrix} 1 &1\\ 0 &1 \end{pmatrix}, \quad S:=\begin{pmatrix} 0 &1/\sqrt a\\  -\sqrt a&0 \end{pmatrix}, \quad  (ST)^{-1}, \text{ where }\, a= \lfloor \frac{m}{2}\rfloor. $$ 
Thus the fixed points of $(ST)^{-1}$ are 
$ \frac{-1+\sqrt{-3}}2$, $ \frac{-1+\sqrt{-1}}2$, $ \frac{-1+\sqrt{-1/3}}2$ respectively. Also 
$$\mr{Sym}^{2}T=\begin{pmatrix}
  1&2&1\\
  0&1&1\\
  0&0&1
  \end{pmatrix}, 
  \quad \mr{Sym}^{2}S=\begin{pmatrix}
  0&0&1/a\\
  0&-1&0\\
  a&0&0
  \end{pmatrix},\quad \mr{Sym}^{2}(ST)^{-1}=\begin{pmatrix}
  a&2&1/a\\
   -a&-1&0\\
a&0&0
  \end{pmatrix}.$$

We choose a Hauptmodul $\l$ for $\G$ (such as $\l=1728/j$ for PSL$_2(\ZZ)$) whose values at the singular points are as described in Theorem \ref{thm:traceformula}. Precisely, it has a simple zero at infinity, takes value 1 at  
the fixed point of $S$, and has a simple pole at that of $ST$. The sheaf $V^2(\G)_\CCC$ is defined in \S \ref{SS:autosh1}. While $V^1(\G)_\CCC$ is not defined because the lift of $\G$ in SL$_2(\RR)$ contains $-I_2$, nonetheless the above three symmetric square matrices can be used to describe  $V^2(\G)_\CCC$. Indeed 
the Jordan normal forms of the local monodromy matrices of $V^2(\G)_\CCC\otimes\mathcal L(\G)_\CCC$ are

$$
  \begin{array}{ccc}
     1& 0&  \infty\\ \hline
      \begin{pmatrix}
  1&&\\
  &1&\\
  &&-1
  \end{pmatrix}  &  
  \begin{pmatrix}
  1&1&0\\
  0&1&1\\
  0&0&1
  \end{pmatrix}& \begin{pmatrix}
  -1&&\\
  &-\zeta_m&\\
  &&-\zeta_m^{-1}
  \end{pmatrix}
  \end{array}
$$In particular, the monodromy at 1 is a pseudoreflection. The monodromy matrices at $1, 0, \infty$ have characteristic polynomials $(T-1)^2(T+1), (T-1)^3,  (T+1)(T+\zeta_m)(T+\zeta_m^{-1})$ respectively,  which coincide with 
those of the hypergeometric local system {$\mathcal H(HD(\G))_\CCC$} with datum $HD(\G)$ (see Theorem \ref{thm:localexponent}). We conclude from Theorem   \ref{thm: Katz-Rigidity} that they are isomorphic. 

The same argument applies to the fourth group $\G = (2,\infty,\infty)$ in b.I), which is isomorphic to $\G_0(2)/\{\pm I_2\}$. The only difference is the generator $S = \begin{pmatrix}-1 & 1\\-2 & 1\end{pmatrix}$ so that $\mr{Sym}^{2}S$ has eigenvalues $-1, 1, -1$ and $\mr{Sym}^2(ST)^{-1}$ has eigenvalues $1,1,1$.

The remaining case can be verified in the like manner. From the exponents of the generators for (2,4,6) we know the  exponents of their symmetric squares are $\{\frac14,\, 1,\, \frac34\}$, $\{\frac76,\,1,\,\frac56\}$, and $\{\f12,\,\f12,1\}$ respectively, which coincide with the local exponents of the order-3 differential equation satisfied by  $(1-t)^{-\f12}\pFq32{\frac12&\frac1{4}&\frac3{4}}{&\frac76&\frac56}{t}.$  
\end{proof}

\subsection{Comparison Theorem}
\label{SS:compare}
The results of this section appear in \cite{Hoffman-Tu}. We include them here for the 
convenience of the reader.
Let $F$ be a number field and $R = O_F[1/N]$ the localization of the ring of 
integers of $F$ for an integer $N\ge1$. Let $S = \mr{Spec}(R)$. We let 
$\eta = \mr{Spec}(F)$, the generic point of $S$, and $\bar{\eta} = \mr{Spec}(\bar{F})$ for an 
algebraic closure of $F$. Let $X/S$ be an irreducible separated scheme, smooth and of finite type over $S$, 
with geometrically connected fibers. 
We can choose a geometric generic point $\bar{\xi} : \mr{Spec} (\overline{F(X)}) \to X$ which 
lies over $\bar{\eta}$, where $F(X) $ is the function field of $X$. We let $X_{\eta}$  and $X_{\bar{\eta}}$  
be the schemes over $\mr{Spec}{F}$ and $\mr{Spec}{\bar{F}}$ obtained from $X$ by base-change. 

We consider two geometrically irreducible lisse $\bar{\QQ}_{\ell}$-
sheaves (for the \'etale topology) $\mathscr{F}$, $\mathscr{G}$ on $X$. These are equivalent to 
two $\ell$-adic representations 
\[
\rho_{\mathscr{F}}: \pi _1(X, \bar{\xi}) \longrightarrow
\mr{GL}(V), \quad
\rho_{\mathscr{G}}: \pi _1(X, \bar{\xi}) \longrightarrow
\mr{GL}(W)
\]
for finite-dimensional $\bar{\QQ}_{\ell}$-vector spaces $V$, $W$. Geometrically irreducible means:
the restrictions 
\[
\rho_{\mathscr{F}}\mid X_{\bar{\eta}}, \ \ \rho_{\mathscr{G}}\mid X_{\bar{\eta}}
\]
of $ \pi _1(X_{\bar{\eta}}, \bar{\xi}) $ are irreducible. Note that we have 
 a surjective homomorphism
\[
 \pi _1(X_{\eta}, \bar{\xi}) \to  \pi _1(X, \bar{\xi})
\]
 and an exact sequence
 $$ 
   0 \longrightarrow \pi _1(X_{\bar{\eta}}, \bar{\xi}) \longrightarrow \pi _1(X_{\eta}, \bar{\xi}) \longrightarrow \mr{Gal}(\bar{F}/F) \longrightarrow 0.
 $$
 
Now let $\varphi :R \to \CCC$ be an embedding. We obtain a scheme
$X_{\varphi, \CCC}$ over $\CCC$. This also defines an analytic space
$X_{\varphi}^{\mr{an}} := X_{\varphi}(\CCC) $. Since $\varphi$ will be fixed, 
we will drop it from the notation. There is a canonical map
\[ 
\pi_1(X^{\mr{an}}, u) \to \pi_1(X_{\CCC}, u)
\]
(left-hand side: topological fundamental group; right-hand side, the \'etale fundamental group)
which identifies the right-hand side with the profinite completion of left-hand side (Riemann's existence
theorem). In particular, this map has dense image. Here we can take
$u$ to be the geometric $u: \mr{Spec}(\CCC) \to X$ 
\[
\mr{Spec}(\CCC) \to \mr{Spec}(\overline{F(X)}) = \bar{\xi} \to  X
\]
where the first arrow is induced by some embedding $\bar{\varphi}: \overline{F(X)} \to \CCC$ which extends $\varphi$.
It is known that there are isomorphisms $\pi_1(X_{\CCC}, u) =\pi_1(X_{\bar{\eta}}, u) =\pi_1(X_{\bar{\eta}}, \bar{\xi})$ induced
by $\bar{\varphi}$. The first holds because $F$ has characteristic 0; the second is a change in base-point.
Choose an isomorphism $\iota : \bar{\QQ}_{\ell} \cong \CCC$. The theorem that follows will not depend 
on this artificial choice.

Composing all these, we get representations 
\[
\rho ^{\mr{an}}_{\mathscr{F}} : \pi_1(X^{\mr{an}}, u) \to \pi_1(X_{\CCC}, u) =  \pi_1(X_{\bar{\eta}}, \bar{\xi})
\overset {\rho _{\mathscr{F}}}{\longrightarrow} \mr{GL}(V) \cong \mr{GL}(V_{\CCC})
\]
where $V_{\CCC} = V \otimes _{\bar{\QQ}_{\ell}, \iota}\CCC$. We get a similar story for  $\rho ^{\mr{an}}_{\mathscr{G}}$.
We let ${\sf F} $ and ${\sf G}$ be the $\CCC$-local systems on $X^{\mr{an}}$ that arise from these 
representations of the fundamental group. Also $\mathcal{D} ({\sf F})$ and  $\mathcal{D} ({\sf G})$ 
the regular holonomic $\mathcal{D}$-modules (=connections with regular singular points) corresponding 
to these by the Riemann-Hilbert correspondence. 

\begin{theorem}
\label{thm:compare}
Under these assumptions (and $\mathscr{F}$, $\mathscr{G}$ geometrically irreducible), if 
the local systems ${\sf F} $ and ${\sf G}$ on  $X^{\mr{an}}$ are isomorphic (equivalently
if the $\mathcal{D}$-modules $\mathcal{D} ({\sf F})$ and  $\mathcal{D} ({\sf G})$  are isomorphic), 
then there is a continuous character $\chi : \mr{Gal}(\bar{F}/F)\to \bar{\QQ}_{\ell}^{\times}$, such that 
$\mathscr{G}_{\eta} \cong \mathscr{F}_{\eta} \otimes \chi.$
\end{theorem}
\begin{proof}
There is a matrix $M : V_{\CCC}\to W _{\CCC} $ that intertwines 
the representation of $\pi_1(X^{\mr{an}}, u)$ given by ${\sf F} $ and ${\sf G}$. Then
$\iota ^{-1}(M) : V \to W$ is a matrix that intertwines the representations of  $\pi_1(X^{\mr{an}}, u)$ in 
$\mr{GL}(V)$ and $\mr{GL}(W)$. 
But $\pi_1(X^{\mr{an}}, u)$ has dense image in  $\pi_1(X_{\bar{\eta}}, \bar{\xi})$ and the 
representations given by $\mathscr{G}$ and $\mathscr{F}$ on $V$, $W$ are continuous. Thus
by continuity $\iota ^{-1}(M) $ will intertwine those representations. Therefore the representations
\[
\rho_{\mathscr{F}}\mid X_{\bar{\eta}}, \ \ \rho_{\mathscr{G}}\mid X_{\bar{\eta}}
\]
of $ \pi _1(X_{\bar{\eta}}, \bar{\eta}) $ are isomorphic. From the exact sequence above, and 
the fact that these representations are isomorphic we get, by the lemma below, a character 
$\chi : \mr{Gal}(\bar{F}/F)\to \bar{\QQ}_{\ell}^{\times}$ and an isomorphism
\[
(\rho_{\mathscr{G}}\mid X_{\eta}) = (\rho_{\mathscr{F}}\mid X_{\eta})\otimes \chi
\]
as representations of $ \pi _1(X_{\eta}, \bar{\xi}) $. 

 \end{proof}

The following is well-known. 

\begin{lemma}
\label{L:ext}
Given an exact sequence of groups 
\[
 \begin{diagram}
 \node{0} \arrow{e}{}
\node{H} \arrow{e,t}{a} \node{G} \arrow{e,t}{b} \node{G/H} \arrow{e,t}{}
\node{0}
 \end{diagram}
 \]
 and two finite-dimensional representations $\rho : G \to \mr{GL}(V)$ and
$\sigma : G \to \mr{GL}(W)$ where $V, W$ are vector spaces over an
algebraically closed field $k$.
Suppose that $\rho \mid H$ and $\sigma \mid H$ are irreducible and isomorphic. 
Then there is a character $\chi : G/H \to \mr{GL}_1(k) = k ^{\times}$, such that 
$\sigma \cong \rho \otimes \chi := \rho \otimes (\chi\circ b)  $. If $k$ is a topological field and $\rho, \sigma $ are continuous 
representations, then $\chi $  is a continuous character. 
\end{lemma}
\begin{proof}
Let $Z = \mr{Hom}_k (V, W)$. This is a finite-dimensional $k$-space with a $G$-action:
\[ 
(g\cdot f)(v) := \sigma(g)f(\rho(g)^{-1} v),  \text{ where } v\in V,f\in Z,g\in G. 
\]
A fixed vector $f$, namely $g\cdot f = f, \forall g \in G$, is a $G$-equivariant map $f : V \to W$. By hypothesis
there is an $H$-fixed vector corresponding to an $H$-equivariant isomorphism 
$f : V \overset {\sim}{\longrightarrow}W$. Since these $H$-spaces are irreducible, any nonzero
$H$-equivariant map $f:V \to W$ is an isomorphsm, and any two such differ by a scalar 
multiple (Schur's lemma). Thus, $Z^H = k\cdot f$. The $G$-action on $Z$ induces
a $G/H$-action on $Z^H$. Since this is one-dimensional, this is a character $\chi$ which is defined 
by $g\cdot f = \chi(g) f$ for all $g\in G$. One checks that  $f : V \to W$ is an isomorphism
which is equivariant for the $\rho \otimes \chi$ action  on $V$ and $\sigma $ on $W$.
\end{proof}

We can give a stronger version  of  Theorem \ref{thm:compare} if we assume in addition that $X/S$ has a section, and that 
$X/S$ is the complement in $Z/S$ of a divisor with normal crossings $D/S$, where $Z/S$ is proper
and smooth. We also assume that $\ell$ is invertible on $X$ and $S$. Under those assumptions, then we have an exact sequence
 (\cite[Ch. XIII, Prop. 4.3, and Examples 4.4]{SGA1})
 \[
 \begin{diagram}
\node{0} \arrow{e}{}\node{  \pi _1 ^{\mathbb{L}}(X_{\bar{\eta}}, \bar{\xi})   } \arrow{e}{} 
  \node{   \pi _1 '( X, \bar{\xi})    } \arrow{e} \node{  \pi _1 (S, \bar{\eta})   }
  \arrow{e} \node{0}
 \end{diagram}.
 \]

Here $\mathbb{L}$ is a 
set of primes invertible on $S$, namely coprime to $N$. $\pi _1 ^{\mathbb{L}}(X_{\bar{\eta}}, \bar{\xi})$
is the pro-$\mathbb{L}$-quotient of $\pi _1 (X_{\bar{\eta}}, \bar{\xi})$. If $K$ is the kernel of the canonical homomorphism
$\pi _1 ( X, \bar{\xi}) \to \pi _1 (S, \bar{\eta})$ and $N\subset K$ is the smallest normal subgroup such that 
$K/N$ is a pro-$\mathbb{L}$-group, then $N \subset \pi _1 ( X, \bar{\xi}) $ is a normal subgroup, and 
we denote $\pi _1 '( X, \bar{\xi})  = \pi _1 '( X, \bar{\xi}) /N$. If we assume that the representations
$\rho_{\mathscr{F}}$ and $\rho_{\mathscr{G}}$ factor through $\pi _1 '( X, \bar{\xi}) $ and are 
geometrically irreducible, we can conclude that there exists a continuous character 
$\chi: \pi _1 (S, \bar{\eta})\to \bar{\QQ}^{\times}_{\ell}$ such that  $\mathscr{G} \cong \mathscr{F} \otimes \chi$
on $X$.

\subsection{The $\ell$-adic sheaves isomorphism}\label{ss:3.5}
The main application of this theorem in our paper is to the situation where $X$ is a (modular) curve of genus 0 defined over $\QQ$ 
and the two local systems are an automorphic sheaf and a hypergeometric sheaf respectively. The isomorphism of local systems 
on the analytic space $X^{\mr{an}}$ results from rigidity, {as shown in \S \ref{ss: triangle gp and rigidity}}. Because our local systems arise from geometry, they have special properties. For instance, they belong to a strictly compatible family of rational $\ell$-adic representations. 

To determine the twisting character $\chi$ in the statement of the Comparison Theorem \ref{thm:compare} it suffices 
to restrict the sheaves to a $\QQ$-rational point of the curve $X^\circ$. Then  $\mathscr{F}$ and
$\mathscr{G}$ define representations of $G_\QQ :=\mr{Gal}(\bar{\QQ}/\QQ)$. The twisting character $\chi$ in the statement of the Comparison Theorem \ref{thm:compare} is 
 an algebraic Hecke character of the number field $\QQ$. The reader is referred to Serre \cite{SerreAbelian} for more detail. 

In particular, we recall  Corollary 2.1.6 of \cite{ALLL}. 
\begin{lemma}\label{lem:chi-finite}
    If $\chi$ 
    is a 1-dimensional representation of $G_\QQ$ taking values in the group of units of a finite extension of $\QQ_\ell$, 
    which is unramified almost everywhere and whose restriction to the decomposition group at $\ell$ is crystalline, then $\chi=\epsilon_\ell^m \xi$, where $\epsilon_\ell$ is the $\ell$-adic cyclotomic character of $G_\QQ$, $m \in \ZZ$ and $\xi$ is a character of $G_\QQ$ of finite order.
\end{lemma} 

As we have seen in \S\ref{S:autosh}, the automorphic sheaves $V^k(\G)_\CCC$ have $\ell$-adic counterparts $V^k(\G)_\ell$. For the groups $\G$ we consider, we have introduced the hypergeometric sheaves $\mathcal H(HD(\G))_\CCC$ attached to the hypergeometric data $HD(\G)$ in Tables \ref{tab:a-list} and \ref{tab:2}. To these $HD(\G)$, Katz has associated the $\ell$-adic sheaves $\mathcal H(HD(\G))_\ell$, whose behavior at $0, 1, \infty$ agrees with that of $\mathcal H(HD(\G))_\CCC$ under the change of generator $\lambda \mapsto 1/\lambda$ (see \cite{Katz}). Theorems \ref{thm:alist} and \ref{thm:blist} tell us how the two complex sheaves of different origin are related. Using the Comparison Theorem \ref{thm:compare}, we know that the same relation holds for their $\ell$-adic counterparts up to twist by a character $\chi_\G$. In the next section we will give more details of the $\ell$-adic hypergeometric sheaves and determine the character  $\chi_\G$.

\section{Hypergeometric character sums}\label{ss:char}
Given a hypergeometric datum $HD = \{\alpha, \beta\}$ of length $n$, in \S\ref{ss:groups} we associated the hypergeometric function $F(HD;t)$ over $\CCC$ 
and a rank-$n$ sheaf $\mathcal H(HD)_\CCC$ on ${\mathbb P}_\CCC^1\smallsetminus \{0,1,\infty\}$ 
along with the monodromy representation of the fundamental group $\pi_1({\mathbb P}_\CCC^1 \smallsetminus\{0,1,\infty\}, *)$ with a base point. When the elements in $\alpha, \beta$ are in $\QQ^\times$ and $HD$ is primitive, for each prime $\ell$, Katz in \cite{Katz,Katz09} has constructed an $\ell$-adic sheaf $\mathcal H(HD)_\ell$ of rank-$n$ on $\mathbb G_m$ together with the action of the absolute Galois group $G(M)=Gal(\bar{\QQ}/\QQ(\zeta_M))$ of the cyclotomic field $\QQ(\zeta_M)$. Here $M = M(HD) = lcd(\alpha \cup \beta)$, the least positive common denominator of elements in $\alpha$ and $\beta$, is the level of $HD$. {After changing the generator $\lambda$ by $1/\lambda$,} this is the $\ell$-adic counterpart of the complex hypergeometric sheaf $\mathcal H(HD)_\CCC$. The Galois action on each stalk of $\mathcal H(HD)_\ell$ is described by its Frobenius traces which are given by hypergeometric character sums expressed either in terms of the period functions ${\mathbb P}(HD)$ over finite fields introduced in \cite{Win3X} or the $H_q(HD)$-functions introduced by Beukers-Cohen-Mellit in \cite{BCM} based on McCarthy's work \cite{McCarthy}. These will be recalled below.

\subsection{Hypergeometric functions over finite fields}\label{ss:HG-FF}
Let $\F_q$ be a finite field of odd characteristic and use $ \widehat{\F_q^\times}$ to denote the group of multiplicative characters of $\F_q^\times$.   
We use $\eps_q$ or simply $\eps$ to denote the trivial character, and  $\phi_q$ or  $\phi$ to denote the quadratic character. For any $A\in \widehat{\F_q^\times}$, use $\overline A$ to denote its inverse, and extend $A$ to $\F_q$ by setting 
$A(0) = 0$.  For any characters  $A_i$, $B_i \in \widehat{\F_q^\times}$, $i=1,\cdots,n$ with $B_1=\eps$ in $\widehat{\F_q^\times}$, the $_{n}\mathbb P_{n-1}$-function (cf. \cite{Win3X}) is defined as follows: 

\begin{multline}\label{eq:PP}
  \pPPq{n}{n-1}{A_1& A_2&\cdots &A_{n}}{&  B_2&\cdots &B_{n}}{\lambda;q}\\
    : = \prod_{i=2}^{n} \left(-A_iB_i(-1)\right)\cdot  \left (\frac{1}{q-1}
      \sum_{\chi\in \widehat{\F_q^\times}}\CC{A_1\chi}{\chi} \CC{A_2\chi}{B_2\chi}\cdots \CC{A_{n}\chi}{B_{n}\chi}\chi(\l)
      +\delta(\lambda) \prod_{i=2}^{n} \CC{A_{i}}{B_i}\right),
\end{multline} 
where
$$   \displaystyle\CC AB :=-B(-1)J(A,\ol B), \quad J(A, B):=\sum_{t\in \F_q}A(t)B(1-t), \quad   \mr{ and } \quad \delta(\l):=
\begin{cases}
  1,& \mbox{ if } \l=0, \\
  0,& \mbox{ otherwise. } \end{cases}
$$

Let $\alpha=\{a_1,\cdots,a_n\},$ and $\beta=\{1,b_2,\cdots,b_n\}$ be multi-sets with entries in $\QQ^\times$. Denote  the positive least  common denominator of all $a_i,b_j$ by $M:=\mr{lcd}(\alpha\cup \beta)$.
For a finite field $\F_q$ containing a primitive $M$th root of $1$ 
and any $\l\in \F_q$, let $\omega$ be a generator of $\widehat{\F_q^\times}$.   Following \cite{LLT2}, we write \begin{equation}\label{eq:P}
    \mathbb P(\alpha,\beta;\l;\F_q; \omega):=\pPPq{n}{n-1}{\omega^{(q-1)a_1}&\omega^{(q-1)a_2}&\cdots&\omega^{(q-1)a_n}}{&\omega^{(q-1)b_2}&\cdots&\omega^{(q-1)b_n}}{\l;\F_q}.
\end{equation} 

Write $G_K =\mr{Gal}(\overline \QQ/K)$ for the absolute Galois group of a number field $K$. When $K=\QQ(\zeta_M)$ where $\zeta_M$ is a primitive $M$th root of unity, we  abbreviate $G_{\QQ(\zeta_M)}$ by $G(M)$. Note that the Frobenius conjugacy classes $\Frob_\wp$ over the degree-1 prime ideals $\wp$ of $\QQ(\zeta_M)$ are dense in $G(M)$, hence irreducible Galois representations of $G(M)$ are uniquely determined by their Frobenius traces on almost all degree-1 prime ideals $\wp$, which are those above the primes $p \equiv 1 \mod M$. Denote by $\chi_d$ the quadratic character of $G_\QQ$ corresponding to the quadratic field $\QQ(\sqrt{d})$.

The following frequently used fact concerning extensions of a Galois representation results from the Frobenius reciprocity law.

\begin{theorem}\label{thm:Galois-ext'n}
Let $K$ be a finite Galois extension of $\QQ$ and $\sigma$ be a finite dimensional irreducible $\ell$-adic representation  of $G_K$. 
If for every $g\in G_\QQ$, the conjugate $\sigma^g\cong \sigma$ as $G_K$-modules, then $\sigma$ can be extended to a representation of $G_\QQ$, which is  unique up to a linear character of $G_\QQ$ trivial on $G_K$.
\end{theorem}

Now we describe the stalks of an $\ell$-adic hypergeometric sheaf. 

\begin{theorem}[Katz \cite{Katz, Katz09}]\label{thm:Katz}Let  $\ell$ be a prime. Given a primitive hypergeometric datum HD = $\{\alpha, \beta\}$ consisting  of $\alpha=\{a_1,\cdots,a_n\}$, $\beta=\{1,b_2,\cdots,b_n\}$ with $a_i, b_j \in \QQ^\times$ and $M:=M(HD) = \mr{lcd}(\alpha \cup \beta)$, 
for each $\l \in \ZZ[\zeta_M,1/M]\smallsetminus \{0\}$ the following hold for the $\ell$-adic representation $\rho_{HD,\l,\ell}$ of $G(M)$ on the stalk at $\l$ of the hypergeometric sheaf $\mathcal H(HD)_\ell$:  
\begin{itemize}
\item [i).] 
$\rho_{HD,\l,\ell}$ is unramified 
at 
the prime ideals $\wp$ of  $ \ZZ[\zeta_M,1/(M\ell \l)]$ such that its trace at the conjugacy class $\Frob_\wp$ of the geometric Frobenius  at $\wp$ in $G(M)$ with residue field $\kappa_\wp$ and norm $N(\wp)=|\kappa_\wp|$ is  
\begin{equation}\label{eq:Tr1} \Tr \rho_{HD,\l,\ell}(\mr{Frob}_\wp)= (-1)^{n-1}{\omega_{\wp}^{(N(\wp)-1)a_1}(-1)}
{\mathbb P}(HD; 1/\l;\kappa_\wp; {\omega_\wp}),  
\end{equation} 
where 
$\omega_{\wp}$ is a generator of $\widehat {\k_\wp^\times}$ satisfying $\omega_{\wp}(\zeta_M\mod \wp) = \zeta_M^{-1}$.

\item[ii).] When $\l\neq 1$,  the stalk $(\mathcal H(HD)_\ell)_{\bar \l}$ has dimension 
$n$ and all roots of the characteristic polynomial of $\rho_{HD,\l,\ell}(\Frob_\wp)$  are algebraic numbers and have the same absolute value $N(\wp)^{(n-1)/2}$ under all archimedean embeddings.

\item[iii).] When $\l=1$, the dimension of $(\mathcal H(HD)_\ell)_{\bar \l}$ equals $n-1$. 
\end{itemize}
\end{theorem}
\begin{remark}\label{rem:phi(M,a)}
Recall that for $a = r/M$ with $r \in \ZZ$, $\omega_\wp^{(N(\wp)-1)a}$ is the character $\iota_\wp(a) = \left((\frac{\cdot}{\wp})_M\right)^r$, the $r$th power of the $M$th norm residue symbol at a nonzero unramified prime ideal $\wp$ as in Definition 5.9 of \cite{Win3X}. In particular, when $\wp$ varies, the sign $\omega_{\wp}^{(N(\wp)-1)a}(-1) {= \iota_\wp(a)(-1)}$ defines a character of $G(M)$, denoted by $\phi(M, a)$. 
As the character $\phi(M,a)$ is invariant under conjugation by $G_\QQ/G(M)$, it  
extends to a character of $G_\QQ$ and our convention here is use the extension with minimal conductor. {It is nontrivial only when} ${\rm ord}_2~M = -{\rm ord}_2 ~a = r\ge 1$. 
Especially, when ${\rm ord}_2~M = -{\rm ord}_2~a = 1$, $\phi(M,a)(\Frob_p)=\left(\frac{-1}{p}\right)$ is given by the Legendre symbol at odd primes $p$. 
\end{remark}

\subsection{Beukers-Cohen-Mellit's finite Hypergeometric functions}\label{ss:HG-BCM}
For $q\equiv 1\mod M$, let 
$$
H_q(\alpha,\beta;\l;\omega):=\frac {1 }{1-q} \sum_{k=0}^{q-2} \omega^k((-1)^n\l)
  \prod_{j=1}^n \frac{\g(\omega^{(q-1)a_j+ k})}{\g(\omega^{(q-1)a_j})}
  \frac{\g(\ol\omega^{(q-1)b_j+ k})}{\g(\ol\omega^{(q-1)b_j})}, 
		$$  where $\g(\chi):=\sum_{x\in\F_q}\Psi(x)\chi(x)$ denotes the Gauss sum of $\chi$ with respect to a fixed choice of non-trivial additive character $\Psi$ of $\F_q$. However the above sum is independent of the choice of $\Psi$ since $|\alpha|=|\beta|$.
The function $H_q$ is 
equivalent to the normalized ${\mathbb P}$-function (denoted by $\mathbb F$-function in \cite{Win3X}) 
when the pair $\alpha$ and $\beta$ is primitive, in which case they are related by 
\begin{equation}\label{eq:P-H}
  {\mathbb P} (\alpha,\beta;\l;\F_q;\omega) =H_q(\alpha,\beta;\l;\omega) \cdot \prod_{i=2}^n\omega^{(q-1)a_i}(-1)J(\omega^{(q-1)a_i},\ol\omega^{(q-1)b_i}).
 \end{equation} By Weil \cite{Weil52}, as $\F_q$ runs through the residue fields of almost all prime ideals of $\QQ(\zeta_M)$, these Jacobi sums define a 
 Gr\"ossencharacter of $G(M)$. The above relation indicates that instead of using the ${\mathbb P}$-functions to describe the trace function of the Galois representations in Theorem \ref{thm:Katz}, one can also use the $H$-functions.  

When $HD = \{\alpha, \beta\}$ and $\l$ are both defined over $\QQ$, 
one can rewrite $H_q(\alpha,\beta;\l;\omega)$ by
using the reflection and multiplication formulas of Gauss sums  \cite[Theorem 1.3]{BCM}:
$$
 \g(A)\g(\ol{A})=A(-1)q,\;\; \forall A\neq \eps, \quad 
	\prod_{\substack{\eta \in \widehat{\F_q^\times} \\ \eta^m=\eps}}\frac{\g(A\eta)}{ \g(\eta)}=-\g(A^m) A(m^{-m}),  \text{ when } m\mid q-1.
	$$ 
Firstly  write $\displaystyle \prod_{j=1}^n \frac{X-e^{2\pi i a_j}}{X-e^{2\pi i b_j}}=\frac{\prod_{j=1}^r(X^{p_j}-1)}{\prod_{k=1}^s(X^{q_k}-1)}$ where $p_j,q_k\in \ZZ_{>0}$ and $p_j \ne q_k$ for all $j, k$. Then 
\begin{equation}\label{eq:hq}
H_q(\alpha,\beta;\l;\omega)=\frac{(-1)^{r+s}}{1-q}\sum_{m=0}^{q-2}q^{-s(0)+s(m)}\prod_{j=1}^r \mathfrak{g}(\omega^{mp_j}) \prod_{k=1}^s \mathfrak{g}(\omega^{-mq_k})\omega(N^{-1} \l),
\end{equation} where
 $s(m)$ is the multiplicity of $X-e^{\frac{2\pi i m}{q-1}}$ in 
$gcd(\prod_{j=1}^r(X^{p_j}-1),\prod_{k=1}^s(X^{q_k}-1))$, and $$N=(-1)^{ \sum_{k=1}^s q_k}\cdot \frac{p_1^{p_1}\cdots p_r^{p_r}}{q_1^{q_1}\cdots q_s^{q_s}}.$$ In this case $H_q$ is $\QQ$-valued, and hence is independent of the choice of $\omega$,  which will be dropped from the left hand side of \eqref{eq:hq}.  Also \eqref{eq:hq} holds  now for 
$\F_q$ as long as $q$ is coprime to $lcd(\{1/M,\l\})$.

 The next theorem follows from Theorems \ref{thm:Katz} and \ref{thm:Galois-ext'n}. 

\begin{theorem}[Katz, Beukers-Cohen-Mellit]\label{thm:KBCM} Let  $HD =\{\alpha=\{a_1,\cdots, a_n\}, \beta=\{1,b_2,\cdots, b_n\}\}$ be a primitive hypergeometric datum  defined over $\QQ$ with level $M=M(HD)$. Let 
$ \l\in \ZZ[1/M]\smallsetminus \{0\}$.   Assume that exactly $m$ elements in $\beta$ are in $\ZZ$. 
Then, for each prime $\ell$, there exists an $\ell$-adic  representation $\rho_{HD,\l,\ell}^{BCM}$ of $G_\QQ$,  unramified almost everywhere,  with the following properties:
\begin{itemize}
\item[i).] $\rho_{HD,\l,\ell}^{BCM}|_{G(M)}\cong \rho_{HD,\l,\ell}$.
\item[ii).] For any prime  $p\nmid \ell M$ such that  $\mr{ord}_p \l = 0$, 
\begin{equation}\label{eq:BCMtrace}
\Tr\, \rho_{HD,\l,\ell}^{BCM}(\Frob_p)= \phi(M,\sum_{i=1}^n a_i)(\Frob_p) 
H_p(HD;1/\l)\cdot p^{(n-m)/2}\in \ZZ,
\end{equation} 
 where 
$\phi(M,a)$, as explained in Remark \ref{rem:phi(M,a)}, is a character of $G_\QQ$ of order at most 2.

\item[iii).]  When $\l=1$,   $\rho_{HD,\l,\ell}^{BCM}$ is $(n-1)$-dimensional and it has a subrepresentation, denoted by $\rho_{HD,\l,\ell}^{BCM,prim}$, of dimension $2\lfloor \frac {n-1}2 \rfloor$. 
 For each unramified prime $p$,  all roots of the characteristic polynomial of $\rho_{HD,\l,\ell}^{BCM,prim}(\Frob_p)$ have absolute value $p^{(n-1)/2}$. 
\end{itemize}
\end{theorem}

Under the assumptions of Theorem \ref{thm:KBCM}, the relation \eqref{eq:P-H} at a prime $\wp$ of $\ZZ[\zeta_M]$ above a prime $p \equiv 1 \mod M$ simplifies as $\mathbb P(HD;\l;\kappa_\wp;\omega_\wp) =H_p(HD;\l;\omega_\wp)p^{\frac{n-m}{2}}\phi(M, \sum_{i=2}^n a_i)(-1)^{n-1}$. This fact is incorporated in the expression of \eqref{eq:BCMtrace}. 

Under the above assumption, we now  clarify the relation between $H_p(\alpha,\beta;1/\l)$ and $H_q(\alpha,\beta;1/\l)$ where $q=p^r$. Let $\phi(M,\sum_{i=1}^n a_i)(\Frob_p) p^{(n-m)/2} u_i, 1 \le i \le d,$ be the roots of the characteristic polynomial of $\rho_{HD,\l,\ell}^{BCM}(\Frob_p)$ so that $H_p(\alpha,\beta;1/\l) = \sum_{i=1}^d u_i$. Here $d = n$ for $\l \neq 1$, and $d=n-1$ for $\l=1$. Then $\Tr\, \rho_{HD,\l,\ell}^{BCM}((\Frob_p)^r)= \phi(M,\sum_{i=1}^n a_i)((\Frob_p)^r) q^{(n-m)/2}=\sum_{i=1}^d u_i^r$. On the other hand, let $K$ be a finite Galois extension of $\Q$ unramified at $p$ such that any prime ideal $\wp$ of $K$ above $p$ has residue field $\F_q$. Then 
 $$\Tr\, \left(\rho_{HD,\l,\ell}^{BCM}\right)|_{G_K}(\Frob_\wp)= \phi(M,\sum_{i=1}^n a_i)(\Frob_\wp) 
H_q(\alpha,\beta;1/\l)\cdot q^{(n-m)/2}.$$ 
Since $\Frob_\wp = (\Frob_p)^r$, we conclude that $H_{p^r}(\alpha,\beta;1/\l)= \sum_{i=1}^d u_i^r.$

\begin{remark}\label{cor:1}
When $HD$ and $\l$ are both defined over $\QQ$, $\rho_{HD,\l,\ell}^{BCM}$ is an extension of $\rho_{HD,\l,\ell}$ from $G(M)$ to $G_\QQ$, which is unique up to a character of $G_\QQ/G(M)\cong {\rm Gal}(\QQ(\zeta_M)/\QQ)$ when $\rho_{HD,\l,\ell}$ is an irreducible $G(M)$-module. Thus on each $G(M)$-coset in $G_\QQ$, $\rho_{HD,\l,\ell}^{BCM}$ is determined by the action at one Frobenius element belonging to that coset. This property is used to extend identities involving the ${\mathbb P}$-functions to identities involving the $H_p$-functions. 
\end{remark}

\subsection{For groups in a) list}\label{ss:EC-a}
We now relate the groups of the form $\G=(\infty,\infty,e_\infty)$ in Figure \ref{fig:class-I} and Table \ref{tab:a-list} to elliptic curves. In the table below $t$ is the Hauptmodul of $\G$ the same as $\l$ in Theorem 3, and $E_{\G, t}$ is a family of elliptic curves over points in $X_\G$ with $t \ne 0, 1, \infty$.

\begin{center}
\begin{table}[h]
    $$
\begin{array}{|c|c|c|c|c|c|}\hline
e_\infty  &  E_{\G,t}& HD(\G)& j_\G(t)&c_4(\mathcal E_{j_\G(t))})/ c_4(E_{\G,t})\\
\hline
    \infty   & \displaystyle   y^2+xy+\frac t{16}y=x^3+\frac t{16}x^2& \{ \{\f12,\f12 \},\{1,1\}\}& \displaystyle \frac{16(t^2-16t+16)^3}{t^4(1-t)} &\left(\frac{4(t^2-16t+16)}{(t-2)(t^2+32t-32)}\right)^2 \\ \hline
   3 & \displaystyle y^2+xy+\frac{t}{27}y=x^3 & \{ \{\frac13,\frac23 \},\{1,1\}\}& \displaystyle \frac{27(8t-9)^3}{t^3(t-1)}&\left(-\frac{3(8t - 9)}{(8t^2 - 36t + 27)}\right)^2  \\ \hline
\end{array}
$$
    \caption{Elliptic curves for groups in  a) list}
    \label{tab:EC-a}
\end{table}
\end{center}

The universal elliptic curve has been recalled  in  \eqref{eq:Universal-E}. Both $\mathcal E_{j_\G(t)}$ and $E_{\G,t}$ have the same $j$-invariant and their $c_4$-constants quotients are squares. 
Thus  $E_{\G,t}$
is isomorphic to the pullback of the universal family $\mathcal E_j$ 
via the covering map $j_\G(t)$.

Based on \cite{Koike-ellipticcurves} by Koike  and  \cite{Ono} by Ono (who used a model isogenous to the one above over $\QQ(t)$) for the case $e_\infty=\infty$ and  \cite{BCM} by Beukers-Cohen-Mellit  for  the case $e_\infty=3$, for  $t\in\QQ\setminus\{0,1\}$
and any prime  $p>3$ not dividing the level of $E_{\G,t}$, we have 
\begin{equation}\label{eq:EC-count-byH}
     \#E_{\G,t}(\F_p)=p+1-H_p(HD(\G); t).
\end{equation}For any elliptic curve defined over $\QQ$ and prime $\ell$, the determinant of the associated  $\ell$-adic representation of $G_\Q$  is the cyclotomic character, denoted by $\epsilon_\ell$. Putting together, one has
\begin{prop}\label{prop:det-I}
Let $\G=(\infty,\infty,e_\infty)$ where $e_\infty\in\{3,\infty\}$.
Given  $t\in \QQ\setminus\{0,1\}$ and prime $\ell$, at any prime $p \ne \ell$ where the 2-dimensional representation 
$\rho^{BCM}_{HD(\G),1/t,\ell}$ of $G_\Q$ is unramified, the characteristic polynomial of $ 
\rho^{BCM}_{HD(\G),1/t,\ell}
(\Frob_p)$ is
$T^2- H_p(HD(\G);t)T+p.$     
\end{prop}

Proposition \ref{prop:det-I} implies, for each $\G$ in a) list, the isomorphism between the stalks $(\mathcal H(HD(\G))_\ell)_{\bar {\frac 1t}}$ of the hypergeometric sheaf and $ (V^1(\G)_\ell)_{\bar t}$ of the automorphic sheaf at $t \in \QQ \setminus \{0, 1\}$.

This Proposition also holds for $\G=(\infty, \infty, 2)$ with $E_{\G, t} : y^2=x(x^2+x+\frac t4)$ and $HD(\G)'= \{\{\frac14,\frac34 \},\{1,1\}\}$ such that at almost all primes $p$, the characteristic polynomial of $\rho^{BCM}_{HD(\G)',1/t,\ell}(\Frob_p)$ is $T^2 - H_p(HD(\G)'; t)T + p$, as shown in \cite{BCM}.
\begin{remark}
In the elliptic modular cases, we can define $V^1(\G)_{\ell}$ even when $\G$   contains $-I_2$, as
 $R^1 f_* \bar{\QQ}_{\ell}$ for a suitable family of elliptic curves $f:E \to X_{\G}^{{\circ}}$, not unique in general. 
 As Scholl proves, for computing the Hecke traces for modular forms of even weight, any such suitable 
 family of elliptic curves will do.  By suitable we mean that  
 the $j$-invariant of $E \to X_{\G}$ is the given map $ X_{\G}\to X_{\Gamma (1)}$  and model is defined over $\Q(j)$. See \cite[Prop. 1.3]{Sch88}.
\end{remark}

\subsection{For groups in b) list} 
Throughout the remaining discussion of this section, 
we assume all hypergeometric data under consideration are primitive for simplicity. 

We begin by recalling the character sum analogue of the Clausen formula \eqref{eq:Clausen1} established by Evans-Greene \cite{Evans-Greene}, and rephrased in a way compatible with our current setting in \cite{LLT2}.  To ease the notation, we use $J_{\omega}(a,b)$ to denote $J(\omega^{(q-1)a},\omega^{(q-1)b})$ below. 

\begin{theorem}[Evans-Greene \cite{Evans-Greene}]\label{thm:E-G} Let $a,b,s\in\QQ$ such that $2s\equiv a+b \mod \ZZ$.  Let $q$ be a prime power which is congruent to 1 modulo $\mr{lcd} \{a, b,s, \f12+b-s\}$ and $\omega$ a generator of $\widehat{\F_q^\times}$. 
For any
$t\in \F_q\setminus\{ 0,1\} $, we have 

\begin{multline}\label{eq:3P1(not1)}
 H_q\left( \{-\frac 12+b-s, s\}, \{1,b\};{t;\omega}\right)  H_q\left(\{-\frac12 -b+s , 1-s \}, \{1, 1-b\};{t;\omega}\right) \\=\phi_q(1-t)  H_q\left( \{\frac 12 ,a,1-a\}, \{1,b,1-b\};{t;\omega}\right)
   + q^{\delta(b)},
\end{multline} 
 where $\delta(b) =1$ if $b\in \ZZ$ and $0$ otherwise. 
 
When $t=1$, we have 
\begin{multline}
    \label{eq:3P2(1)}
 H_q\left( \{\frac 12 ,a,1-a\}, \{1,b,1-b\};{1;\omega}\right) \\
 =\frac 1 { q^{1-\delta(b)} }\frac{J_\omega(a+b,b-a)}{J_\omega(\frac12, -b)}\left(J_\omega(s-b,\frac12-s)^2 + J_\omega(\frac12+s-b,s)^2\right),
\end{multline}
if $\omega^{(q-1)(a+b)}$ is a square in $ \widehat{\F_q^\times}$; otherwise
 
 \begin{equation}\label{eq:not-square} H_q\left( \{\frac 12 ,a,1-a\}, \{1,b,1-b\};{1;\omega}\right)=0. 
 \end{equation}  
\end{theorem}

By Prop. 8.7 of \cite{Win3X}, the two $H_q$-functions on the left side of \eqref{eq:3P1(not1)} are related by 
 \begin{multline}\label{eq:cxconj} 
    H_q\left( \{-\frac 12+b-s, s\}, \{1,b\};{t;\omega}\right) \\
    =\ol\omega^{(q-1)b}(t) \phi_q(1-t) \frac{J_\omega(-s,s-b)}{J_\omega(\frac12+s,\frac 12+b-s)} H_q\left(\{-\frac12 -b+s , 1-s \}, \{1, 1-b\};{t;\omega}\right).
 \end{multline} 

For each group $\G$ in Theorem \ref{thm:traceformula}, the hypergeometric datum $HD(\G) = \{\alpha(\G), \beta(\G)\}$  is self-dual. Thus one can find a (non-unique) pair $(a,b)$ such that $\alpha(\G)=\{\frac12,a,1-a\},\beta(\G)=\{1,b,1-b\}$. 
With $s=\frac{a+b-1}2$, the hypergeometric datum $HD_1(\G) :=\{\{-\frac 12+b-s, s\}, \{1,b\}\}$ on the left hand side of \eqref{eq:3P1(not1)} 
agrees with the datum associated to the $_2F_1$ listed in the second to the last column of Table \ref{tab:2}. 

\begin{defn}\label{defn:2} For $\l \in \QQ \smallsetminus \{0, 1\}$, denote by $\chi_{1-1/\l}$ the quadratic character $\big(\frac{1-1/\l}{\cdot}\big)_2$ of $G_\Q$. {It is the $G_\QQ$-representation on the stalk at $\l$ of the $\ell$-adic sheaf $\mathcal H(\{1/2\}, \{1\})_\ell$.}
\end{defn}

\begin{prop}  
    \label{prop:eta}
     Fix a prime $\ell$. 
     For each group $\G$ in Table \ref{tab:6} below with $N(\G)=\text{lcm}(4,M(HD_1(\G)))$  and $\l\in\QQ(\zeta_{N(\G)})\setminus\{0,1\}$, let $\sigma_{\G,\l,\ell}:=  \rho_{HD_1(\G), \l,\ell}|_{G(N(\G))}$, whose Frobenius trace at 
     each prime ideal $\wp$ of $\ZZ[\zeta_{N(\G)}, 1/(N(\G)\ell\l)]$ above a prime $p\equiv 1\mod N(\G)$ 
     is given in the table below, in which $J_{\wp}(HD_1(\G)):= -J_\omega\left(\frac7{24},-\frac56\right).$ 
     Let $\bar \sigma_{\G,\l,\ell}:=  \rho_{HD_1'(\G), \l,\ell}|_{G(N(\G))}$. Then
     
     1) $\sigma_{\G,\l,\ell}$ and $\bar \sigma_{\G,\l,\ell}$ differ by  $\chi_{1-1/\l}|_{G(N(\G))}$ times a finite order character $\xi_{\G, \l, \ell}$; 
     
     2) $\sigma_{\G,\l,\ell} \otimes \bar \sigma_{\G,\l,\ell}$ 
     contains a 3-dimensional subrepresentation {$\chi_{1-1/\l}|_{G(N(\G))}\xi_{\G, \l, \ell}\otimes {\rm Sym}^2 \sigma_{\G, \l, \ell}$}, 
     which is isomorphic to  $\chi_{1-1/\l}\otimes\rho_{HD(\G),\l,\ell}|_{G(N(\G))}$ with determinant $\epsilon_\ell^3$, where $\epsilon_\ell$ is the cyclotomic character on $G(N(\G))$. 
\begin{table}[h]
    \centering
        \begin{tabular}{|c|c|c|c|c|c|c} \hline
    $\G$&$N(\G)$&$(a(\G),b(\G))$ & $HD_1(\G)$ & $HD_1'(\G)$ &$\mr{Tr}\sigma_{\G,\l,\ell}(\Frob_{\wp})$\\ \hline
     $ (\infty,2,\infty)$ &4& $(\frac12,1)$ & $\{\{\frac14,\frac14\},\{1,1\} \}$ & $\{\{\frac34,\frac34\},\{1,1\} \}$& $H_p\left(HD_1(\G);1/\l;\omega_\wp\right)$ \\ \hline
      $ (\infty,2,3)$ &12& $(\frac16, 1)$ & $\{\{\frac1{12},\frac5{12}\},\{1,1\} \}$& $\{\{\frac{11}{12},\frac7{12}\},\{1,1\} \}$ & $H_p\left(HD_1(\G);1/\l;\omega_\wp\right)$ \\ \hline
       $ (\infty,2,4)$ &8& $(\frac 14, 1)$ & $ \{\{\frac18,\frac38\},\{1,1\} \}$ & $ \{\{\frac78,\frac58\},\{1,1\} \}$ & $H_p(HD_1(\G);1/\l;\omega_\wp)$  \\ \hline
        $ (\infty,2,6)$ &12& $(\frac13, 1)$ & $\{\{\frac16,\frac13\},\{1,1\} \}$ & $\{\{\frac56,\frac23\},\{1,1\} \}$ & $H_p(HD_1(\G);1/\l;\omega_\wp)$ \\ \hline
         $ (6,2,4)$ &24& $(\frac14, \frac56)$ & $\{\{\frac1{24},\frac7{24}\},\{1,\frac56\} \}$ &  $\{\{\frac{23}{24},\frac{17}{24}\},\{1,\frac76\} \}$& $\mathcal J_{\wp}(HD_1(\G)) \times $\\&&&&&$H_p(HD_1(\G);1/\l;\omega_\wp)$ \\ \hline
    \end{tabular}
    \caption{$\sigma$ representations for groups in Table \ref{tab:2}}
    \label{tab:6}
\end{table}
\end{prop} 

\begin{proof}By Remark \ref{rem:phi(M,a)}, the characters $\phi(M(HD_1(\G)),\frac12) $ and $\phi(M(HD(\G)),\frac12) $ are trivial on $G(N(\G))$, 
so they will not show up in the discussion below.  Since a prime $p \equiv 1 \mod N(\G)$ splits completely in $\QQ(\zeta_{N(\G)})$, the prime ideals $\wp$ above $p$ have norm equal to $p$. The trace of $\sigma_{\G,\l,\ell}(\Frob_{\wp})$ follows from \eqref{eq:Tr1} and \eqref{eq:P-H}. 
Statement 1) follows from \eqref{eq:cxconj} while statement 2) follows from the finite field Clausen formula (Theorem \ref{thm:E-G}). More precisely, when $b=1$, the quotient of the Jacobi sums on the right hand side of \eqref{eq:cxconj} is 1. 
In the case of $(2,4,6)$, letting  $a=1/4$ and $b=5/6$ in \eqref{eq:cxconj}  gives  the following trace function for $\bar \sigma_{\G,\l,\ell}$:
 \begin{align*}
 \Tr \bar \sigma_{\G,\l,\ell} (\Frob_\wp)
&=-J_\omega\left(\frac{17}{24},\frac56\right)H_p\left(\{\frac {17}{24},\frac{23}{24}\},\{1,\frac {7}6\};1/\l;\omega_\wp\right)\\&=   -\phi_p(1-1/\l)\ol\omega_\wp(1/\l)^{\frac{p-1}{6}}  \cdot C(\G, \omega_\wp)\cdot J_\omega\left(\frac{7}{24},-\frac56\right)H_p(HD_1(\G);1/\l;\omega_\wp)\\
&=\phi_p(1-1/\l)\ol\omega_\wp(1/\l)^{\frac{p-1}{6}}  \cdot C(\G, \omega_\wp)\cdot  \Tr  \sigma_{\G,\l,\ell} (\Frob_\wp)
\end{align*} \bk where $C(\G, \omega_\wp)= \frac{J_\omega(\frac{17}{24},\frac56)}{J_\omega(\frac{7}{24},-\frac56)}  \frac{J_\omega(\frac7{24},\frac{13}{24})}{J_\omega(\frac5{24},-\frac1{24})}$. By converting $J(A,B) = \g(A)\g(B)/\g(AB)$  when $AB$ is nontrivial and using identities on Gauss sums, one verifies that $C(\G,\omega_\wp) = 1.$ 

Thus $\bar \sigma_{\G,\l,\ell}\cong\xi_{\G,\l,\ell}\otimes \chi_{1-1/\l}|_{G(N(\G))}\otimes \sigma_{\G,\l,\ell}$, where $\xi_{\G,\l,\ell}$ is a finite order character of $G_{N(\G)}$; this is 1). Consequently, 
\begin{equation*}
  \sigma_{\G,\l,\ell}\otimes  \bar \sigma_{\G,\l,\ell}= {\chi_{1-\frac{1}{\l}}|_{G(N(\G))}}\xi_{\G,\l,\ell}\otimes( \sigma_{\G,\l,\ell}\otimes {\sigma}_{\G,\l,\ell}) \cong {\chi_{1-\frac{1}{\l}}}|_{G(N(\G))}\xi_{\G,\l,\ell}\otimes \left(  \mr{Sym}^2 \sigma_{\G,\l,\ell} \oplus \mr{Alt}^2 \sigma_{\G,\l,\ell}\right), 
\end{equation*} showing that $\sigma_{\G,\l,\ell}\otimes  \bar \sigma_{\G,\l,\ell}$ contains the degree-3 subrepresentation $\chi_{1-1/\l}|_{G(N(\G))}\xi_{\G,\l,\ell}\otimes \mr{Sym}^2 \sigma_{\G,\l,\ell} $. On the other hand, it follows from the Frobenius traces of $\rho_{HD(\G),\l,\ell}$ shown in Table \ref{tab:6} that 
\begin{align*}
\Tr  \sigma_{\G,\l,\ell}\otimes  \bar \sigma_{\G,\l,\ell}(\Frob_\wp)&=p^{1-\delta(b)} H_p(HD_1(\G);1/\l; \omega_\wp)H_p(HD_1'(\G);1/\l; \omega_\wp)\\
&= p^{1-\delta(b)}\phi_p(1-1/\l)H_p(HD(\G); 1/\l; \omega_\wp) + p \\
&= \chi_{1-1/\l}(\Frob_\wp)\Tr \rho_{HD(\G), \l, \ell}(\Frob_\wp) + \epsilon_\ell(\Frob_\wp).
\end{align*}
In the above formula, the second equality follows from the Clausen formula  \eqref{eq:3P1(not1)}, and the third equality results from Theorem \ref{thm:Katz} and the relation \eqref{eq:P-H}. This proves that $\chi_{1-1/\l}\otimes \rho_{HD(\G), \l, \ell}|_{G(N(\G))} \cong \chi_{1-1/\l}|_{G(N(\G))}\xi_{\G,\l,\ell}\otimes \mr{Sym}^2 \sigma_{\G,\l,\ell}$
and $\epsilon_\ell \cong \chi_{1-1/\l}|_{G(N(\G))} \xi_{\G, \l, \ell} \otimes \mr{Alt}^2 \sigma_{\G,\l,\ell}$, which in turn implies that
$$\det \left( \chi_{1-1/\l}|_{G(N(\G))}\xi_{\G,\l,\ell}\otimes  \mr{Sym}^2 \sigma_{\G,\l,\ell}\right)= \det \left(\chi_{1-1/\l}\otimes \rho_{HD(\G), \l, \ell}|_{G(N(\G))}\right) =\epsilon_\ell^3. $$
This proves 2).
\end{proof}

Proposition \ref{prop:eta} is the hypergeometric sheaf version of Proposition \ref{P:ALLL2}. For each $\G$ in Proposition \ref{prop:eta}, the hypergeometric datum $HD(\G)$ is defined over $\QQ$ so that the representation $\rho_{HD(\G), \l, \ell}$ for $\l$ in $\QQ^\times$ can be extended to $G_\QQ$; we choose $\rho^{BCM}_{HD(\G),\l,\ell}$ as its extension. This defines 
the $\ell$-adic hypergeometric sheaf $\mathcal H(HD(\G))_\ell$ on $X_\G$. Recall from definition \ref{defn:2} that the characters $\chi_{1-1/\l}$ come from the sheaf $\mathcal H(\{1/2\}, \{1\})_\ell$ on $X_\G$.
These representations have $\QQ$-valued traces. Other extensions differ from this by a character of $G_\QQ/G({N(\G)})$. By the Isomorphism Theorem \ref{thm:blist} over $\CCC$ and the Comparison Theorem \ref{thm:compare}, $\mathcal H(\{1/2\}, \{1\})_\ell\otimes \mathcal H(HD(\G))_\ell$ (upon changing $\l$ to $1/\l$)  and $ V^2(\G)_\ell$ differ by a character $\chi_\G$ of $G_\Q$. Both sheaves are punctually pure of the same weight, which can be seen from 
Propositions \ref{P:ALLL2} and 
\ref{prop:eta}. Thus,  by Lemma \ref{lem:chi-finite}  $\chi_\G$ has finite order. From definition, the traces of $\mathcal H(\{1/2\}, \{1\})_\ell\otimes \mathcal H(HD(\G))_\ell$ at rational stalks  have values in $\QQ$. The traces of $V^2(\G)_\ell$ at rational stalks, which originated in the local zeta functions of abelian varieties as in \cite{KS2}, also have values in $\QQ$ (See \eqref{eq:traceZ}). Thus  the order of $\chi_\G$  is at most 2. 
We summarize the discussion in the following Corollary.

\begin{cor}\label{cor:chi}
 $\mathcal H(\{1/2\}, \{1\})_\ell\otimes \mathcal H(HD(\G))_\ell$ (upon changing $\l$ by $1/\l$) and $V^2(\G)_\ell$ differ by a character $\chi_\G$ of order at most 2. 
\end{cor}

\begin{prop}\label{prop:det}
Let $\G$ be a group in Table \ref{tab:6} with datum $HD(\G) = \{\alpha(\G), \beta(\G)\}$. Write $\varphi_\G$ for the character $\phi(lcd(HD(\G)), \sum_{a_i \in \alpha(\G)} a_i)$ of $G_\QQ$ in Theorem \ref{thm:KBCM} ii). 
For any $\l\in \Q\setminus\{0,1\}$, 
    the determinant of $\varphi_\G \otimes \chi_{1-1/\l}\otimes \rho_{HD(\G), \l, \ell}^{BCM}$ is $\epsilon_\ell^3$ for the first 4 groups and is $\chi_{-3}\otimes \epsilon_\ell^3$ for $(2,4,6)$. 
\end{prop}

\begin{proof}
Recall that $X_\G$ is a genus zero curve defined over $\QQ$. Using the model we choose, we identify $X_\G^\circ(\CCC)$ with $\mathbb P_\CCC^1 \setminus \{0,1,\infty\}$.  The representation $\chi_{1-1/\l}\otimes  \rho_{HD(\G), \l, \ell}^{BCM}$ is the Galois action on the stalk at $\l $ of the rank-3 sheaf $\mathcal L(\G)_\ell\otimes \mathcal H(HD(\G))_\ell$, 
so its determinant, as $\l$ varies, constitutes a rank-1 $\ell$-adic sheaf, denoted by $\mathcal F(\G)_\ell$, on $\mathbb{P}^1_{\QQ} \setminus \{ 0, 1, \infty\}$.  According to Proposition \ref{prop:eta}, when restricted to $\mathbb{P}^1 _{\QQ (\zeta_{   N(\Gamma)})}\setminus \{ 0, 1, \infty\}$, 
$\mathcal F(\G)_\ell$ is the constant sheaf $\epsilon _{\ell}^3$.  Therefore 
$\mathcal F(\G)_\ell$ will be a constant sheaf with value $\epsilon _{\ell}^3$ times a character
of $G_{\QQ}/G(N(\G))$ by Theorem \ref{thm:compare}.
As $\varphi_\G$ is also a constant degree-1 sheaf on $\mathbb{P}^1_{\QQ} \setminus \{ 0, 1, \infty\}$, trivial on $G({   N(\G)})$ by Remark \ref{rem:phi(M,a)}, 
this proves that the determinant of $\varphi_\G \otimes {\chi_{1-1/\l} \otimes}\rho_{HD(\G), \l, \ell}^{BCM}$ is $\epsilon_\ell^3$ up to a character $\psi_\l$ trivial on $G({   N(\G)})$ and independent of $\l$. By Theorem \ref{thm:KBCM} ii) the trace of $\varphi_\G \otimes {\chi_{1-1/\l} \otimes}\rho_{HD(\G), \l, \ell}^{BCM}(\Frob_p)$ is $\left(\frac{1-1/\l}p\right)H_p(HD(\G),1/\l)p^{(3-m)/2}$ for primes $p \nmid {N(\G)}\ell$. Here $m$ is the number of $1$'s in $\beta(\G)$. 

Let $S$ be the group of characters of $G_\QQ/G({   N(\G)})\cong \mr{Gal}(\QQ(\zeta_{{N(\G)}})/\QQ)$, which is a product of finitely many copies of $\ZZ/2\ZZ$. Hence any character in $S$ has order at most 2. We proceed to show that $\psi_\l$ is $\chi_{-3}$ for $\G=(2,4,6)$ and trivial otherwise.
As it is independent of $\l$, it suffices to choose a convenient $\l$ and compute the determinant of the representation at a few primes using the \texttt{Magma} hypergeometric motive package implemented by Watkins \cite{Watkins-HGM-documentation} to 
make the final conclusion. 

For $(2,\infty,\infty)$, $N(\G)=4$,  $S= \langle \chi_{\pm 1} \rangle$. 
It suffices to consider $\l=3$ at $p=7$; the determinant is $7^3$. Hence $\psi_\l$ is trivial.

For $(2,3,\infty)$ and for $(2,6,\infty)$, $N(\G)=12$, $S= \langle \chi_{\pm 1},\chi_{\pm 3}\rangle$. It suffices to consider $\l=3$ at $p=7, 11$; the determinants are $7^3$ and $11^3$, respectively. Hence $\psi_\l$ is trivial.

For $(2,4,\infty)$, $N(\G)=8$, the final choice is among trivial, $S=\langle \chi_{\pm 1}, \chi_{\pm 2}\rangle$. It suffices to consider $\l=3$ at $p=5, 7, 11$; the determinants are $5^3$, $7^3$ and $11^3$, respectively. Hence $\psi_\l$ is trivial.

For $(2,4,6)$, $M(\G)=24$, $S=\langle \chi_{\pm 1},\chi_{\pm 2},\chi_{\pm 3} \rangle$. It suffices to consider $\l=3$ at $p=5$, $7$, $11$, $13$, $17$, $19$, and the corresponding determinants are $-5^3, 7^3, -11^3,13^3,$ $-17^3,$ $19^3$ from which we conclude $\psi_\l = \chi_{-3}.$
\end{proof}
 
A straightforward computation shows that the quadratic character $\varphi_\G$ in Proposition \ref{prop:det} is equal to $\chi_{-1}$ for $\G = (2,\infty, \infty), (2,3,\infty)$ and $(2,6,\infty)$, and is trivial for $\G = (2,4,\infty)$ and $(2,4,6)$. It follows from Proposition \ref{P:ALLL2} and the definition of $V^2(\G)_\ell$ as the pushforward of $V^2(U')_\ell$ for a neat normal compact open finite-index subgroup $U'$ of the group $U$ corresponding to $\Gamma$ that, at a $\QQ$-rational point $\l$ of $X_\G^\circ$, $\det \rho^2(\G)_{\l,\ell} = \epsilon_\ell^3$. This also can be seen from the fact that corresponding modular forms have trivial central characters.  The representation at the stalk at $\lambda$ of $\mathcal H(\{1/2\}, \{1\})_\ell\otimes \mathcal H(HD(\G))_\ell$ is $\chi_{1-1/\l}\otimes \rho_{HD(\G), \l, \ell}^{BCM}$, which is isomorphic to $\chi_\G \otimes \rho^2(\G)_{\l,\ell}$ for a character $\chi_\G$ of order $\le 2$ by Corollary \ref{cor:chi}. The above proposition then determines $\chi_\G^3 = \chi_\G$. We record the conclusion below.

\begin{theorem}\label{thm:chi-Gamma} Let $\G$ be a group in Table \ref{tab:6} with datum $HD(\G) = \{\alpha(\G), \beta(\G)\}$ as in Theorem \ref{thm:traceformula}. Then, upon changing $\l$ to $1/\l$, $\mathcal H(\{1/2\}, \{1\})_\ell\otimes \mathcal H(HD(\G))_\ell = \chi_\G \otimes V^2(\G)_\ell$ with $\chi_\G =\chi_{-1}$ for $\G = (2,\infty, \infty), (2,3,\infty)$ and $(2,6,\infty)$, $\chi_\G = \chi_{-3}$ for $\G=(2,4,6)$, and $\chi_\G$ trivial for $\G=(2,4,\infty)$.
\end{theorem}

\subsection{Some properties of $H_p(HD(\G);t)$}
We explore properties of $H_p(HD(\G);t)$ for $\G$ in Table \ref{tab:6} according to their type. The case $\G=(2,4,6)$ of type b.II) will be treated separately. 
For the four groups $\G$ of type b.I), $HD(\G)=\{\{ \frac12, a(\G), 1-a(\G)\}, \{1,1,1\}\}$ where $a(\G)$ is listed in Table \ref{tab:6}. 
The discussion for these groups will be in terms of $d=d(\G) = 1/a(\G)$, which takes values 2,3,4,6, respectively.  Define $\kappa_d$ as follows as in Table \ref{tab:kappa_d}.   
\begin{table}[h]
    \centering
    \begin{tabular}{|c||c|c|c|c|}
    \hline
    $\G$ & $(\infty, 2, \infty)$ & $(\infty, 2,6)$ & $(\infty, 2, 4)$ & $(\infty, 2, 3)$ \\
    \hline
        $d=d(\G)$ &2&3&4&6  \\
        \hline
        $\kappa_d$ &$-1$&$-3$&$-2$&$-1$ \\
        \hline
    \end{tabular}
    \caption{$\kappa_d$ values}
    \label{tab:kappa_d}
\end{table}

\begin{lemma}\label{thm:x->1-x}Let $\G$, $d=d(\G)$ and $\kappa_d$ be as in Table \ref{tab:kappa_d} and  $HD_2(\G) = HD_2(d):=\{\{\frac 1d,\frac{d-1}d\},\{1,1\}\}$. 
For any prime $p>5$,  $x\in \F_p\setminus\{0,1\}$, the following identities hold:
\begin{equation}
    H_p\left (HD_2(d);x\right)=\left ( \frac{\kappa_d}p\right) H_p\left(HD_2(d);1-x\right), 
\end{equation}
\begin{equation}\label{eq:CMat1/2}
    H_p\left (HD_2(d);1/2\right)= 0 \quad  \mbox{ if } \left ( \frac{\kappa_d}p\right)=-1.
\end{equation}
\end{lemma} 

\begin{remark}\label{rem:CMbykappad} Since $HD_2(\G)$ is defined over $\QQ$, for $\l \in \QQ^\times$, the representation $\rho_{HD_2(\G), \l, \ell}$ of $G(d)$ can be extended to $G_\QQ$. It follows from Theorem \ref{thm:KBCM} 
and \eqref{eq:CMat1/2} that $\rho_{HD_2(\G), 2, \ell}^{BCM}$ has CM by $\QQ( \sqrt{\kappa_d})$. 
\end{remark}

\begin{proof} 
 By (9.13) of \cite{Win3a},  these identities hold when $p\equiv 1\mod d$.  For $d \in \{3,4,6\}$, by Remark \ref{cor:1}
both sides agree  up to a sign for $p\equiv -1\mod d$. To decide the sign, we  use the following three-step strategy. 

Firstly, we relate the $H_p$ function to truncated hypergeometric series. Theorem 4 of \cite{Long18} says for $d\in \{3,4,6\}$ and $x\in \F_p\setminus\{0,1\}$, $$H_p\left (\{\frac 1d,\frac{d-1}d\},\{1,1\};x\right)\equiv \pFq21{\frac 1d&\frac{d-1}d}{&1}{x}_{p-1} \mod p.$$ 

Next, we find a terminating series which is $p$-adically close to the truncated  series. Let $p_0$ be the least positive integer such that $n=\frac{pp_0-1}d\in \Z$. 
Note that $n\equiv -\frac {1}d\mod p$ and $(-1)^k\binom{n}k=\frac{(-n)_k}{k!}\equiv \frac{(\frac1d)_k}{k!}\mod p$, $\binom{n+k}k=\frac{(1+n)_k}{k!}\equiv \frac{(\frac{d-1}d)_k}{k!}\mod p$ for $0\le k<p$. Thus $$\pFq21{\frac1d&\frac{d-1}d}{&1}{x}_{p-1}\equiv \pFq21{-n&n+1}{&1}{x}\equiv \sum_{k=0}^n \binom{n}k\binom{n+k}{k}(-x)^k\mod p.$$

At last we use  identities for the terminating series to reach the  goal. Presently, recall that  \begin{equation*}
    \sum_{k=0}^n\binom{n}{k}\binom{n+k}{k}(-x)^k=(-1)^n \sum_{k=0}^n\binom{n}{k}\binom{n+k}{k}(x-1)^k
\end{equation*} (see (2.3.11) of \cite{AAR}).  For $d=3$, when $p=6k+1$ (resp. $6k-1$) for  $k\in\ZZ$, $n=\frac{p - 1}3=2k$  (resp. $n=\frac{2p - 1}3=4k-1$). Thus $(-1)^n=\left(\frac{-3}p\right).$ The other cases can be verified similarly.
\end{proof}

\begin{remark} In the proof above, a congruence relation between a certain $H_p$ function and a truncated hypergeometric series is used. Along a similar vein, the paper \cite{SRiM} by Barbei, Swisher, Tobin and Tu describes certain supercongruences satisfied by truncated Ramanujan-Sato series by drawing connections to arithmetic triangle groups. The supercongruences can be predicted by the properties of  $H_p(HD(\G);t)$ described in this section and the recent development, which is recorded in \cite[Appendix]{SRiM}. 
\end{remark}

The next Lemma is a special version of the Clausen formula 
that works for almost all $q$. 

\begin{lemma}\label{lem:111-Clausen} For $\G$ of type b.I), $d=d(\G)\in\{2,3,4,6\}$,  $p$ being any prime not dividing $2d$, 
and $t\in \F_p  \setminus\{0,1, 1/2\}$, we have 
    \begin{equation}
    H_p(HD_2(\G); t)^2 = H_p\left(\{\frac 1d,1-\frac1d\},\{1,1\}; t\right)^2
    = H_p(HD(\G);4t(1-t))+p.
\end{equation}
\end{lemma}

\begin{proof}
    By our assumption, 
      $1-4t(1-t) = (1-2t)^2 \ne 0, 1$. For  $p\equiv 1 \mod 2d$, we have
    \begin{align*}
      H_p&\left(\{\frac12,\frac 1d,1-\frac1d\},\{1,1,1\};4t(1-t)\right)+p\\
      & \overset{Clausen}{=}H_p\left(\{1-\frac1{2d}, \frac 12+\frac 1{2d}\},\{1,1\}; 4t(1-t)\right)H_p\left(\{\frac1{2d}, \frac 12-\frac 1{2d}\},\{1,1\}; 4t(1-t)\right)\\
     & {=}  
       H_p\left(\{\frac1{d}, 1-\frac 1{d}\},\{1,1\}; t\right) ^2.
      \end{align*}In the above, we use  \cite[(9.12)]{Win3X} to get the last equality. 
    For $p$ with $p\not\equiv 1 \mod 2d$, the conclusion can be drawn using the three-step strategy outlined in the proof of Lemma \ref{thm:x->1-x} that 
    $$
        H_p\left(\{\frac12,\frac 1d,1-\frac1d\},\{1,1,1\};t\right) \equiv \pFq32{\frac 12&\frac 1d&\frac{d-1}d }{&1&1}{t}_{p-1}\mod p.
    $$
    and 
    $$
      \pFq21{\frac1d&\frac{d-1}d}{&1}{t}_{p-1}^2\equiv \pFq32{\frac 12&\frac 1d&\frac{d-1}d }{&1&1}{4t(1-t)}_{p-1}\mod p^2. 
    $$ In the last step,  Clausen formula over $\C$ is used.
\end{proof}

\subsection{Special fiber at $1$}
In this subsection we give a description for $H_p(HD(\G);1)$.  This information will be useful for computing the contributions at the special fibers in \S \ref{S:singular}.  Firstly, we give an evaluation formula for the length-2 datum at $t=1/2$. 

\begin{lemma}\label{lem:value at 1/2} Let $a,s\in\QQ^\times$ such that $2s\equiv a\mod \ZZ$.  For any prime  $p$ congruent to 1 modulo $\mr{lcd} \{a, \f12\}$, 
 $$
      H_p\left(\{a,1-a\},\{1,1\};\frac 12\right)=-\phi_p(-1)\cdot \left(J_\omega(s, \frac12-s) + J_\omega(-s, \frac12+s)\right), 
    $$
    when $2 | (p-1)a$; otherwise
    \begin{equation*}
        H_p\left( \{a,1-a\}, \{1,1\};{\frac 12;\omega}\right)=0. 
    \end{equation*}
\end{lemma}

\begin{proof}
When $\omega^{(p-1)a}$ is a not square in $\widehat{F_q^\times}$, the  evaluation formula \cite[(4.5)]{Greene} gives  $$ H_p\left( \{a,1-a\}, \{1,1\};{\frac 12;\omega}\right)=0.$$  If $\omega^{(p-1)a}$ is a square in $\widehat{\F_p^\times}$, 
    combining the  evaluation formula \cite[(4.5)]{Greene} and the Jacobi sum identities \cite[(2.16), Lemma 2.8]{Win3X}, we obtain  firstly by \cite[(4.5)]{Greene}       
\begin{align*}
     ~& H_p\left(\{a,1-a\},\{1,1\};\frac 12; \omega \right) =-\omega^{(p-1)a}(2)\cdot \left(J_\omega(s, -a) + J_\omega(\frac12+s, -a)\right)
\end{align*}     Further by \cite[Lem. 2.8]{Win3X}, we get
\begin{align*} H_p&\left(\{a,1-a\},\{1,1\};\frac 12; \omega \right)     \\& = -\omega^{(p-1)a}(2)\cdot \left(\omega^{(p-1)s}(-1)J_\omega(s, s) +\omega^{(p-1)(\frac 12+s)}(-1) J_\omega(\frac12+s, \frac12+s)\right) \\ 
       & = -\phi_p(-1)\cdot \left(J_\omega(s, \frac12-s) + J_\omega(-s, \frac12+s)\right).
\end{align*}    
\end{proof}

\begin{cor}\label{cor:2}
{For $\G$ of type b.I), $d=d(\G)$ and $\kappa_d$ as} in Table \ref{tab:kappa_d}, and any prime $ p\nmid d$, we have

\begin{eqnarray}\label{eq:CMbykappad}
{H_p(HD(\G); 1)} = H_p\left(\{\frac 12,\frac1d,\frac{d-1}d\},\{1,1,1\};1\right)=0 ,\quad  {~ if~ } \left(\frac{\kappa_d}p\right)=-1; 
\end{eqnarray}

\begin{eqnarray}\label{eq:clausenextended}
   H_p\left(HD_2(\G); \frac 12\right)^2 = H_p\left(\{\frac1d,\frac{d-1}d\},\{1,1\};\frac12\right)^2= H_p(HD(\G);1) + 2p, \,  {~~~~~if~}  \left(\frac{\kappa_d}p\right)=1. 
\end{eqnarray}
    
For $\G=(2,4,6)$ 
$H_p(HD(\G);1)= H_p\left (\{\frac 12,\frac14,\frac34\},\{1,\frac16,\frac56\};1\right)=0$ for primes $p$ with $\left(\frac{-6}p\right)=-1$. 
\end{cor}

\begin{proof} For each of the five groups $\G$ considered here, $HD(\G)$ is defined over $\QQ$, so that the degree-2 representation $\rho_{HD(\G),1,\ell}$ can be extended to $\rho_{HD(\G),1,\ell}^{BCM}$ of $G_\QQ$. It follows from Theorem \ref{thm:KBCM}, ii) that at an unramified odd prime $p$, we have
$${\rm Tr} \rho_{HD(\G),1,\ell}^{BCM}(\Frob_p) = \phi(lcd(HD(\G)), \f12)(\Frob_p)p^uH_p(HD(\G);1)$$ where $u = 0$ (resp. 1) for $\G$ of type b.I) (resp. b.II)), and ${\rm Tr} \rho_{HD_2(\G), 2, \ell}^{BCM}(\Frob_p) = H_p(HD_2(\G);1/2)$ for $\G$ of type b.I). Here $HD_2(\G)$ is defined above Table \ref{tab:kappa_d}.

When $\G$ is of type b.I), by Clausen formula {\eqref{eq:3P2(1)} with $2s \equiv a \mod \ZZ$ and \eqref{eq:not-square}}, for each $p\equiv 1\mod 2d$ we have
$$ 
    H_p\left(\{\f12, a,1-a\},\{1,1,1\};1\right) =J_\omega(s, \frac12-s)^2 + J_\omega(-s, \frac12+s)^2.
$$
This shows that the degree-2 representation $\rho_{HD(\G), 1,\ell}^{BCM}$ restricted to $G(2d)$ decomposes into the sum of two distinct characters of $G(2d)$. Thus, by Clifford theory, the 2-dimensional $\rho_{HD(\G),1,\ell}^{{BCM}}$ is an irreducible representation induced from a one-dimensional representation $\xi$ of $G(2d)$ corresponding to the first Jacobi sum  extended to an index-2 subgroup $H(2d)$ of $G_\QQ$ generated by $G(2d)$ and all $\{\tau\in G_\QQ/G(2d)\mid \xi^\tau\cong \xi\}$. To decide $H$, we use Lemma \ref{lem:111-Clausen} with $a=\frac 1d$ and $s=\frac1{2d}$: 

\begin{multline}
H_p\left(\{ a,1-a\},\{1,1\};\frac 12\right)^2 
       =J_\omega(s, \frac12-s)^2 + J_\omega(-s, \frac12+s)^2+2p \\
       = H_p(HD(\G);1) + 2p, \quad \forall p\equiv 1\mod 2d.
\end{multline} 
This shows that the representation $\rho_{HD_2(\G),2,\ell}^{BCM}\otimes \rho_{HD_2(\G),2,\ell}^{BCM}$, when restricted to $G(2d)$, agrees with the restriction to $G(2d)$ of the sum of the degree-2 representation $\rho_{HD(\G),1,\ell}^{BCM}$ and two copies of the degree-1 representation $\epsilon_\ell$. Since all representations involved are of $G_\QQ$ and $\rho_{HD(\G),1,\ell}^{BCM}$ is irreducible, we see that $\rho_{HD_2(\G),2,\ell}^{BCM}\otimes \rho_{HD_2(\G),2,\ell}^{BCM}$ can be decomposed in the same way, except that each irreducible component is now replaced by a suitable twist by a character of $G_\QQ/G(2d)$. 
On the other hand, since $\rho_{HD_2(\G),2,\ell}^{BCM}$ has CM by $\Q(\sqrt{\kappa_d})$ by Remark \ref{rem:CMbykappad}, so do $\rho_{HD_2(\G),2,\ell}^{BCM}\otimes \rho_{HD_2(\G),2,\ell}^{BCM}$ and its only degree-2 irreducible component. Therefore $\rho_{HD(\G),1,\ell}^{BCM}$ admits the same CM. This proves \eqref{eq:CMbykappad}. 

The identity \eqref{eq:clausenextended} can be drawn by checking a few  primes representing  other cosets of $G(2d)$ in $G_\QQ$.  For instance for $d=2$, we only need to check one unramified prime in $p\equiv 3\mod 4$ case such as $p=7$. 

 For $\G=(2,4,6)$, $\rho_{HD(\G),1,\ell}$ is a representation of $G(12)$. If $p\equiv 1 \mod 24$, \eqref{eq:3P2(1)} gives 
 \begin{align*}
 pH_p &\left (\{\frac 12,\frac14,\frac34\},\{1,\frac16,\frac56\};1\right) 
           =\frac{J_\omega\left(\frac{5}{12},\frac{11}{12}\right)}{J_\omega\left(\frac1{2},\frac{5}{6}\right)}\left (J_\omega\left(\frac1{24}, \frac{7}{24}\right)^2+J_\omega\left(\frac{13}{24},\frac{19}{24}\right)^2\right).
\end{align*}
By the same argument as above, the representation $\rho_{HD(\G),1,\ell}^{BCM}$ of $G_\QQ$ is invariant under the twist by a quadratic character of $G_\QQ$. To identify this  quadratic character, it suffices to compute the values of $H_p\left (\{\frac 12,\frac14,\frac34\},\{1,\frac16,\frac56\};1\right)$ for $p=5,7,11,13,17,19,23$ which reveals that for these cases the value is 0 if and only if $\left(\frac{-6}p\right)=-1.$
\end{proof}
\begin{remark}\label{remark:CM}
The degree-2 representations $\rho_{HD(\G),1,\ell}^{BCM}$ of $G_\QQ$ in Corollary \ref{cor:2} have CM. Hence they are modular. 
In Table \ref{tab:Hp(1)} below we list the corresponding weight-3 cuspidal Hecke eigenforms, identified using a result of Serre (see   \cite[Theorem 10]{LLT2}). Similarly for $\rho_{HD_2(\G), 2, \ell}^{BCM}$ for $\G$ of type b.I) in Lemma \ref{thm:x->1-x}. 
\begin{table}[h]
    \centering
   $$
  \begin{array}{|c|c|c|c|c|} \hline
\G&H_p(HD_2(\G);\f12) & { {\rm Tr} \rho_{HD(\G), 1, \ell}^{BCM}(\Frob_p)} 
& { {\rm det} \rho_{HD(\G), 1, \ell}^{BCM}} \\ \hline
(2, \infty,\infty)& a_p(f_{64.2.a.a})&  (\frac {-1}p)H_p\left (\{\frac 12,\frac12,\frac12\},\{1,1,1\};1\right) = a_p(\eta(4\tau)^6) 
& \chi_{-1}\epsilon_\ell^2 \\ \hline
(2, 3, \infty)&  a_p(f_{36.2.a.a})  & (\frac {-1}p)H_p\left (\{\frac 12,\frac13,\frac23\},\{1,1,1\};1\right)=a_p(f_{48.3.e.a}) 
&\chi_{-3}\epsilon_\ell^2 \\ \hline
(2, 4, \infty)& a_p(f_{256.2.a.a})& H_p\left (\{\frac 12,\frac14,\frac34\},\{1,1,1\};1\right)= a_p(f_{8.3.d.a})
&\chi_{-2}\epsilon_\ell^2\\ \hline
(2, 6, \infty)& a_p(f_{676.2.a.a})& (\frac {-1}p)H_p\left (\{\frac 12,\frac16,\frac56\},\{1,1,1\};1\right)=a_p(f_{144.3.g.a})& 
\chi_{-1}\epsilon_\ell^2 \\\hline
(2, 4, 6)&  -& pH_p\left (\{\frac 12,\frac14,\frac34\},\{1,\frac16,\frac56\};1\right)=a_p(f_{24.3.h.a})& 
 \chi_{-6}\epsilon_\ell^2 \\\hline
  \end{array}
$$
    \caption{$H_p(HD(\G);1)$}
    \label{tab:Hp(1)}
\end{table}
\end{remark}

The identities in the two remaining subsections will be used in \S \ref{S:6} to identify the character $\chi_\G$ alluded to by Theorem \ref{thm:compare} as mentioned in \S \ref{ss:3.5}.
\subsection{Some character sum identities for the  $k=1$ cases} \label{ss:k=1}

\begin{lemma}\label{lem:k=1-2infinfty}For any odd prime $p>3$, we have
\begin{eqnarray*}
       1+ \sum_{t\in \F_p^\times }H_p\left (\{\frac12,\frac12\},\{1,1\};t \right)=0;\quad
       H_p\left (\{\frac12,\frac12\},\{1,1\};1 \right)&=&\left(\frac{-1}p\right).\\
       1+ \sum_{t\in \F_p^\times }H_p\left (\{\frac13,\frac23\},\{1,1\};t \right)=0;\quad 
       H_p\left (\{\frac13,\frac23\},\{1,1\};1 \right)&=&\left(\frac{-3}p\right).
\end{eqnarray*}
\end{lemma}
\begin{proof} By definition,
$$\displaystyle
H_p\left(\{\f12,\f12\},\{1,1\};t\right)=\frac{1}{1-p}\sum_{\chi\in\widehat{\F_p^\times}}G(\chi)\chi(t), \quad \mr{where } \quad G(\chi)=\frac{\g(\phi \chi)^2}{\g(\phi)^2}\g(\ol\chi)^2.
$$
It follows from the  orthogonality result
$ \displaystyle
  \sum_{t\in \F_p^\times }\chi(t)=
  \begin{cases}
    p-1,    &  \mbox{ if } \chi=\eps,\\
    0,&  \mbox{ if } \chi\neq \eps,
  \end{cases}
$
 that 
\begin{align*}
  \sum_{t\in \F_p^\times }H_p\left (\{\frac12,\frac12\},\{1,1\};t \right)=& \frac{1}{1-p}\sum_{\chi\in\widehat{\F_p^\times}}G(\chi)\sum_{t\in \F_p^\times }\chi(t)=-G(\eps)=-1.
\end{align*}When $t=1$, by \cite[(Eqn (4.5))]{Win3X}
$$
   H_p\left (\{\frac12,\frac12\},\{1,1\};1 \right)= -\left(\frac{-1}p\right){\mathbb P}\left (\{\frac12,\frac12\},\{1,1\};1;\F_p \right)=\left(\frac{-1}p\right).
$$

Similarly for any prime $p\ge 3$,
$\displaystyle
  1+ \sum_{t\in \F_p^\times }H_p\left (\{\frac13,\frac23\},\{1,1\};t \right)=0 
$
with 
$$
   H_p\left (\{\frac13,\frac23\},\{1,1\};t \right)=\frac{1}{1-p}\sum_{\chi\in\widehat{\F_p^\times}} \frac{\g(\chi^3)}{\g(\chi)}\g(\ol\chi)\chi(3^{-3}t). 
$$
When $p\equiv 1 \mod 3$, 
$$
    H_p\left (\{\frac13,\frac23\},\{1,1\};1 \right)= -{\mathbb P}\left (\{\frac13,\frac23\},\{1,1\};1;\F_p \right)= 1.
$$ For primes $p\equiv 5\mod 6$, by Theorem \ref{thm:Galois-ext'n} it suffices to check one prime in the class. For $p=5$,  $H_p\left (\{\frac13,\frac23\},\{1,1\};1 \right)=-1$. Thus the one dimensional representation $\rho_{\{\frac13,\frac23\},\{1,1\},1,\ell}$ is isomorphic to $\chi_{-3}$. 
\end{proof}
\subsection{Some character sum identities for the  $k=2$ cases}\label{ss:k=2}
\begin{lemma}\label{lem:m=2-(2,4,6)}
  For each  prime $p\ge 5,$ the following identities hold: 
\begin{eqnarray*}\label{eq:246k=2} \G=(2,4,6):&& \sum_{t\in \F_p^\times}\phi(-3(1-t))H_p\left (\{\frac12,\frac14,\frac34\},\{1,\frac16,\frac56\};t \right)+\left (\left (\frac {-6}p\right )+ \left (\frac {-1}p\right )+\left (\frac {-3}p\right )\right )=0.\\
\G=(2,4,\infty):&& \sum_{t\in \F_p^\times}\phi(1-t)H_p\left (\{\frac12,\frac14,\frac34\},\{1,1,1\};t \right)+1+\left ( \left (\frac {-2}p\right )+\left (\frac {-1}p\right )\right)p =0.\\ \G=(2,6,\infty):&& \label{eq:26inftyk=2}\sum_{t\in \F_p^\times}\phi(1-t)H_p\left (\{\frac12,\frac13,\frac23\},\{1,1,1\};t \right)+1+2 \left (\frac {-3}p\right )p =0.\\ \G=(2,\infty,\infty): &&\sum_{t\in \F_p^\times}\phi(1-t)H_p\left (\{\frac12,\frac12,\frac12\},\{1,1,1\};t \right)+2+ \left (\frac {-1}p\right )p =0.\\ \G=(2,3,\infty): && \sum_{t\in \F_p^\times}\phi(1-t)H_p\left (\{\frac12,\frac16,\frac56\},\{1,1,1\};t \right)+1+\left ( \left (\frac {-3}p\right )+\left (\frac {-1}p\right )\right)p =0.\end{eqnarray*} 
\end{lemma} 

\begin{proof}
We give a proof for the $(2,4,6)$-case. The other cases can be proved similarly.
For  $\chi \in \widehat{\F_p^\times}$,
\begin{equation}
    \sum_{s,t\in \F_p} \phi((1-t)s(s-1))\chi(st)=\phi(-1)J(\phi,\chi)J(\phi\chi,\phi)=\begin{cases}
    1& \mr{ if } \quad \chi^2=\eps\\
    p& \mr{ otherwise. }
    \end{cases}
\end{equation}
$$
 H_p\left (\{ \frac12,\frac14,\frac34\},\{1,\frac16,\frac56\};t\right)= \frac 1{p-1}\sum_{s} \phi(-s(1-s)) \sum_\chi \frac{\g(\chi^4)\g(\ol\chi)\g(\ol \chi^6)}{\g(\chi^2)\g(\ol \chi^2)\g(\ol \chi^3)}\chi(27st/4),
$$
where we apply that $\frac{\g(\phi \chi)\g(\ol \chi)}{\g(\phi)}=J(\phi\chi,\ol \chi)=\phi(-1)J(\phi\chi, \phi)=\displaystyle \sum_{s\in \F_p} \phi(-s(1-s))\chi(s)$.  Thus, 
\begin{eqnarray*}
&&\sum_{t\in \F_p^\times}\phi(1-t)H_p\left (\{\frac12,\frac14,\frac34\},\{1,\frac16,\frac56\};t \right)\\
&=&\frac1{p-1}\sum_{t\in \F_p} \phi(1-t) \sum _{s \in \F_p}\phi(s(s-1))\sum_\chi \frac{\g(\chi^4)\g(\ol\chi)\g(\ol \chi^6)}{\g(\chi^2)\g(\ol \chi^2)\g(\ol \chi^3)}\chi(27st/4)\\
&=&\frac{1}{p-1}(1-p)\left (1+\left (\frac3p\right)+p H_p\left (\{\frac14,\frac34\},\{\frac16,\frac56\};1 \right)\right).
\end{eqnarray*}
It remains to prove  $$p\cdot H_p\left (\{\frac14,\frac34\},\{\frac16,\frac56\};1 \right)=\frac p{1-p}\sum_{\chi\in \widehat{\F_p^\times}} \frac{\g(\chi^4)\g(\ol\chi)\g(\ol \chi^6)}{\g(\chi^2)\g(\ol \chi^2)\g(\ol \chi^3)}\chi(27/4)=\left (\frac {2}p\right ).$$  
Let $HD:= \{\{\frac14,\frac34\},\{\frac16,\frac56\}\}$, which has level $12$. Since HD is defined over $\QQ$, the 
1-dimensional representation $\rho_{HD, 1, \ell}$ of $G(12)$  extends to a character $\rho_{HD,1,\ell}^{BCM}$ of $G_\QQ$ satisfying $\rho_{HD,1,\ell}^{BCM}(\Frob_p)=pH_p(HD; 1)$ at primes $p \ne 2, 3$ by Theorem \ref{thm:KBCM} ii). The desired identity can then be reinterpreted as the equality between two characters $\rho_{HD,1,\ell}^{BCM}$ and  $\chi_2=\left(\frac 2{\cdot}\right)$. For $p\equiv 1\mod 12$, this follows from   the identity of Helversen-Pasotto, see (8.1) of \cite{Win3X}. 
For the remaining primes, it suffices to check the equality at primes  $p=5,7,11$, which indeed holds.  Therefore,  
$$
  \sum_{t\in \F_p^\times}\phi(1-t)H_p\left (\{\frac12,\frac14,\frac34\},\{1,\frac16,\frac56\};t \right)=-1-\left (\frac {3}p\right )- \left (\frac {2}p\right ),
  $$
which is equivalent to  the first identity stated in the Lemma. 
\end{proof}

\section{Contributions from the singular fibers}
\label{S:singular}

\subsection{The CM structure at elliptic points}
\label{SS:singular1}
 
Let $\G$ be a group in Theorem \ref{thm:traceformula} or Theorem \ref{thm:traceformula-2F1}. In this section we compute the contributions of Frobenius traces on stalks at elliptic points and cusps of $\G$ whenever they exist. Recall from Proposition \ref{P:ALLL2} that our sheaves, over the complement $X_{\Gamma}^{\circ}$ of the cusps and elliptic points, are in the form 
$f_* ^G (\mathcal {L}^{k}\otimes\mr{Sym}^{2k}(\sigma))$ for a finite Galois $G$-covering $f: Y \to X_{\Gamma} = X$, where
$\sigma $ (resp. $\mathcal {L}$) is an $\ell$-adic rank 2 (resp. rank 1) local system on $Y^{\circ } = f^{-1}(X^{\circ})$. The sheaf $\sigma$ arises from a family of abelian varieties: in the elliptic modular cases, these are
families of elliptic curves; 
in the quaternion
 cases we obtain $\sigma $ from the  $H^1$ of the universal family of abelian varieties with quaternion structure by utilizing 
Proposition \ref{P:quatrep}. In fact, $\mathcal {L}$ is a finite order character on each stalk, and this 
is absent if $Y = X_{\Gamma '}$ where $\Gamma'$ is torsion-free, and we are not considering extension of the field of definition 
of $\sigma$.

In principle, all information about cusps or elliptic points can be obtained from Tate's algorithm or its
generalizations by Liu, \cite{QingLiu}, and by Bouw and Wewers, \cite{BouwWew}. By ramified base extension we obtain a 
family of abelian varieties which has stable 
reduction (see \cite{BLR} for a general reference). In the elliptic modular cases, the N\'eron fiber at a cusp will be of multiplicative type. 
At the elliptic points,   in both the elliptic modular and quaternion cases, we have potential good reduction (see \cite[Section 2]{Serre-Tate}), so 
that the stable fiber is an abelian variety. 

Cusps only occur in the elliptic modular cases.  By Tate's algorithm, we have multiplicative or potentially multiplicative reduction, 
and the trace of Frobenius on the stalk of $V^k(\G)_\ell$ at a cusp is 1 for $k$ even. For $k$ odd and $-id \notin \G$, the contribution is $\pm 1$  or 0. 

To facilitate our discussion at elliptic points, we recall the moduli interpretation of the curves $X_\G$ when $\G$ arises from an indefinite quaternion algebra $B$ defined over $\QQ$. 
  These results are due to Shimura and are summarized in the first chapter of Jordan's thesis \cite{Jordan}, on which our discussion below is based. 
 We fix an identification $\varphi: B \otimes _{\QQ}\RR = M_2(\RR) $.  
The group $B^+$ of elements in $B$ with positive reduced norm acts on $\mathfrak{H}$ via $\varphi(B^+)$. 
 Let $X_B$ be the 
  Shimura canonical model of the Riemann surface $O_B^1\backslash \mfr{H}$. The points
 $ X_B (\CCC) = O_B^1\backslash \mfr{H}$ are in canonical bijection with the equivalence 
 classes of $(A, i, \Theta)$, where $A$ is a  2-dimensional abelian variety, 
  $i: O_B \to \mathrm{End}(A)$ is an embedding and $\Theta$ is a principal polarization. Then $(A, i, \Theta) \mapsto  (A, \Theta) $ defines a map
 from $X_B$ to the moduli space of principally polarized abelian surfaces, whose 
 explicit form is described  by Hashimoto in \cite{Hash95}.
  
Concretely, for each $z \in \mathfrak{H}$ let 
  \[
  L_z =  \varphi(   O_B )   \begin{pmatrix} z \\ 1\end{pmatrix} \subset 
  \CCC^2 \quad {\rm and} \quad 
  A_z = \CCC ^2/  L_z. 
  \]
  Then $A_z$ is an abelian surface and the elements $b\in O_B$ define 
  endomorphisms of $A_z$ by sending $w \in \CCC^2$  
  to $\varphi(b)w$ which carries the lattice $L_z$ to itself.
   This defines $i_z : B \to \mr{End}(A_z)$. 
 
   Let $\mr{End} (A, i,  \Theta)$ be the endomorphisms of this structure, and 
  $\mr{End}^0 (A, i , \Theta)  = \mr{End} (A, i,  \Theta)\otimes \QQ$. Then
  it follows that 
  \[
  \mr{End} (A, i,  \Theta) = 
   \mr{End}_{O_B} (A) =  \{ \gamma \in \mr{End}(A) \mid \gamma i(b)  = i(b) \gamma \text{\ for all\ } b \in O_B\}  
  \]
  and similarly for $\mr{End} ^0(A, i,  \Theta)$. 
  
  For any $z \in \mathfrak{H}$ we see that $\mr{End}(A_z, i_z, \Theta_z)$ consists of 
 $M \in M_2 (\CCC)$ that commute with the image of $\varphi (O_B)$. This forces
$M$ to be a scalar matrix $\alpha \text{I}_2$ for some $\alpha \in \CCC$ such that $\alpha L_z \subset L_z$. Now if 
$z=z_x$ is a fixed point of an element $x$ of $B ^{+}$, 
by multiplying by an element of $\QQ ^\times$ we can assume that 
$x \in O_B ^+$, the set of elements in $O_B$ with positive reduced norm.  Write 
\[
\varphi(x)   = \begin{pmatrix} a & b \\ c &d\end{pmatrix},\quad x  \in 
 O_B ^{+}
\]
and let $\mu(x) = c z + d$. Then 
multiplication by $\mu(x)$ on $\CCC^2$ maps $L_z \to L_z$ because of
\[
\mu(x) \varphi (O_B)  \begin{pmatrix} z \\ 1 \end{pmatrix}  = 
 \varphi (O_B) (cz+d) \begin{pmatrix} \frac{a z + b}{cz+d} \\ 1 \end{pmatrix} =
  \varphi (O_B)  \begin{pmatrix} a & b \\ c &d\end{pmatrix}  \begin{pmatrix} z\\ 1 \end{pmatrix} 
  =  \varphi (O_B)\varphi(x) \begin{pmatrix} z\\ 1 \end{pmatrix}  = 
    \varphi (O_B x)\begin{pmatrix} z\\ 1 \end{pmatrix}  .
\] 
Since $x$ is not fixed by any parabolic elements in $B^+$ (if they exist), the map $x \mapsto \mu(x)$ is injective. We have shown: 
  \begin{prop}
  For all $z \in \mathfrak{H}$ there is an embedding
  \[
  \mathrm{Stab} _{O _B ^{+}} (z)\hookrightarrow  \mr{End} (A_z, i_z,  \Theta_z) 
\] by sending $x \in \mathrm{Stab} _{O _B ^{+}} (z)$ with $\varphi(x)   = \begin{pmatrix} a & b \\ c &d\end{pmatrix}$ to $\mu(x) = cz + d$. 
\end{prop} 
  
For a general point $z$ we have  $\mr{End} (A_z, i_z,  \Theta_z)  = \ZZ$. We say 
 $z$ is a CM point if there is an embedding $q: K \to B$, where $K$ is an 
 imaginary quadratic field, such that $z $ is the fixed point 
 of $q(K)\cap O_B^+$. An argument similar to the above shows that we have an embedding
 $K \subset  \mr{End}^0 (A_z, i_z,  \Theta_z).$
 Then  $\mr{End} (A_z, i_z,  \Theta_z) $ is an order in $K$. In fact
 $K$ is a subfield of $B$ which splits $B$. By Hasse's criterion, a quadratic field $K$ splits 
 $B$ if and only if every prime which ramifies in $B$ fails to split in $K$. 
 
\begin{prop}
If the QM abelian surface $(A, i, \Theta)$ has CM by the imaginary quadratic field $K$, 
then $A$ is isogenous to $E \times E$ where $E$ is an elliptic curve with CM by $K$. 
\end{prop}  
  
The space $V_{\QQ} = H^1 (A_z, \QQ)$ is a 4-dimensional $\QQ$-vector space, but because of the quaternion structure this becomes a 2-dimensional module over every quadratic subfield $F \subset B$. Fix a real subfield $F$; this defines an isomorphism $B \otimes _{\QQ}\RR = M_2 (\RR)$. 
  
Now suppose $( A_z, i_z,  \Theta_z)$ has CM by $K$. Because the endomorphisms in   
  $\mr{End} ^0 (A_z, i_z,  \Theta_z)$ commute 
  with $B$, they commute with $F$ and we have an embedding
  \[
  K  \subset\mr{End} ^0 (A_z, i_z,  \Theta_z) \to \mr{End}_F (V_{\QQ}) \sim M_2 (F).
  \] 
  One shows that the subalgebra of $M_2 (F)$ commuting with $K$ is $K$ itself. For $z$ corresponding to a point 
  in $X_B(\bar{\QQ})$, $A_z$ has a model defined over a number field, say, $k(z)$. 
  If $L$ is a number field such that the  $\mr{Gal} (\bar{L}/L)$ action on the $\ell$-adic Tate-module 
  $H^1 (A_z \otimes _{k(z)} \bar{\QQ}, \QQ _{\ell} )=V _{\QQ} \otimes _{\QQ} \QQ _{\ell}  = V _{\ell}$ commutes with 
  the induced action of $\mr{End}(A_z)$,  
 we get a representation
  \[
  \mr{Gal} (\bar{L}/L) \to (K \otimes \QQ _{\ell})^{\times} \subset
   \mr{GL} _{F \otimes \QQ _{\ell}} (V _{\QQ} \otimes _{\QQ} \QQ _{\ell}) = 
   \mr{GL} _{F \otimes \QQ _{\ell}} (V _{\ell} )  
  \sim \mr{GL}_2 (F \otimes \QQ _{\ell}).
  \] 
  
Because this is an abelian representation, this $\ell$-adic representation  is given by a pair of  characters conjugated by suitable liftings of Gal$(K/\QQ)$. In fact, these are the   characters giving the Galois action on the Tate module of the elliptic curve $E_z$ such that $A_z$ is isogenous to $E_z \times E_z$. 

We now describe the Atkin-Lehner involutions. Let $\Gamma$ be a finite index subgroup of $O_B^1$ 
  and  $\Gamma^*$ 
  be the normalizer of $\Gamma$ in $O_B^+$. For each $b \in \Gamma^*$ we normalize it by dividing by the positive square root of its reduced norm $n(b)$ so that $\tilde{b} := \frac{1}{\sqrt{n(b)}} b$ is an element in $(B \otimes_{\QQ} \RR)^\times$ with reduced norm 1. The collection of $\tilde b$ forms a group 
  $$ \Gamma^+ = \left\{ \tilde b = \frac{1}{\sqrt{n(b)}} b \mid b \in \Gamma^* \right \}$$
  which contains $\{\pm id\}\Gamma $ as a normal subgroup of finite index. In fact, the quotient $\Gamma^+/\{\pm id\}\Gamma$ is an elementary 2-group. Each non-identity element $\tilde b$ in $\Gamma^+/\{\pm id\}\Gamma$ is called an Atkin-Lehner involution. Sometimes, by abuse of language, we will also call $b$ an Atkin-Lehner involution. \footnote{The group $\Gamma^+/\{\pm id\}$ is isomorphic to the group 
  $N_B(\Gamma):= \{ b \in B ^+\mid b \Gamma = \Gamma b \}/\QQ^\times$, which contains $P(\Gamma) :=\Gamma/(\Gamma \cap \{\pm id\})$ as a normal subgroup. The nontrivial elements in the quotient $N_B(\Gamma)/P(\Gamma)$ are called Atkin-Lehner involutions traditionally. This agrees with our definition because $N_B(\Gamma)/P(\Gamma) \simeq \Gamma^+/(\{\pm id\}\Gamma)$.} The discussion of fixed points of Atkin-Lehner involutions can be found in Ogg's paper \cite{Ogg}.  
   
When $B$ does not split and $\Gamma = O_B^1$, the group $\Gamma^+/\Gamma \simeq (\ZZ/2\ZZ)^r$, where $r$ is the number of distinct primes dividing the discriminant $D_B$ of $B$. For each $d>1$ dividing $D_B$, there is a $b_d \in \Gamma^*$ with reduced norm $d$, and the Atkin-Lehner involutions of $O_B^1$ are $\tilde {b_d}$, $d | D_B$ with $d>1$. The action $b_d$ also has a moduli interpretation.   

In what follows we restrict our attention to the quaternion algebras $B$ pertaining to the groups considered in this paper. There are two cases: $B=B_6$ and $B=M_2(\QQ)$. The groups appearing in Theorem \ref{thm:traceformula} are contained in PSL$_2(\RR)$. For a subgroup $\Gamma$ of PSL$_2(\RR)$, denote by $\pm \Gamma$ its preimage in SL$_2(\RR)$. 

First consider the case $B= B_6=\left(\frac{-1,3}\QQ\right)$ with discriminant 6 and a maximal order $O_{B_6}=\ZZ+\ZZ I+\ZZ J+\ZZ\frac{1+I+J+IJ}2$, where $I^2=-1,J^2=3, IJ=-JI$. Fix the embedding $\varphi: B_6\otimes _{\QQ} \RR \to M_2(\RR)$ given by $\varphi(I) = \begin{pmatrix} 0 & -1\\1 & 0 \end{pmatrix}$ and $\varphi(J) = \begin{pmatrix}\sqrt 3 & 0 \\0 & -\sqrt 3 \end{pmatrix}$. 

The group $\pm(2,4,6)$ is generated by $s_2$, $s_4$ and $s_6$ given in \eqref{eq:B6-embed}. It comes from $\Gamma^+$ for $\Gamma = O_B^1 = \pm(2,2,3,3)$. The points $z = i, \frac{1+i}{\sqrt 3 + 1}, \frac{\sqrt {6} i}{3 + \sqrt 3}$ in the upper half-plane correspond to the $\QQ$-rational elliptic points on $X_{(2,4,6)}$ of order $4$, $6$, $2$, respectively. These are CM points stabilized by  $b_2=1+I$, $b_3=(3+3I+J + IJ)/2$ and $b_6 = 3I+IJ$ with respective reduced norm 2, 3, 6; these stabilizers induce the Atkin-Lehner involutions  $w_2, w_3, w_6$ on the curve $X_{(2,2,3,3)}$ explained before. The normalized stabilizers are the Atkin-Lehner involutions $ \tilde{b_2}, \tilde{b_3}, \tilde{b_6}$ in $\Gamma^+$ with respective order twice that of the elliptic points they fixed. Via $\varphi(\tilde{b_2}) = s_2$, $\varphi(\tilde{b_3}) = s_4$, $\varphi(\tilde{b_6}) = s_6$, they are viewed as elements in $\pm(2,4,6)$, and as such they generate the stabilizers of the respective elliptic points in $\pm(2,4,6)$. We proceed to examine the structure of the three abelian surfaces $A_z = \CCC^2 /\varphi (O_{B_6}) \begin{pmatrix} z \\1\end{pmatrix}$ in more detail. Below for each positive integer $n$, let $\zeta_n = e^{2\pi i/n}$.

\begin{itemize} 
\item[1.] $z=i$. Consider the sublattice of finite index $L = (\ZZ + \ZZ I) +  (\ZZ + \ZZ I) J \subset O_B$. It defines
an abelian variety $B_z := \CCC^2 /\varphi (L) \begin{pmatrix} z \\1\end{pmatrix}$ which is isogenous to $A_z$.  We claim that $B_z$ is isomorphic 
to $E_z \times E_z$, where $E_z$ is the elliptic curve $\CCC/\ZZ[i]$, which has 
complex multiplication by $\ZZ[i]$. To see this note the following 
\begin{align*}
\varphi((\ZZ + \ZZ I )) \begin{pmatrix} i \\1\end{pmatrix}  = (\ZZ + \ZZ i ) \begin{pmatrix} i \\1\end{pmatrix}
\\
\varphi((\ZZ + \ZZ I )J) \begin{pmatrix} i \\1\end{pmatrix}  = (\ZZ +\ZZ i )  \varphi(J)   \begin{pmatrix} i \\1\end{pmatrix}.
\end{align*}
The vectors  $\begin{pmatrix} i \\1\end{pmatrix}, \varphi(J)   \begin{pmatrix} i \\1\end{pmatrix}$ in $\CCC^2$ are $\CCC$-linearly independent. Therefore 
 $B_z$ is isomorphic to $E_z \times E_z$ where $E_z = \CCC/ \ZZ[i] $ is defined over $\QQ$ with CM by $\ZZ[i]$.  

In this case $b_2 = 1+I$ acts on $A_z$ and $B_z$ as multiplication by $\mu(b_2)=1+i = \sqrt{2} \zeta_8$. Consequently, for $\ell >2$, it induces an automorphism on the $\ell$-adic Tate module of $A_z$. As $H^1(A_z \otimes_{\QQ}\bar{\QQ}, \QQ_\ell)$ is isomorphic to two copies of $ H^1(E_z\otimes _{\QQ} \bar{\QQ}, \QQ_\ell)$ on which multiplication by $1+i$ acts as an automorphism, after extending the scalar field to $\QQ_\ell(\sqrt 2)$, we see that the action of $\tilde {b_2} = \frac{1}{\sqrt 2} (1+I)$ on $H^1(A_z \otimes_{\QQ}\bar{\QQ}, \QQ_\ell(\sqrt 2))$ is two copies of that on $ H^1(E_z\otimes _{\QQ} \bar{\QQ}, \QQ_\ell(\sqrt 2))$, an automorphism of order 8. 

\item[2.] $z= \frac{1+i}{\sqrt 3 + 1}$. Consider the sublattice of finite index $L = (\ZZ + \ZZ x) +  (\ZZ + \ZZ x) y \subset O_B$, 
where $x, y \in O_B$ are defined as follows: $x =  I+J+IJ$ and $y = J + IJ$. One checks that 
\[
x^2 = -3, \quad y^2 = 6, \quad xy + yx = 0.
\]
Note that $\ZZ + \ZZ x$ is isomorphic to $\ZZ[\sqrt{-3}]$ which is of index 2 in the ring of integers
of $\QQ[\zeta _3]$. At $z= \frac{1+i}{\sqrt 3 + 1}$, one finds
\[
\varphi(x) \begin{pmatrix} z \\1\end{pmatrix}  = \sqrt{-3} \begin{pmatrix} z\\ 1 \end{pmatrix}, \text{\ so \ } 
\varphi(\ZZ + \ZZ x ) \begin{pmatrix} z \\1\end{pmatrix}  = \ZZ[\sqrt{-3}]\begin{pmatrix} z \\1\end{pmatrix}.  
\]
Similarly one finds
\[
\varphi((\ZZ + \ZZ x )y) \begin{pmatrix} z \\1\end{pmatrix}  = \ZZ[\sqrt{-3}]
\varphi(y)
\begin{pmatrix} z \\1\end{pmatrix}.   
\]
The vectors $ \begin{pmatrix} z \\1\end{pmatrix} $ and  $ \varphi (y)\begin{pmatrix} z \\1\end{pmatrix} $ 
form a $\CCC$-linear basis of $\CCC^2$. Therefore 
the lattice $\varphi(L) \subset \CCC^2$ is isomorphic to $ (\ZZ + \ZZ \sqrt{-3}) \oplus  (\ZZ + \ZZ \sqrt{-3} )$, making it 
clear that the abelian surface $B_z := \CCC^2 /\varphi (L) \begin{pmatrix} z \\1\end{pmatrix}$ is isomorphic to $E_z^2$ where $E_z = \CCC/ \ZZ[\sqrt{-3}]$ is an elliptic curve over $\QQ$ with CM by $\ZZ[\sqrt{-3}]$. 

In this case $b_3=(3+3I+J + IJ)/2$ acts on $A_z$ and $B_z$ by multiplication by $\mu(b_3)=1+ \frac{1+\sqrt{-3}}{2} = \sqrt{3} \zeta_{12}$. By the same argument as in the previous case, for $\ell \ge 5$, the action of $\tilde {b_3}$ on the Tate module of $A_z$ over $\QQ_\ell(\sqrt{3})$ is two copies of that on the Tate module of $E_z$ over $\QQ_\ell(\sqrt{3})$, which is an automorphism of order 12. 

\item[3.]  $z = \frac{\sqrt {6} i}{3 + \sqrt 3}$. Define a sublattice of finite index $L = (\ZZ + \ZZ b_6) +  (\ZZ + \ZZ b_6) J \subset O_B$.  
Since $b_6 ^2 = -6$,  $\ZZ + \ZZ b_6$ is isomorphic to the 
maximal order $\ZZ[\sqrt{-6}]$ in the field $\QQ(\sqrt{-6})$. Note that $b_6 J = -Jb_6$.
At the  point $z = \sqrt{-6}/(3 + \sqrt{3})$ one finds
\[\
\varphi(\ZZ + \ZZ b_6 ) \begin{pmatrix} z \\1\end{pmatrix}  = \ZZ[\sqrt{-6}]
\varphi(b_6)\begin{pmatrix} z \\1\end{pmatrix}. 
\]
Similarly one finds
\[
\varphi((\ZZ + \ZZ b_6 )J) \begin{pmatrix} z \\1\end{pmatrix}  = \ZZ[\sqrt{-6}]
\varphi(J b_6)
\begin{pmatrix} z \\1\end{pmatrix}  .
\]
The vectors $\varphi(b_6) \begin{pmatrix} z \\1\end{pmatrix} $ and  
$ \varphi (J b_6)\begin{pmatrix} z \\1\end{pmatrix} $ 
form a $\CCC$-linear basis of $\CCC^2$ relative to which the lattice $L$ 
is isomorphic to $\ZZ [\sqrt{-6}]\oplus \ZZ [\sqrt{-6}]$.
This gives the isomorphism $B_z = \CCC^2 /\varphi (L) \begin{pmatrix} z \\1\end{pmatrix} \simeq E_z \times E_z$ where
$E_z$ is the elliptic curve $\CCC/\ZZ[\sqrt{-6}]$, which is defined over $\QQ(\sqrt 2)$ with CM by $\ZZ[\sqrt{-6}]$. 

In this case $b_6= 3I+IJ$ acts on $A_z$ and $B_z$ by multiplication by $\mu(b_6)=\sqrt{-6} = \sqrt{6} \zeta_4$. Similar to the previous two cases, for $\ell \ge 7$, 
the action of $\tilde {b_6}$ on the Tate module of $A_z$ over $\QQ_\ell(\sqrt{6})$ is two copies of that on the Tate module of $E_z$ over the same field, which is an automorphism of order 4. 
\end{itemize}

For $B = M_2(\QQ)$, 
we choose the standard maximal order $O_B = M_2(\ZZ)$ so that $O_B^1 = {\rm SL}_2(\ZZ)$. Using the basis $e_1 = \begin{pmatrix}1 & 0\\0 & 0\end{pmatrix}$, $e_2 = \begin{pmatrix}0 & 1\\0 & 0\end{pmatrix}$, $e_3 = \begin{pmatrix}0 & 0\\1 & 0\end{pmatrix}$, and $e_4 = \begin{pmatrix}0 & 0\\0 & 1\end{pmatrix}$ of $M_2(\ZZ)$, we see that, for each $z \in \mathfrak H$, 
$$O_B\begin{pmatrix}z\\1 \end{pmatrix} = \ZZ e_1\begin{pmatrix}z\\1 \end{pmatrix} + \ZZ e_2\begin{pmatrix}z\\1 \end{pmatrix} + \ZZ e_3\begin{pmatrix}z\\1 \end{pmatrix} + \ZZ e_4\begin{pmatrix}z\\1 \end{pmatrix} = \ZZ \begin{pmatrix}z\\0 \end{pmatrix} + \ZZ \begin{pmatrix}1\\0 \end{pmatrix} + \ZZ \begin{pmatrix}0\\z \end{pmatrix} + \ZZ \begin{pmatrix}0\\1 \end{pmatrix}.$$ 
This shows that the abelian surface $A_z = \CCC^2/O_B\begin{pmatrix}z\\1 \end{pmatrix}$ decomposes as $E_z^2$ with the elliptic curve $E_z= \CCC/(\ZZ\cdot z + \ZZ \cdot 1)$. Hence for elliptic modular groups we shall replace $A_z$ by $E_z$ and follow the classical moduli interpretation. 

The preimages in SL$_2(\RR)$ of the four elliptic modular groups in Theorem \ref{thm:traceformula} are $\pm(2, \infty, \infty) = \Gamma_0(2)$, $\pm(2,3, \infty) = {\rm SL}_2(\ZZ)$, $\pm(2, 4, \infty) = \Gamma_0(2)^+$, and $\pm(2,6,\infty) = \Gamma_0(3)^+$. The two elliptic modular groups in Theorem \ref{thm:traceformula-2F1} are subgroups of SL$_2(\ZZ)$. For each of them it is straightforward to compute stabilizers at elliptic points $z$ and determine the elliptic curves $E_z$ explicitly. In the table below, for each $\Gamma \subset {\rm SL}_2(\RR)$ considered in this paper, we summarize the information on its elliptic points $z$, the group $\langle T_z \rangle$ of stabilizers in $\Gamma$ of $z$ and the order of $T_z$. At each elliptic point $z$, $A_z$ is isogenous to $E_z\times E_z$ for an elliptic curve $E_z = \CCC/L_z$.  For each group except (2,4,6), the lattice $L_z=\Z z \oplus \Z$; for $(2,4,6)$,  $L_z=\Z[\sqrt{-6}],$ $\Z[\sqrt{-1}],$ $\Z[\sqrt{-3}]$  at elliptic points of orders 2, 4, 6 respectively, as computed above. Let $K_z=L_z\otimes_\Z\QQ$ be the CM field of $E_z$. We also record $\mu(z)$, the action of $T_z$,  before normalization by a scalar, on $E_z$ as complex multiplication, and $\kappa_z$, the field of definition of $E_z$.
   
\begin{table}[h]
  $$
 \begin{array}{|c|c|c|c|c|c|c|c|c|c|c|c|}
\hline
\G&  \G_1(3)& \G_0(2) &{\rm SL}_2(\ZZ)& \G_0(2)^+& \G_0(3)^+& \pm (2,4,6)\\
\hline
z &   \frac{3+\sqrt{-3}}6 & \frac{1+i}2 & i, \frac{-1+\sqrt{-3}}2 & \frac {i}{\sqrt {2}}, \frac{1+i}2& \frac {i}{\sqrt {3}}, \frac{3+\sqrt{-3}}6& \frac{\sqrt{-6}}{3+\sqrt 3},i, \frac{1+i}{1+\sqrt 3}\\\hline
 T_z &\SM1{-1}3{-2}&\SM1{-1}2{-1}&
 \begin{array}{c}\SM01{-1}0\\\SM0{-1}11\end{array}
 & 
 \begin{array}{c} \frac1{\sqrt 2}\SM0{1}{-2}0\\ {\small \sqrt 2\SM1{1/2}{-1}{0}} \end{array}&\begin{array}{c}\frac1{\sqrt 3}\SM0{1}{-3}0\\ {\small \sqrt 3 \SM1{1/3}{-1}{0}}\end{array}&\begin{array}{c}s_2, s_4, s_6 \\
 cf.\, \eqref{eq:B6-embed}
 \end{array}\\\hline
\# \langle T_z\rangle & 3& 4&4, 6&4, 8&4,12 &4,8,12 \\\hline
\mu(z) &\zeta_3&i&i, \zeta_6&-\sqrt{-2}, -(1+i) &-\sqrt{-3},  -(1 + \zeta_6) 
& \sqrt{-6}, 1+i, 1+\zeta_6 \\\hline
\kappa_z &\QQ& \QQ &\QQ,\QQ   & \QQ,\QQ  &\QQ,\QQ  & \QQ(\sqrt 2), \QQ, \QQ \\\hline
 K_z\cdot \kappa_z &\QQ(\sqrt{-3})& \QQ(i) &\begin{array}{c}\QQ(i), \\ \QQ(\sqrt{-3}) \end{array}  & \begin{array}{c}\QQ(\sqrt{-2}), \\\QQ(i)\end{array}  &\begin{array}{c}\QQ(\sqrt{-3}), \\\QQ(\sqrt{-3})\end{array}  & \begin{array}{c}\QQ(\sqrt 2,\sqrt{-6}),\\ \QQ(i), \QQ(\sqrt{-3}) \end{array}\\\hline
\end{array}
$$
\caption{CM points information}
\label{tab:8}

\end{table}

Observe that the CM field $K_z$ has class number 1 except at the elliptic point $z$ of order 2 for the group (2,4,6), in which case it has class number 2. 

The interested reader may consult \cite{Baba-Granath-g2} by Baba and Granath for explicit models of genus-two curves with QM at points  not admitting CM by $-3$ or $-4$,  as well as the discussion of the fields of moduli. For example, at the elliptic point $z$ of order $2$ for (2,4,6) admitting CM by $\QQ(\sqrt{-24})$, they give a curve $C$ with $\mbox{Jac}(C)\simeq E_{z,1}\times E_{z,2}$ defined over $\QQ(\sqrt 2)$. 

\subsection{Contributions from elliptic points}
\label{SS:singular2} Let $\G$ be one of the groups considered in Theorems \ref{thm:traceformula} and \ref{thm:traceformula-2F1} and let $z$ be an elliptic point of $\G$ of order $N_z$. (The group $(\infty, \infty, \infty)$ in Theorem \ref{thm:traceformula-2F1} is the only group containing no elliptic points.) Then $z$ is a $\QQ$-rational point on the curve $X_\G$. Fix a prime $\ell$. In this subsection we compute the traces of the Frobenius elements $\Frob_p$ for primes $p \ge 7$ on  
 $V^{k}(   \Gamma )_{\bar {z}, \ell } $, the stalk at $z$ of the automorphic sheaf on $X_\G$, for integers $k \ge 1$.  Since the groups $\G$ in Theorem \ref{thm:traceformula} are projective, only  even $k$ will be considered, so we shall not distinguish $\G$ and $\pm \G$. In the discussion below, we assume $\G$ and $z$ are as listed in Table \ref{tab:8} and will use notation from that table. Recall that the abelian surface $A_z$, defined over $\kappa_z$, is isogenous to $E_z \times E_z$ for an elliptic curve $E_z$ defined over $\kappa_z$. Denote by $\sigma_{z,\ell}$ the Tate module on $E_z$ over $\overline{\QQ_\ell}$. It is a degree-2 representation of $G_{\kappa_z}$. 

The elliptic curve $E_z$, defined over the real field $\kappa_z$, has CM by the imaginary quadratic field $K_z$. Thus there is an $\ell$-adic  character $\xi_z$ of the Galois group $G_{K_z\kappa_z}$ such that $\sigma_{z, \ell}$ is the induced representation ${\rm Ind}_{G_{K_z\kappa_z}}^{G_{\kappa_z}} \xi_z$. Let $X$ be a vector spanning the line on which $G_{K_z\kappa_z}$ acts via $\xi_z$. The complex conjugation $c \in G_\QQ$ lies in $G_{\kappa_z}$ but not $G_{K_z\kappa_z}$. Let $Y = (\sigma_{z, \ell}(c))X$. Then $X$ and $Y$ form a basis of $\sigma_{z,\ell}$ since $G_{\kappa_z}= G_{K_z\kappa_z} \cup G_{K_z\kappa_z}c$ and $\sigma_{z, \ell}$ is 2-dimensional and irreducible. Therefore we can represent $\sigma_{z,\ell}(g)$ for $g \in G_{\kappa_z}$ by a $2 \times 2$ matrix in the basis $X, Y$ as follows:

$$\sigma_{z,\ell}(g) = \begin{pmatrix} \xi_z(g) & \\ & \xi_z^c(g)\end{pmatrix} \quad {\rm if}~ g \in G_{K_z\kappa_z}, $$
 and
$$\sigma_{z,\ell}(c) = \begin{pmatrix} & \xi_z(c^2)\\1 & \end{pmatrix} = \begin{pmatrix} & 1\\1 & \end{pmatrix}$$ since $c^2 = id$. Here $\xi_z^c : g \mapsto \xi_z(c^{-1}gc)$ is the conjugate of $\xi_z$ by $c$. In particular we get the decomposition $\sigma_{z,\ell}|_{G_{K_z\k_z}} = \xi_z \oplus \xi_z^c$. As $\sigma_{z,\ell}$ is a Tate module of $E_z$, we have $\det \sigma_{z, \ell}= \epsilon_\ell|_{G_{\kappa_z}}$. Therefore $\xi_z \cdot \xi_z^c = \epsilon_\ell|_{G_{K_z\kappa_z}}$. 

Let $\p$ be a nonzero prime ideal of $\kappa_z$ unramified in $K_z\kappa_z$ and let $\P$ be a prime ideal of $\kappa_zK_z$ above $\p$. As a consequence of the above, we obtain the following action of $\Frob_\p$ with respect to the basis $X, Y$: 
\smallskip

\noindent (5.2.1) If $\p$ splits in $K_z\kappa_z$, then 
$$\sigma_{z, \ell}(\Frob_\p)X = a_\p X,~~ \sigma_{z, \ell}(\Frob_\p)Y = d_\p Y, ~~a_\p d_\p = N\p.$$
In fact, $(a_\p, d_\p) = (\xi_z(\Frob_\P), \xi_z^c(\Frob_\P))$ or $ (\xi_z^c(\Frob_{\P}), \xi_z(\Frob_\P))$. Note that $\xi_z^c(\Frob_\P) = \xi_z(\Frob_{\P^c})$.

\smallskip

\noindent (5.2.2) If $\p$ is inert in $K_z\kappa_z$, then
$$\sigma_{z, \ell}(\Frob_\p)X = c_\p Y, ~~\sigma_{z, \ell}(\Frob_\p)Y = b_\p X, ~~c_\p b_\p = -N\p.$$

\smallskip

As the trace and determinant of $\sigma_{z,\ell}(\Frob_\p)$ are $\ZZ$-valued, upon fixing an isomorphism $\iota: \overline{\QQ_\ell} \to \CCC$, we shall regard $\xi_z(\Frob_\P)$ as an algebraic integer and $\xi_z^c(\Frob_\P)$ its complex conjugation. When $\p$ is inert in $\kappa_zK_z$, we have $\xi_z(\Frob_\P) = \xi^c_z(\Frob_\P) = -N(\p)$. This identifies $\xi_z$ with a Hecke Grossencharacter of $\kappa_zK_z$ which takes values in the ring of integers of the field $\kappa_zK_z$. 

On the other hand, as explained in \S5.1, the stabilizer $T_z$, before normalization by a scalar to make $\det T_z= 1$, acts on $E_z$ via multiplication by $\mu(z) \in K_z$ given in Table \ref{tab:8}. It follows from the CM theory that CM by $\mu(z)$ on $E_z$ commutes with the action of the Galois group $G_{K_z\kappa_z}$ on $E_z$. This is carried over to the space $\langle X, Y \rangle$ of $\sigma_{z,\ell}$. Consequently, $X$ and $Y$ are eigenvectors of $T_z$. We know $T_z$ has one eigenvalue  $\mu(z)/|\mu(z)| =: \zeta_z$, which is a primitive $M_z$th root of unity. Here $M_z$ is the order of $T_z$, equal to $N_z$ for $\Gamma = \G_1(3) =(3, \infty, \infty)$, and $2N_z$ otherwise. Since $\det T_z = 1$, the other eigenvalue of $T_z$ is $\zeta_z^{-1}$. Without loss of generality, we may assume
\[
T_z X = \zeta_z X \quad {\rm and}\quad  T_z Y = \zeta_z^{-1} Y.
\]

As explained in \S2.2.2, the sheaf $V^{k}(   \Gamma )_{\bar {z}, \ell } $ is defined as the pushforward of a sheaf $V^{k}(   \Gamma' )_{\bar {z}, \ell }$,  where $\Gamma'$ is a torsion-free normal subgroup of $\Gamma$ of finite index. We first consider the groups $\G$ in Theorem \ref{thm:traceformula} for which $k \ge 2$ is even. It follows from Proposition \ref{P:ALLL2} and the discussion in the previous subsection that $V^{k}(   \Gamma' )_{\bar {z}, \ell }= \mr{Sym}^{k} (\eta_{z, \ell})$, where $\eta_{z,\ell}$ is the restriction of $\sigma_{z,\ell}$ to a suitable Galois subgroup of $G_{\kappa_z}$ depending on the moduli interpretation of $X_{\Gamma'}$. As the stabilizers of $z$ in $\Gamma$ form the cyclic group $\langle T_z\rangle$ of finite order, the space of $V^{k}(   \Gamma )_{\bar {z}, \ell } $ is $\mr{Sym}^{k}(\sigma_{z, \ell})^{T_z}$. We shall describe the action of $G_\QQ$ on this space by computing the trace of $\Frob_p$ for primes $p \ge 7$. 

In view of the action of $T_z$ described above, the space $\mr{Sym}^{k}(\sigma_{z, \ell})^{T_z}$ has a basis 
\begin{eqnarray*}
\{X^a Y^b \mid \text{
 $a+b=k$, $a, b \ge 0$ and  $a\equiv b $ mod $M_z$}\}\\
 = \{(XY)^{k/2}X^{rN_z}Y^{-rN_z} \mid ~r \in \ZZ,~ -\frac{k}{2N_z} \le r \le \frac{k}{2N_z} \}.
\end{eqnarray*}
In particular, when $k=2$, the space $\mr{Sym}^{2}(\sigma_{z,\ell} )^{T_z}$ is 1-dimensional, spanned by $XY$. 

Now we proceed to compute the trace of $\Frob_p$ on $\mr{Sym}^{k}(\sigma_{z, \ell})^{T_z}$. We shall split the discussion into two cases.

Case (I). $\kappa_z = \QQ$, which holds for all $z$ except the order 2 elliptic point for $\G=(2,4,6)$. It follows from (5.2.1) and (5.2.2) that $G_\QQ$ acts on $\langle XY \rangle$ via the character $\chi_{d_z}\epsilon_\ell$, where $d_z$ is the discriminant of $K_z$ as given in Table 1. Furthermore, if $p$ is inert in $K_z$, then, by (5.2.2), ${\rm Sym}^{k}\sigma_{z,\ell}(\Frob_p)$ swaps $(XY)^{k/2}X^{rN_z}Y^{-rN_z}$ and $(XY)^{k/2}Y^{rN_z}X^{-rN_z}$ so that $(XY)^{k/2}$ is the only invariant vector in the chosen basis. Consequently the trace of the $\Frob_p$ action on $\mr{Sym}^{k}(\sigma_{z, \ell})^{T_z}$ is $(-p)^{k/2}$ for $p$ inert in $K_z$. If $p$ splits in $K_z$, then, by (5.2.1), each basis vector $(XY)^{k/2}X^{rN_z}Y^{-rN_z}$ is an eigenvector of ${\rm Sym}^{k}\sigma_{z,\ell}(\Frob_p)$ with eigenvalue 
$$p^{k/2} (\xi_z(\Frob_\P)\xi_z^c(\Frob_\P)^{-1})^{\pm rN_z} = p^{k/2} (\xi_z(\Frob_\P)^2/p)^{\pm rN_z},$$ 
where $\P$ is a prime ideal of $K_z$ dividing $p$ and the sign in front of $r$ is independent of $r$. The second expression above uses  $\xi_z(\Frob_\P)\xi_z^c(\Frob_\P) = N(\P)=p$. Summing over $-\frac{k}{2N_z} \le r \le \frac{k}{2N_z}$ gives the trace of $\Frob_p$ on $\mr{Sym}^{k}(\sigma_{z, \ell})^{T_z}$ as 
\begin{eqnarray}\label{eq:traceatz}
 \sum_{-\frac{k}{2N_z} \le r \le \frac{k}{2N_z}}p^{k/2} (\xi_z(\Frob_\P)^2/p)^{rN_z},
 \end{eqnarray}independent of the choice of $\P$ and the choice of $\xi_z$ or $\xi_z^c$. 

Note that the CM field $K_z$ as class number 1. By the CM theory, when $(p)=\P \P^c$ splits in $K_r$, the value $\xi_z(\Frob_\P)$ generates the ideal $\P$. It differs from an arbitrary generator $\alpha_{z,p}$ of $\P$ by a unit of $K_z$, which has order dividing $2N_z$ for all $z$ with $\kappa_z = \QQ$ except for the elliptic point of order $2$ for the group $(2,6,\infty)$. So for the non-exceptional $z$ we may replace $\xi_z(\Frob_\P)$ in \eqref{eq:traceatz} by any generator $\alpha_{z,p}$ of a prime ideal of $K_z$ above a splitting $p$, as described in Theorems \ref{thm:traceformula} and \ref{thm:traceformula-2F1}. The exceptional $z$, namely the order 2 elliptic point of $(2,6,\infty)$, has coordinate $\lambda(z)= 1$ as given in Theorem \ref{thm:traceformula}. It is the image of the unique point with coordinate $t = 1/2$ on the 2-fold cover $X_{(3, \infty, \infty)}$ of $X_{(2,6,\infty)}$ under the projection $\lambda = -4(t-1)t$. The elliptic curve 
at the point $t=1/2$ on $X_{(3, \infty, \infty)}$ is $E_{(3, \infty, \infty), 1/2}$ defined by $x^3 = y^2 + xy + \frac{1/2}{27}y$ as given in Table \ref{tab:EC-a}. It is isogenous to $E_z$. So we have ${\rm Tr}\sigma_{z, \ell}(\Frob_p) = H_p(\{1/3, 2/3\}, \{1,1\}; 1/2)$ by Proposition \ref{prop:det-I}. Therefore at primes $p$ splitting in $\QQ(\sqrt{-3})$, $\alpha_{z,p}=\xi_z(\Frob_\P)$ or $\xi_z^c(\Frob_\P)$ is a root of
$$ x^2 - H_p(\{1/3, 2/3\}, \{1,1\}; 1/2)x + p = 0.$$
The  hypergeometric datum $\{\{1/3, 2/3\}, \{1,1\}\}$ is denoted by $HD_2(\G)$ for $\G = (2,6,\infty)$ in \S4.5. Corollary 3 in \S4.6 states that $H_p(HD_2(\G); 1/2)^2 = H_p(HD(\G), 1) + 2p$ for $p$ splitting in $\QQ(\sqrt{-3})$. In other words, at primes $p$ splitting in $K_z$, 
$\alpha_{z,p}^2=\xi_z(\Frob_\P)^2$ or $\xi_z^c(\Frob_\P)^2$ is a root of
$$ x^2 - H_p(HD(\G); 1)x + p^2 = 0.$$
In fact, the above expression holds at the elliptic point $z$ of order $N_z=2$ for $\G = (2, \infty, \infty), (2,3,\infty), $ $(2,4, \infty),$ and $(2,6,\infty)$ by the same corollary. 

We have shown that the contributions of Frobenius traces at elliptic points $z$ of $\Gamma$ with $\kappa_z = \Q$ for even weight are as described in Theorems \ref{thm:traceformula} and \ref{thm:traceformula-2F1}.   

Case (II). $\kappa_z \ne \QQ$, which occurs only at the order $N_z=2$ elliptic point $z$ for $(2,4,6)$. In this case $\kappa_z = \QQ(\sqrt 2)$, $K_z = \QQ(\sqrt{-6})$ and $\kappa_zK_z = \QQ(\sqrt 2, \sqrt{-6})$ as listed in Table \ref{tab:8}. Let $\tau \in G_\QQ$ be such that $\tau(\sqrt 2) = -\sqrt 2$ and $\tau(\sqrt{-6}) = \sqrt{-6}$. Note that $G_{\QQ(\sqrt{-6})} = G_{\QQ(\sqrt 2, \sqrt{-6})} \cup G_{\QQ(\sqrt 2, \sqrt{-6})} \tau$ and $G_\QQ = G_{\QQ(\sqrt{-6})} \cup G_{\QQ(\sqrt{-6})}c$. As $\tau$ gives rise to an isogeny from $E_z$ to $\tau(E_z)$ over $\Q(\sqrt 2)$, the representation $\sigma_{z,\ell}$ is isomorphic to the conjugate $\sigma_{z,\ell}^\tau$. (The interested reader may refer to \cite{Baba-Granath} by Baba and Granath for algebraic models for $A_z$, $E_z$ and $\tau(E_z)$.) 
This implies that $\sigma_{z, \ell}$ can be extended to $G_\QQ$. Also $\xi_z = \xi_z^\tau$ can be extended to $G_{\QQ(\sqrt{-6})}$. There are two extensions of $\sigma_{z,\ell}$ to $G_\QQ$; they differ by the quadratic character $\chi_2$, hence they induce the same action on the space of $V^{k}(\G)_{\bar z, \ell}$ for even $k\ge 2$. 

Let $\sigma_z'$ be an extension of $\sigma_{z, \ell}$ to $G_\QQ$. By Clifford theory (cf. \cite{LLL}), $\sigma'_z = {\rm Ind}_{G_{\QQ(\sqrt{-6})}}^{G_\QQ} \psi_z$ for a character $\psi_z$ of $G_{\QQ(\sqrt{-6})}$ extending $\xi_z$. 
 Since $\det \sigma'_z$ extends $\det \sigma_{z,\ell} = \epsilon_\ell|_{G_{\QQ(\sqrt 2)}}$, we have $\det \sigma_z' = \epsilon_\ell$ or $\chi_2 \epsilon_\ell$. Its restriction to $G_{\QQ(\sqrt{-6})}$ is $\psi_z\psi_z^c$. 

To proceed, we need to determine the determinant of $\sigma_z'$, which is independent of the extension of $\sigma_z$ we choose.

\begin{lemma}\label{lem:detsigma'} 
$$\det \sigma_z' = \epsilon_\ell \quad and \quad \psi_z\psi_z^c = \epsilon_\ell|_{G_{\QQ(\sqrt{-6})}}.$$
\end{lemma}

\begin{proof} Since $\chi_2$ is nontrivial on $G_{\QQ(\sqrt{-6})}$ and $\det \sigma_z'|_{G_{\QQ(\sqrt{-6})} } =\psi_z\psi_z^c$, the two statements are equivalent. We shall prove $\det \sigma_z' = \epsilon_\ell$. 

Let $\Gamma'$ be a normal torsion-free subgroup of $\pm (2,4,6)$ of finite index so that ${\rm Sym}^2(\G)_\ell$ is the pushforward from ${\rm Sym}^2(\G')_\ell$. Then so are their respective determinant $\rho$ and $\rho'$ as rank-1 sheaves on the curves $X_{(2,4,6)}$ and $X_{\G'}$ respectively. On the other hand, we know that $\rho'= \epsilon_\ell^3$ by Proposition \ref{P:ALLL2}, the same holds for the pushforward $\rho$ at nonelliptic and hence all points on $X_{(2,4,6)}$ by the comparison theorem (Theorem \ref{thm:compare}). This shows that $\rho$ restricted to the fiber at $z$, which is $\det \sigma_z'$, is $\epsilon_\ell$, as claimed.
\end{proof}

A similar argument as before gives the following action of $\Frob_p$ for primes $p \ge 5$:

\smallskip

\noindent (5.2.3) If $p$ splits in $K_z=\QQ(\sqrt{-6})$, then 
$$\sigma_{z}'(\Frob_p)X = a_p X,~~ \sigma_{z}'(\Frob_p)Y = d_p Y,~~a_p d_p = p.$$
More precisely, let $\P$ be a prime ideal of $K_z$ dividing $p$. Then $(a_p, d_p) = (\psi_z(\Frob_\P), \psi_z^c(\Frob_\P))$ or $ (\psi_z^c(\Frob_{\P}), \psi_z(\Frob_\P))$. In both cases, we have $a_pd_p = \psi_z\psi_z^c(\Frob_\P) = \det \sigma_z'(\Frob_p) = p$. 

\smallskip

\noindent (5.2.4) If $p$ is inert in $K_z$, then
$$\sigma_{z}'(\Frob_p)X = c_p Y, ~~\sigma_{z}'(\Frob_p)Y = b_p X, ~~c_p b_p = -p.$$

In all cases, we have
\begin{cor}\label{cor:XY}
Let $\G$ be a group in Theorem \ref{thm:traceformula} and $z$ be an elliptic point of $\G$. Then $G_\QQ$ acts on the space ${\rm Sym}^2(\sigma_{z,\ell})^{T_z} =\langle XY \rangle$ via the character $\chi_{d_z}\epsilon_\ell$, where $d_z$ is the discriminant of the CM field $K_z$ listed in Table \ref{tab:CM-fields}. 
\end{cor}

Similar computations as in Case (I) give similar results for Case (II). More precisely, at primes $p$ inert in $K_z = \QQ(\sqrt{-6})$, the trace of $\sigma_z'(\Frob_p)$ on $\mr{Sym}^{k}(\sigma_{z, \ell})^{T_z}$ is $(-p)^{k/2}$. When $p$ splits in $K_z$, the trace of $\sigma_z'(\Frob_p)$ on $\mr{Sym}^{k}(\sigma_{z, \ell})^{T_z}$ is
\begin{eqnarray}\label{eq:traceatz=2}
 \sum_{-\frac{k}{2N_z} \le r \le \frac{k}{2N_z}} p^{k/2} (\psi_z(\Frob_\P)^2/p)^{rN_z},
 \end{eqnarray}independent of the choice of a degree-1 prime ideal $\P$ of $K_z$ above $p$ and the choice of $\psi_z$ or $\psi_z^c$. In this case the field $K_z =\QQ(\sqrt{-6})$ has class number $2$ so that the square of every integral ideal of $K_z$ is principal. Let $\mathfrak R$ be a prime ideal of $\kappa_zK_z$ above $\P$, and let $\mathfrak Q$ be the prime ideal of $\kappa_z$ divisible by $\mathfrak R$. If $\P$ splits in $\kappa_zK_z$, then $\psi_z(\Frob_\P) = \xi_z(\Frob_{\mathfrak R}) = \xi_z^\tau(\Frob_{\mathfrak R})$ so that $\psi_z(\Frob_\P)^2 = \xi_z(\Frob_{\mathfrak R})^2$. In this case, $\psi_z(\Frob_\P)^2$, being in $\mathfrak R^2$ and invariant under $\tau$, lies in $\P^2$. Furthermore, its norm is $\xi_z(\Frob_{\mathfrak R})^2\xi_z^c(\Frob_{\mathfrak R})^2 = p^2$, which is the norm of $\P^2$. Therefore $\psi_z(\Frob_\P)^2$ generates $\P^2$. Since the units of $\ZZ[\sqrt{-6}]$ are $\pm 1$ and $N_z=2$, we may replace $\psi_z(\Frob_\P)^2$ by any generator $\alpha_{z,p}$ of $\P^2$. If $\P$ is inert in $\kappa_zK_z$, then $p$ is inert in $\kappa_z$ and $\mathfrak Q = (p)$ splits in $\kappa_zK_z$. So $\psi_z(\Frob_\P)^2 = \xi_z(\Frob_{\mathfrak R})$ has norm $p^2$ and it generates $\P^2$. By a similar argument, we may replace $\psi_z(\Frob_\P)^2$ by any generator $\alpha_{z,p}$ of $\P^2$. On the other hand, as pointed out in Remark \ref{remark:CM}, $\rho_{HD(2,4,6),1,\ell}^{BCM}$ is the $\ell$-adic representation of $G_\QQ$ attached to the weight-3 modular form $f_{24,3,h,a}$, which has $\ZZ$-coefficients and CM by $\QQ(\sqrt{-6})=K_z$. Hence $\rho_{HD(2,4,6),1,\ell}^{BCM}$ is induced from an $\ell$-adic character $\nu$ of $G_{K_z}$ such that above primes $p$ splitting in $K_z$, $\nu(\Frob_\P)^2$ generates $\P^2$. Hence $\psi_z(\Frob_\P)^2 = \pm \nu(\Frob_\P)^2$. As $\Tr \rho_{HD(2,4,6),1,\ell}^{BCM}(\Frob_p) = pH_p(HD(2,4,6), 1)$, we may choose $\alpha_{z,p}$ to be a root of
 $$x^2 - \left(\frac{-3}{p}\right)pH_p(HD(2,4,6), 1) x + p^2 = 0,$$
as asserted in Theorem \ref{thm:traceformula}.
\medskip

Summarizing the above discussions, we have proved that the  contributions from elliptic points of the groups in Theorem \ref{thm:traceformula} for $k$ even are as stated.
\bigskip

It remains to consider the unique elliptic point $z$ of order $N_z=3$ for $\G = (3, \infty, \infty) = \G_1(3)$ in Theorem \ref{thm:traceformula-2F1}. We have $\kappa_z = \QQ$ and $K_z = \QQ(\sqrt{-3})$. Since $\G$ does not contain $-I_2$, the sheaf $V^k(\G)_\ell$ is defined as the pushforward of $V^k(\G')_\ell$ and $V^k(\G')_\ell = {\rm Sym}^k(V^1(\G')_\ell)$. Therefore the space of $V^k(\G)_{\bar z, \ell}$ is ${\rm Sym}^k(\sigma_{z,\ell})^{T_z}$, which has a basis (for $k \ge 1$)
\begin{eqnarray*}
\{X^a Y^b \mid \text{
 $a+b=k$, $a, b \ge 0$ and  $a\equiv b $ mod $3$}\} \\
 = \{X^rY^{k-r} \mid~~ r \in \ZZ, ~k \ge r \ge 0, ~{\rm and}~ 2r \equiv k \mod 3\}\end{eqnarray*} since $T_z$ has order 3. Write $2r-k = 3m$. Then $m \equiv k \mod 2$, $-\frac{k}{3} \le m \le \frac{k}{3}$, $r = (k+3m)/2$ and $k-r = (k-3m)/2$. We rewrite the above basis as
 $$\{X^{(k+3m)/2}Y^{(k-3m)/2} \mid ~~ m \equiv k \mod 2, -\frac{k}{3} \le m \le \frac{k}{3} \}.$$

At a prime $p$ inert in $K_z$, $\sigma_z(\Frob_p)$ swaps $X^{(k+3m)/2}Y^{(k-3m)/2}$ and $X^{(k-3m)/2}Y^{(k+3m)/2}$ up to scalar. This shows that the trace of $\sigma_z(\Frob_p)$ on   ${\rm Sym}^k(\sigma_{z,\ell})^{T_z}$ is 0 for $k$ odd and $(-p)^{k/2}$ for $k$ even. For a prime $p$ split in $K_z$, let $\p$ be a prime ideal of $K_z$ above $p$. Because of the symmetry in $m$, we may assume $\sigma_z(\Frob_p)X=\xi_z(\Frob_\p)X$ so that
\begin{eqnarray*}
\sigma_z(\Frob_p)X^{(k+3m)/2}Y^{(k-3m)/2} &=& \xi_z(\Frob_\p)^{(k+3m)/2}\xi_z^c(\Frob_\p)^{(k-3m)/2}X^{(k+3m)/2}Y^{(k-3m)/2} \\
&=& p^{(k-3m)/2}\xi_z(\Frob_\p)^{3m}X^{(k+3m)/2}Y^{(k-3m)/2}.
\end{eqnarray*}
This shows that the trace of $\sigma_z(\Frob_p)$ on ${\rm Sym}^k(\sigma_{z,\ell})^{T_z}$ is
$$ \sum_{m \equiv k ~{\rm mod}~ 2, -\frac{k}{3} \le m \le \frac{k}{3}} p^{(k-3m)/2}\xi_z(\Frob_\p)^{3m}.$$
By the CM theory, $\xi_z(\Frob_\p)$ is a generator of the ideal $\p$, which is unique up to sixth roots of unity. When $k$ is even, so is $m$ and we can replace $\xi_z(\Frob_p)$ in the above formula by any generator $\alpha_{z,p}$ of $\p$. While for $k$ odd, so is $m$, then $\alpha_{z,p}^{3m}$ agrees with $\xi_z(\Frob_\p)^{3m}$ up to sign. 
To determine $\xi_z(\Frob_\p)^3$, recall that every elliptic curve defined over $\QQ$ with a $\QQ$-rational 3-torsion point  and vanishing $j$-invariant is isomorphic to either $x^3=y^2+ay$ for some $a\in \QQ$ or $x^3=y^2-(\frac{1}{2^4\cdot3\cdot \sqrt{-3}})^2$ which is isomorphic to the curve above $\frac{1+\sqrt{-3}}2$ on $X_1(3)$.  In either case, for each prime $p\equiv 1\mod 3$ where the elliptic curve has a good reduction, $\xi_z(\Frob_\p)^3=-J_\omega(1/3,1/3)^3$. 

\begin{prop}\label{prop:elliptic}
The contributions of the Frobenius traces at the elliptic points for groups in Theorems \ref{thm:traceformula} and \ref{thm:traceformula-2F1} are as stated.
\end{prop}
 
\subsection{Contributions from cusps}\label{ss:5.3} In this section we discuss the contribution from a cusp $z$ of a group $\Gamma$ in Theorems \ref{thm:traceformula} and \ref{thm:traceformula-2F1}. 
Recall that $V^k(\Gamma)_\ell = {\rm Sym}^k(V^1(\Gamma)_\ell)$. The inertia subgroup at $z$, denoted $I_z$, is cyclic generated by an element with Jordan form $\begin{pmatrix}1&1\\0& 1\end{pmatrix}$ or $\begin{pmatrix} -1&1\\0&-1\end{pmatrix}$. As explained in \S2.2.1, the $G_\QQ$ action on the stalk $V^1(\Gamma)_{\bar z, \ell}$ arises from the action on the $I_z$-invariant space of a degree-2 representation $\rho$ of the \'etale fundamental group of the curve $X_\Gamma^\circ$ at the generic point. Let $\{X, Y\}$ be a basis of the space of $\rho$. Then the space of $V^k(\Gamma)_{\bar z, \ell}$ is 1-dimensional,  spanned by $Y^k$, if 1 is an eigenvalue of $I_z$; it is $\langle Y^k \rangle$ for $k$ even and $0$ for $k$ odd if $-1$ is an eigenvalue of $I_z$.  

The contribution at $z$ is equal to 1 for $k$ even, following the same argument and computation as in Scholl \cite{Sch88}.  Hence we assume $k$ odd and the group $\Gamma = \G_1(3)$ or $\G_1(4)$. 

In Table \ref{tab:EC-a} an explicit model for the \lq\lq universal elliptic curve\rq\rq $E_{\Gamma, t}$ with $t$ running through the open curve $X_\Gamma^\circ$ for $\Gamma \in \{\G_1(3), \G_1(4)\}$ is given. The two cusps of $\G_1(3)$ and the three cusps of $\G_1(4)$ all have coordinates $t(z) \in \{0, 1, \infty\}$.  
 The contribution from $z$ is determined by the reduction type of the curve $E_{\Gamma, t(z)}$ at its unique singular point. 
  More precisely, the trace of $\Frob_p$ on the 1-dimensional space $V^1(\Gamma)_{\bar z, \ell}$ is 1, $-1$, or $0$ according as  $E_{\Gamma, t(z)}$ mod $p$ being split multiplicative reduction, nonsplit multiplicative reduction, or additive reduction. 
  In the case of $\G_1(3)$, when $t=0$, the curve $E_{\G_1(3), 0}$, defined by $y^2+xy = x^3$,  
 has two distinct tangents over $\QQ$ at the singular point $(0, 0)$.  At $t=1$,  the singular point is $(-1/9,1/27)$, which can be shifted to the origin $(0,0)$ in the equation $y^2+xy=x^3-x^2/3$. Hence,  the curve has multiplicative reduction at primes $p \ne 3$, which splits  if and only if $-3$ is a square mod $p$. This proves case (2) in Theorem \ref{thm:traceformula-2F1} for $\G = \G_1(3)$.

 We make similar computations for $E_{\G_1(4), t}$. At {$t=0$}, the curve admits a splitting multiplicative reduction at the singular point $(0,0)$ for all odd
 primes $p$. 
 At {$t=1$}, the singular point is $(-1/8,1/32)$ where the curve has  multiplicative reduction at odd primes, which splits if and only if $-1$ is a square mod $p$.  At {$t=\infty$}, we set $t=a/b$ and make the change of variable $(x,y)\mapsto (x/b^3,y/b^2)$. By letting $b=0$, the curve become $y^2=x^3$, whose singularity $(0,0)$ is a cusp. Hence the contribution at this irregular cusp is zero. This proves case (1) of Theorem \ref{thm:traceformula-2F1}.
 
\section{Verification for the lowest weight cases} \label{S:6}

Let $\G=(e_1,e_2,e_3)$ be one of the groups in Theorem \ref{thm:traceformula}. We now verify the trace formula for the $k=2$ case.  Recall from Theorem \ref{thm:chi-Gamma} that $\chi_\Gamma \otimes V^2(\Gamma)_\ell = \varphi_\Gamma \otimes \mathcal H(\{1/2\}, \{1\})_\ell \otimes \mathcal H(HD(\Gamma))_\ell$ with $\chi_\G =\chi_{-1}$ for $\G = (2,\infty, \infty), (2,3,\infty)$ and $(2,6,\infty)$, $\chi_\G = \chi_{-3}$ for $\G=(2,4,6)$, and $\chi_\G$ trivial for $\G=(2,4,\infty)$. 

Firstly, we recall  the dimension formula for $S_{2+k}(\G)$,  the space of holomorphic  cusp forms for $\G$ of weight $2+k$:  $$\dim S_{2+k}(\G)=-1-k+\sum_{j=1}^3 \left \lfloor \frac{k-2}2 \left (1-\frac 1{e_i}\right)\right \rfloor, \quad \mr{for } \quad  k\ge 2 \quad \mr{even},$$ (cf Theorem 4 of \cite{Yang-Schwarzian}). In particular, $S_4(\G)$ is $0$-dimensional for each $\G$ listed in the Theorem \ref{thm:traceformula}. Thus the left hand side of the Hecke trace formula \eqref{E:mot1} is 0 for each unramified prime $p$. 

When $k=2$, at each unramified prime $p$, the contributions from the elliptic points and cusps have been computed in \S  \ref{S:singular}. Namely,
 from each cusp the contribution is 1, and the contribution from an elliptic point $z$ is $\chi_{d_z}(p)p$ as in Corollary \ref{cor:XY}.  
 By Proposition \ref{prop:det} and Theorem \ref{thm:chi-Gamma}  the contribution from a generic fiber is $\chi_\G(p)\varphi_\G(p)\phi_p(1-1/\l)p^uH_p(HD(\G),1/\l)$ where $u=1$ when $\G=(2,4,6)$ and $u=0$ otherwise; also as in Proposition \ref{prop:det}, $\varphi_\G=\chi_{-1}$ for $(2,d,\infty)$ when $d=2,6$ or infinity and otherwise it is the trivial character.    Putting together, when $k=2$ the right hand side of \eqref{E:mot1} is of the form $$\displaystyle \sum_{\l=2}^{p-1} \chi_{\G}(p)\varphi_\G(p) \phi_p(1-1/\l) p^u H_p(HD(\G),1/\l) + \text{ the  contributions from singular fibers}.$$ 

When   $\G=(2,4,6)$, by Table \ref{tab:CM-fields}, at elliptic points of order 2, 4, 6, their contributions $\chi_{d_z}(p)p$ are $\chi_{-6}(p)p$, $\chi_{-1}(p)p$, $\chi_{-3}(p)p$ respectively. In this case $u=1$. Thus for each unramified prime $p$
$$ \sum_{\l=2}^{p-1} \chi_{-3}(p) \phi_p(1-1/\l)  pH_p(HD(\G),1/\l)+\chi_{-6}(p)p+\chi_{-1}(p)p+\chi_{-3}(p)p=0.$$

When $\G=(2,4,\infty)$, $u=0$, and  by Table \ref{tab:CM-fields}, at elliptic points of order 2 and 4,  their contributions  are $\chi_{-2}(p)p$ and $\chi_{-1}(p)p$ respectively, and at the cusp $\infty$, the contribution is 1.  Putting together
$$ \sum_{\l=2}^{p-1}  \phi_p(1-1/\l)  H_p(HD(\G),1/\l)+1 +\chi_{-2}(p)p+\chi_{-1}(p)p=0.$$

The same analysis applies to other cases.

Next we consider the $k=1$ case for Theorem \ref{thm:traceformula-2F1}.  Let $e_\infty\in\{\infty,3\}$ and $\G=(\infty,\infty,e_\infty)$. Note that $\dim S_3(\G)=0$ (see \cite{DS-modularforms} for example).  It follows from Proposition \ref{prop:det-I} that $\chi_\G$ is trivial and 
$$a_\G(\l,p)=\begin{cases}
    H_p(\{\frac 12,\frac12\},\{1,1\};1/\l)  &\text{if $e_\infty=\infty$}\\
     H_p(\{\frac 13,\frac23\},\{1,1\};1/\l)  &\text{if $e_\infty=3$}.
\end{cases}$$ 
     When $e_\infty=\infty$, the group $\G=\G_1(4)$ has 3 cusps; as explained in \S5.3, they contribute 1, $\left(\frac{-1}p\right)^k$, and $\frac{1+(-1)^k}2$ respectively to the trace formula. Combined with  Lemma \ref{lem:k=1-2infinfty}, the trace formula gives value 0 for $k=1$, consistent with $\dim S_3(\G)=0$.

The $e_\infty=3$ case  can be obtained similarly. In this case there are two regular cusps, whose contributions to the trace formula are $1$ and $\left(\frac{-3}p\right)^k$ for each integer $k \ge 1$.


\section{Further consequences}\label{ss:other}
\subsection{Other subgroups} \label{sec: other groups}The key step in the proof of Theorems \ref{thm:traceformula} and \ref{thm:traceformula-2F1} is to {explicitly relate} the $\ell$-adic sheaf $V^2(\G)_\ell$ or $V^1(\Gamma)_\ell$  
to the hypergeometric $\ell$-adic sheaf attached to the datum $HD(\G)$ as introduced by Katz \cite{Katz} (See Theorem \ref{thm:blist}). This method  applies to arithmetic triangle groups $\Gamma$ satisfying the following two conditions: (i)  $X_\G$ is a genus zero curve defined over $\QQ$,  (ii) $\Gamma$ has three elliptic points and cusps altogether, and these points are all $\QQ$-rational.  None of the three index-2 subgroups $(2,6,6), (3,4,4), (2,2,2,3)$ of $(2,4,6)$ listed in the diagram \ref{fig:class-II} satisfy these conditions. Since the automorphic sheaves on the subgroups are pullbacks of those on $(2,4,6)$, in the theorem below we demonstrate that, by using the pullback of the hypergeometric sheaf $\mathcal H(HD(2,4,6))_\ell$ along the explicit covering maps $\pi_i$, $i=3,2,6$, over $\QQ$ given by \eqref{eq:pi} in \S\ref{ss:3.3},  analogous Hecke trace formulae for these three subgroups  can be obtained. This approach also applies to other groups in a similar situation, as long as explicit covering maps are known.  

\begin{theorem} \label{thm:266}  
For the three index-2 subgroups $\G = (2,6,6), (3,4,4), (2,2,2,3)$ of $(2,4,6)$, choose the generator $t = t(\G)$ of $\QQ(X_\G)$ with the following values at the elliptic points of given order:
$$
\begin{tabular}{|c|c|c|c|c|c|c|c|c|c|}
\hline
$\G$&$(2,6,6)$&$(3,4,4)$&$(2,2,2,3)$\\
\hline
$t$&$(0,1/\sqrt{-3},-1/\sqrt{-3})$&$(\infty,\sqrt{-1}, -\sqrt{-1})$& $(\sqrt{-3},-\sqrt{-3},\infty, 0)$\\ \hline
CM field discriminants&$(-4,-3,-3)$&$(-3,-4,-4)$&$(-24,-24,-4,-3)$\\ \hline
$\l$& $1+1/(3t^2)$& $1/(1+t^2)$&$-t^2/3$\\
\hline
\end{tabular}
$$
\noindent so that the generator $\l(2,4,6)$ of $\QQ(X_{(2,4,6)})$ in Theorem \ref{thm:traceformula} is a rational function of $t$ given in the last row. 
Given any even integer $k\ge 2$ and a fixed prime $\ell$, for almost all primes $p \ne \ell$ where $X_\G$ has good reduction, the explicit expression for the terms in the right-hand side of (\ref{E:mot1}) as described in Theorem \ref{thm:traceformula} holds with $a_\G(t,p) = \mr{Tr}(\mr{Frob}_t \mid (V^2(\Gamma)_{\ell})_{\bar{t}})$ in \eqref{eq:2} at  
$t \in X_\G(\FF_p)$ not corresponding to an elliptic point given as follows:     

For $\G=(2,6,6)$, \[a_\G(t,p)=\begin{cases}\left(\frac{{-3}{(1+3t^2)}}p\right)pH_p(HD(2,4,6), 1/(1+\frac1{3t^2})), & {~ if~} 
     {t\neq \infty};\\ pH_p(HD(2,4,6), 1))+\left(\frac{-6}p\right)p, &{~if~} 
     {t= \infty.}\end{cases} \]
     
      For $\G=(3,4,4)$,  
\[
   a_\G(t,p)=\begin{cases}
     \left(\frac{{-3}{(1+t^2)}}p\right)pH_p(HD(2,4,6), 1+t^2),  & {~ if~} t^2\neq -1, 0;\\ 
     \left(\frac{-3}p\right)pH_p(HD(2,4,6), 1))+\left(\frac{-6}p\right)p, &{~if~}t= 0.\end{cases} \]  

For $\G=(2,2,2,3)$, 
\[
   a_\G(t,p)=
     \left(\frac{-3-t^2}p\right)pH_p(HD(2,4,6),  -3/t^2),  {~ if~} t^2\neq -3.
\] 
When $t\in\F_p$ corresponds to an elliptic point, the contribution from $t$ is the same as \eqref{eq:(5)}.
\end{theorem}

\begin{proof}
The proof follows from {noting the sheaf $V^k(\G)_\ell$ on $X_\G$ is the pullback of the sheaf $V^k(2,4,6)_\ell$ on $X_{(2,4,6)}$ 
along the two-fold $\QQ$-rational covering maps $X_\G \to X_{(2,4,6)}$} arising from the canonical models of these Shimura curves  as described in \S \ref{ss:3.3}, {as well as the known Frobenius traces given in  Theorem \ref{thm:traceformula}. Details are omitted.} 
\end{proof} 

As an application, we have the following
\begin{theorem}
For any given prime $\ell$,
 $$\rho^{BCM}_{\{\{\frac12,\frac12,\frac14,\frac34\},\{1,1,\frac16,\frac56\}\};1,\ell}\cong \left(\chi_{-1}\otimes\rho_{f_{6.4.a.a},\ell}\right)\oplus \chi_{-2}\epsilon_\ell$$ as $G_\QQ$-modules, where $\rho_{f_{6.4.a.a},\ell}$ is the $\ell$-adic representation of $G_\QQ$ associated with $f_{6.4.a.a}$.  It is equivalent to \eqref{eq:2223-wt4} which says
    $$pH_p\left (\{\frac12,\frac12,\frac14,\frac34\},\{1,1,\frac16,\frac56\};1\right)=\left(\frac{-1}p\right)a_p(f_{6.4.a.a})+\left(\frac{-2}p\right) p. $$
\end{theorem}
\begin{proof}[Proof of \eqref{eq:2223-wt4}]

For any prime $p> 3$, by the relation between Gauss sum and Jacobi sum and $J(\phi\chi, \ol \chi)=\phi\chi(-1)J(\phi\chi, \phi)$,
  \begin{eqnarray*}
   &&p H_p\left (\{\frac12,\frac12,\frac14,\frac34\},\{1,1,\frac16,\frac56\};1\right)\\
   &=&\frac{p}{1-p}\sum_\chi \frac{\g(\phi\chi)\g(\ol\chi)}{\g(\phi)}\cdot \frac{\g(\phi\chi)\g(\ol\chi)}{\g(\phi)}  \frac{\g(\chi^4)\g(\ol\chi)\g(\ol \chi^6)}{\g(\chi^2)\g(\ol \chi^2)\g(\ol \chi^3)}\chi(27/4)\\ 
     &=&\frac{p}{1-p}\sum_\chi\phi\chi(-1)J(\phi\chi, \phi)\cdot \frac{\g(\phi\chi)\g(\ol\chi)}{\g(\phi)}  \frac{\g(\chi^4)\g(\ol\chi)\g(\ol \chi^6)}{\g(\chi^2)\g(\ol \chi^2)\g(\ol \chi^3)}\chi(27/4)\\ 
 &=&-\sum_{t\in\F_p} \phi(-t(1-t)) pH_p\left (\{\frac12,\frac14,\frac34\},\{1,\frac16,\frac56\};t\right)   \\
    &\overset{\text{Lem. } \ref{lem:m=2-(2,4,6)}}=& -\phi(3)\sum_{t\in\F_p} (\phi(-3t)+1)\phi(1-t) pH_p\left (\{\frac12,\frac14,\frac34\},\{1,\frac16,\frac56\};t\right)\\&&-\left(\frac{ -2}p\right)p-\left(\frac{ -3}p\right)p-\left(\frac{ -1}p\right)p\\
   &\overset{-3t=s^2}=& -\sum_{s\in\F_p}\phi(3 (1+s^2/3)) pH_p\left (\{\frac12,\frac14,\frac34\},\{1,\frac16,\frac56\};-s^2/3\right)-\left(\frac{ -2}p\right)p-\left(\frac{ -3}p\right)p-\left(\frac{ -1}p\right)p\\
    &\overset{\text{Thm.} \ref{thm:266}}= &\phi(-1)\text{Tr}(T_p|S_4(2,2,2,3))+\left(\frac{ -2}p\right)p.
  \end{eqnarray*} 
\end{proof}

\begin{remark}
If a different $\QQ$-model on $X_{(2,6,6)}$ is chosen such that all three elliptic points are $\QQ$-rational, then it has an associated datum $HD(2,6,6)=\{\{\f12,\frac13,\frac23\}\},\{1,\frac16,\frac56\}\}$. {The two $\QQ$-models on $X_{(2,6,6)}$ are isomorphic over $\QQ(\sqrt{-3})$ and not $\QQ$. Consequently, the formula obtained using $H_p(HD(2,6,6);\l)$ agrees with that from Theorem \ref{thm:266}} only for $p\equiv 1\mod 6$. Similarly, $X_{(3,4,4)}$ has a different $\QQ$-model to which one can associate a hypergeometric datum $HD(3,4,4)=\{\{\frac1{12},\frac13\},\{1,\frac 34\}\}$, which is no longer defined over $\QQ$. The expression using $H_p(HD(3,4,4); \l)$ {agrees with that from Theorem \ref{thm:266}} only for $p\equiv 1\mod 12$. The curve $X_{(2,2,2,3)}$ does not have an associated hypergeometric datum because it has 4 elliptic points.
\end{remark} 
\begin{remark}Theorem \ref{thm:266} combined with \eqref{E:mot1} gives an explicit description of the trace of the Hecke operator $T_p$ on $S_{k+2}(\G)$ for $\G = (2,6,6), (3,4,4), (2,2,2,3)$, which consists of forms in $S_{k+2}(2,2,3,3)$ invariant under the involution $w_i$ for $i=3,2,6$ respectively. The intersection of any two out of these three spaces is the space of forms in $S_{k+2}(2,2,3,3)$ invariant under all three Atkin-Lehner involutions $w_2, w_3, w_6$, hence it is $S_{k+2}(2,4,6)$. Therefore combining Theorems \ref{thm:traceformula} and \ref{thm:266}, we also obtain an explicit expression of
\begin{align*}
{\rm Tr} (T_p \mid S_{k+2}(2,2,3,3)) =& {\rm Tr} (T_p \mid S_{k+2}(2,6,6)) + {\rm Tr} (T_p \mid S_{k+2}(3,4,4))\\
&+ {\rm Tr} (T_p \mid S_{k+2}(2,2,2,3)) -2{\rm Tr} (T_p \mid S_{k+2}(2,4,6))
\end{align*}
in terms of hypergeometric sums attached to $HD(2,4,6)$.  Our formulas are motivated in part by the results in \cite{LLT2} by the last three authors. For example it was shown there that for each prime $p>5$, 
$$
-{\rm Tr} (T_p \mid S_{6}(2,2,3,3))
= p^2H_p\left(\{\f12,\f13,\f23,\f12,\frac13,\frac23\},\{1,\f16,\f56,1,\frac56,\frac16\}; 1\right)+pa_p(f_{18.4.a.a})+\left(\frac {-1}p\right)p^2.
$$ 
Under the Jacquet-Langlands correspondence \cite{Yang-Schwarzian}, the space  $S_{k+2} ((2,2,3,3))$ corresponds to the space of newforms in $\G_0(6)$ of the same weight. 
Consequently, we obtain the following  formulae for Hecke traces:
\begin{align*}
     \mr{Tr} (T_p\mid S_{k+2} ((2,4,6))) &=  \mr{Tr} (T_p\mid S_{k+2}^{new} (\G_0(6),-,-)),\\
     \mr{Tr} (T_p\mid S_{k+2} ((2,6,6))) &=  \mr{Tr} (T_p\mid S_{k+2}^{new} (\G_0(6),\pm,-)),\\
     \mr{Tr} (T_p\mid S_{k+2} ((3,4,4))) &=  \mr{Tr} (T_p\mid S_{k+2}^{new} (\G_0(6),-,\pm)),\\
     \mr{Tr} (T_p\mid S_{k+2} ((2,2,2,3))) &=  \mr{Tr} (T_p\mid S_{k+2}^{new} (\G_0(6),\mbox{sgn},\mbox{sgn})),\\
     \mr{Tr} (T_p\mid S_{k+2} ((2,2,3,3))) &=  \mr{Tr} (T_p\mid S_{k+2}^{new} (\G_0(6))).
\end{align*}    
Here $S_{k+2}^{new} (\G_0(6),\mbox{sgn}_1,\mbox{sgn}_2)$ is the Atkin-Lehner subspace of the newforms on $\G_0(6)$ with eigenvalues $\mbox{sgn}_1\cdot 1$  and $\mbox{sgn}_2\cdot 1$ for $\omega_2$ and $\omega_3$, respectively. 
\end{remark}

\subsection{Eigenvalues of Hecke operators}
Let $\G$ be one of the groups discussed in this paper. Let $p$ be a good prime for $X_\G$. Suppose $a_r, 1 \le r \le m,$ are the eigenvalues of $T_p$ on $S_{k+2}(\G)$. By Deligne  \cite{Deligne-mf} for $\Gamma$ elliptic modular and Ohta \cite{Ohta83} for $\G$ quaternionic, for each $a_r$ there are two eigenvalues $b_r$ and $c_r$ of Frob$_p$ on $H^1(X_\G \otimes \bar \QQ, V^k(\G)_\ell)$ such that $a_r = b_r + c_r$. Furthermore, $b_r c_r = p^{k+1}$ since forms in $S_{k+2}(\G)$ have trivial central character. Therefore to determine the $m$ eigenvalues $a_r$ of $T_p$, it suffices to compute the traces of (Frob$_p)^r$ for $1 \le r\le m$ on the same cohomology space, which in turn gives rise to the values of $\sum_{1 \le r \le m} a_r^i$ for $i=1,\cdots, m$ from which we can solve for $a_r$. Since the action of Frob$_p$ is unramified, the action of (Frob$_p)^r$ agrees with that of {an unramified Frob$_{\mathfrak P_r}$ for a degree-$r$ place $\mathfrak P_r$ above $p$ (of a number field).}  Consequently the trace of (Frob$_p)^r$ for $1 \le r \le m$ can be computed from the local Frobenius traces at the points of reduction of $X_\G$ at $\mathfrak P_r$, that is,  
$$\sum_{\l \in X_{\Gamma}(\FF_{p^r})} \mr{Tr}(\mr{Frob}_\l \mid (V^k(\Gamma)_{\ell})_{\bar{\lambda}} ),$$
 the same way as the case $r=1$. In this way we express the eigenvalues of $T_p$ on $S_{k+2}(\G)$ in terms of hypergeometric character sums. 

To illustrate our point, we exhibit two examples, the first one for the elliptic modular group $\G_0(4)$ and the second for (2,4,6) from the quaternion algebra $B_6$. 

\begin{example}
  The space $S_8(\G_0(4))$ has dimension two and all are old forms. Fix an odd prime $p$ and let $a_{1,p}$ and $a_{2,p}$ be the eigenvalues of $T_p$ on $S_8(\G_0(4))$. By the trace formulae, we have the system of equations
  \begin{align*}
    -(a_{1,p}+a_{2,p}) =&\sum_{\l \in X_{\Gamma_0(4)}(\FF_p)} \mr{Tr}(\mr{Frob}_\l \mid (V^6(\Gamma_0(4))_{\ell})_{\bar{\lambda}} )\\
               =&3+\sum_{\l=2}^{p-1}\sum_{j=0}^{3}(-1)^j\CC{6-j}{j}p^j\cdot  H_p\left (\{\frac12, \frac12\},\{1,1\}; 1/\lambda\right)^{6-2j}\\
     -(a_{1,p}^2+a_{2,p}^2)=&-4p^7+\sum_{\l \in X_{\Gamma_0(4)}(\FF_{p^2})} \mr{Tr}(\mr{Frob}_\l \mid (V^6(\Gamma_0(4))_{\ell})_{\bar{\lambda}} )\\ 
     =&-4p^7+3+\sum_{\l\in \FF_{p^2}\setminus \{0,1\}}\sum_{j=0}^{3}(-1)^j\CC{6-j}{j}p^j\cdot  H_{p^2}\left (\{\frac12, \frac12\},\{1,1\}; 1/\lambda\right)^{6-2j}.\\ 
  \end{align*}
  The computations give 
  $$
  \begin{array}{c|ccccc}
  p& 5&7&11&13&17\\ \hline
   a_{1,p}+a_{2,p}&-420&2032&2184&2764& 29412\\
a_{1,p}^2+a_{2,p}^2&88200&2064512&2384928&3819848&432532872
  \end{array}
  $$
  from which we get $a_{1,p}=a_{2,p}=\frac 12(a_{1,p}+a_{2,p})$, as it should be.  
\end{example}

\begin{example}
  The space $S_{24}(2,4,6)$ is 2-dimensional. Fix an odd prime $p > 3$ and let $a_{1,p}$ and $a_{2,p}$ be the eigenvalues of $T_p$ on this space. By the trace formulae, we have the system of equations
  \begin{align*}
    -(a_{1,p}+a_{2,p}) =&\sum_{\l \in X_{(2,4,6)}(\FF_p)} \mr{Tr}(\mr{Frob}_\l \mid (V^{22}(2,4,6)_{\ell})_{\bar{\lambda}} )
               =CE(p)+\sum_{\l \in \F_p,\l \neq 0,-3}F_{11}(a_{(2,4,6)}(\l,p),p),\\
     -(a_{1,p}^2+a_{2,p}^2)
     =&-4p^{23}+CE(p^2)+\sum_{\l \in \F_{p^2},\l \neq 0,-3}F_{11}(a_{(2,4,6)}(\l,p^2),p^2),\\ 
  \end{align*}
  where $CE(q)$ is the total contribution from the elliptic points in $X_{(2,4,6)}(\F_q)$. Its value for $q=p$ is as stated in Theorem \ref{thm:traceformula}, and for $q=p^2$ it is    
$$ CE(p^2)=\sum_{N=2,4,6} \sum_{-\frac{22}{2N} \le i\le \frac{22}{2N}} p^{22} (\alpha_{N,p^2}^2/p^2)^{iN}, $$ where 
  $\alpha_{4,p^2}= J_\omega(\frac14,\frac14)^2$, $\alpha_{6,p^2}=J_\omega(\frac13,\frac13)^2$ for a generator $\omega$ of $\widehat{\FF_{p^2}^\times}$, and  $\alpha_{2,p^2}^2$ is any root of $T^2-p^2H_{p^2}(HD(2,4,6),1)T+p^4=0$. 
  The computations give 
  $$
  \begin{array}{c|ccccc}
  p& 5&7&11\\ \hline
   a_{1,p}+a_{2,p}&25248156&5764462768&1017121470024\\
   a_{1,p}^2+a_{2,p}^2&70010194261011336&60171677733273590912&3068149691314205892000288
  \end{array}
  $$
  from which we get the eigenvalues $12624078\pm  5184\beta$, $2882231384\pm 129600\beta$, and $508560735012 \pm  31363200\beta$ with $\beta=\sqrt{1296640489}$ for $p=5,7,11$, respectively.  The two Hecke eigenforms in $S_{24}(2,4,6)$ correspond to the newforms in the Newform orbit 6.24.a.d in LMFDB.  
\end{example}


\section*{Appendix: $G$-sheaves}
\label{S:Gsheaf}

Here we recall some of the theory of the functors $p_* ^G$. If $X $ is a space on which a group $G$ acts
and $p: X\to Y= G\backslash X$, then Grothendieck introduced functors 
\[
p_* ^G : \{ G - \mr{sheaves\  on \ } X\}   \to \{ \mr{sheaves\ on\ } Y\}
\]
and studied their derived functors in section V of his Tohoku paper, \cite{groth57}. He did this for ordinary topological 
spaces, but the definitions make sense in any generalized topology, such as the \'etale topology. 
By a $G$-sheaf on a $G$-space $X$ we mean a sheaf $F$ on $X$ for which there are isomorphisms
$\phi_g:g_* F \to F$, for all $g \in G$, which satisfy $\phi _{gh} = \phi _{g} \circ g_*(\phi _{h})$ as maps $  (gh)_*F  = g_* h_* F\to g_* F \to F $. Then for any open $U$ in $Y$,  $p _*^G (F) (U) = F(p^{-1} U)^G$, 
a meaningful formula since $p^{-1}U$ is a $G$-invariant open subset of $X$. 
These functors are especially useful under a discontinuity hypothesis (condition (D) of Tohoku V). 

Example: suppose that $F$ is defined by a constant sheaf $X \times M$ for a $G$-module $M$. 
Then $p_* ^G (F) $ is the sheaf defined by $G \backslash (X \times M)$ on $G \backslash X = Y$. 
For instance Bayer-Neukirch \cite{BayerNeu81} let $G = \Gamma$ be a Fuchsian subgroup of $\mr{SL}_2 (\RR)$, 
$X = \mathfrak{H}^*$ the extended upper half-plane, and $M = \mr{Sym}^k (\CCC ^2)$ where
$\Gamma$ acts on $\CCC^2 $ via its embedding in $\mr{SL}_2 (\RR)$. In this way they define 
sheaves $V^k (\Gamma) = p_* ^{\Gamma}  \mr{Sym}^k (\CCC ^2)$, at least when 
$\Gamma$ is torsion-free. Note that if $-1 \in \Gamma$, these sheaves are identically 0
if $k$ is odd. In fact, it is really the quotient $\Gamma /{\pm 1}$ that acts. When $k$ is even, 
 $\mr{Sym}^k (\CCC ^2)$ is a $\Gamma /{\pm 1}$-module, and condition (D) holds. 

In general, if condition (D) holds, then the stalk $p_* ^G (F)_y = F_x ^{G_x}$ where
$p(x)=y$, and $G_x\subset G$ is the isotropy subgroup of $x$, a finite group. Moreover, 
Grothendieck proved that under condition (D), if $G$ acts fixed-point free, the functors
$p_*^G $ and $p^{-1}$ induce an equivalence of categories 
\[
 \{ G - \mr{sheaves\  on \ } X\}   \cong \{ \mr{sheaves\ on\ } Y\}. 
\]
In particular $p^{-1} p_* ^G (F) = F$.

In Tohoku, derived functors $R^i p_* ^G$ and $R^i \Gamma _X ^G$ are introduced. 
In particular he defines abelian groups $H^n (X; G, A)$ and sheaves of abelian groups
$\mathbf{H}^n (G, A)$ on $Y$ for any abelian $G$-sheaf $A$ on $X$. Because of the two 
factorizations $\Gamma _X ^G = \Gamma ^G \Gamma _X  = \Gamma _Y p_* ^G $
there are two spectral sequences of composite functors converging to $H^n (X; G, A)$. 
The derived functors of $\Gamma ^G$ give group cohomology. 

The definitions of  these functors make sense in any topos. If $E$ is a topos and 
$G$ is a group in $E$, then there is a topos $B_{E, G}$ of $G$-objects in $E$. 
See SGA 4 IV, 2.3 and 2.4.
In our application, we consider finite group $G$ acting on a scheme $X$ with quotient 
$p: X \to Y = G\backslash X$. The category of $G$-sheaves on $X_{et}$ is a topos, by Giraud's 
criterion. From this one obtains functors $p_*^G$ and factorizations 
 $\Gamma_X^G = \Gamma^G \Gamma_X  = \Gamma_Y p_*^G $ as above, so that 
 the spectral sequences in Tohoku chapter V carry over to the \'etale topology. 

Our application will be limited to a finite flat $G$-covering 
$p: X \to Y$ of nonsingular algebraic curves for a finite group $G$.  We consider 
finite group $G$ acting on the scheme $X$ with quotient $Y = G\backslash X$ such that 
$p: X \to Y$ is faithfully flat and quasicompact. We assume that locally 
$X= \mr{Spec}(A)$, for a Noetherian ring $A$, the action is given  by an action $G$ on a ring $A$, that 
$Y= \mr{Spec}(B)$, $B = A^G$. This is faithfully flat and of finite type, which 
implies that $A $ is a projective $B$-module of finite rank, hence locally free. 
We are going to work in the \'etale topology, and sheaf will
generally mean constructible sheaf. 

If in addition $p: X \to Y$ is \'etale, the functors $p_* ^G $ and $p^* = p^{-1}$ induce 
an equivalence of categories of constructible sheaves, and in particular $p^{*} p_* ^G (F) = F$. 
Indeed, in this case, to give a $G$-sheaf $F$
on $X$ is the same as to give a descent datum on $F$. 
For more details, see \cite{KO74}. Observe that, even though we still have descent for faithfully flat maps, 
to give a $G$-structure on $F$ is not equivalent to give a descent datum. 

First  consider the case $X = \mr{Spec} (L)\to Y = \mr{Spec}(K)$ for a Galois 
extension of fields $K \subset L$. The \'etale site on $\mr{Spec} (L)$  is equivalent to the category of continuous $G_L$-modules. Under this identity, the 
functors 
$$p_*, p^*: \{\text{Sheaves  on  } X_{et}  \} \leftrightarrow  \{ \text{Sheaves  on  } Y_{et}  \}$$
become 
\[
\mr{Ind}, \mr{Res}: \{  G_L - \mr{modules} \} \leftrightarrow  \{  G_K- \mr{modules} \}. 
\]
Note that the category of $G_L$-modules $M$ with a $G = G(L/K)$-action is equivalent to the category of $G_K$-modules, so to give a 
$G_L$-module $M$ with a $G = G(L/K)$-action is equivalent to give a $G_K$-module structure on $M$.  If the sheaf $F$
corresponds to the $G_L$-module $M$ with a $G = G(L/K)$-action, then
$p_* ^G (F)$ corresponds to the $G_K$-module
\[
M = \mr{Ind} _{G_K}^{G_L} (M) ^{G(L/K)}.
\]
Recall also that in \'etale topology one considers stalks in geometric points: spectra
of algebraically closed fields. In this example, $M = F_{\bar{y}}$ , where
$\bar{y}$ is the geometric point $\mr{Spec} (\bar{L})$ of $Y$, so the above formula
reads
\[
p_* ^G (F) _{\bar{y}} = F_{\bar{x}}. 
\]
A priori, $F_{\bar{x}}$ is only an $G_L$-module, and this isomorphism is as 
$G_L \subset G_K$-modules. The existence of a $G$-structure on this sheaf means that 
in fact this $G_L$-module structure extends to a $G_K$-module structure. 

We can consider more generally 
a $G$-covering $X \to Y = \mr{Spec}(K)$, where $X = \mr{Spec}(A)$ and $A$ is an \'etale algebra
over $K$, that is $A = \prod _{i = 1}^g L_i$ where $L_i $ are separable extensions of $K$. Then 
$G$ acts transitively on the fields $L_i$. 
In this case, we have that 
\[
(p_* ^G (F))_{\bar {y}} = F_{\bar{x}_i}. 
\]
Now let $x \in X$ be a closed point of the smooth curve $X$. The local ring
$\mc{O}_{X, x}$ is a discrete valuation ring with residue field $\kappa (x)$. 
We consider the completion or henselization of this ring 
$\mc{O}^h _{X, x}$. More generally let $X = \mr{Spec}(A)$, with a henselian dvr $A$.
 We let $x, \xi$ be the closed and generic point of $X$, 
 and $\bar{x}, \bar{\xi}$ be geometric points lying over them, corresponding
to algebraic closures of $\kappa(x))$ and $\kappa (\xi)$ respectively. 

The category of sheaves 
for the \'etale topology on $X$ is equivalent to the category of triples
$(M_{\bar{x}}, M_{\bar{\xi}}, \alpha)$, where $M_{\bar{x}}$ is a $G_{\kappa (x)}$-module, 
$M_{\bar{\xi}}$ is a $G_{\kappa (\xi)}$-module, and $\alpha : M_{\bar{x}} \to M_{\bar{\xi}} ^{I_x} $
is a $G_{\kappa (x)}$-morphism, where $I_x \subset G_{\kappa (\xi)}$ is the inertia subgroup, 
so $ G_{\kappa (\xi)}/I_x \cong G_{\kappa (x)}$ ($M_{\bar{\xi}} ^{Ix} $ is automatically 
a $ G_{\kappa (\xi)}/I_x$-module). Note that to say $\alpha : M_{\bar{x}} \to M_{\bar{\xi}} ^{I_x} $
is equivalent to say that there is a morphism  $\alpha ': M_{\bar{x}} \to M_{\bar{\xi}}$ which is compatible 
with the $G_{\kappa (\xi)}$-module structures on source and target, since $I_x$ acts trivially on the source.

Let $p: X = \mr{Spec}(A)  \to Y = \mr{Spec}(B)  $ where 
$B \subset A$ be a finite extension of henselian dvr's, 
with fraction fields $\kappa(\xi)$ and $\kappa (\eta)  $ respectively. 
We assume that $\kappa (\xi)/\kappa(\eta)$  is a finite Galois extension 
with group $G = \mr{Gal} (\kappa (\xi)/\kappa(\eta)) = G_{\kappa (\eta)}/G_{\kappa(\xi)}$. We 
also assume that  $A$ is the integral closure of $B$ in $\kappa (\xi)$. It follows that 
$B = A^G$. Note that  the inertia groups satisfy $I_x = I _y\cap G_{\kappa(\xi)}$.

We wish to calculate the functor $p_* ^G$ as a morphism  $X \to Y$ in the \'etale topology, using the above description. A $G$-structure 
on a sheaf $F = (\alpha : M_{\bar{x}}\to  M_{\bar{\xi}} ^{I_x})$ on $X$ is 
equivalent to giving a $G_{\kappa(\eta)}$-structure on the diagram 
 $\alpha ': M_{\bar{x}} \to M_{\bar{\xi}}$ of $G_{\kappa(\xi)}$-modules.
 There are two stalks to compute: one at $y$ and one at $\eta$. Since 
 the map $p$ is \'etale above $\eta$, the previous sections show that
 $p_* ^G (F) _{\bar{\eta}} = M_{\bar{\xi}}$, now as a   $G_{\kappa (\eta)}$-module.
For the stalk above $y$, a basic result in \'etale cohomology shows that 
\[
A _{\bar{y}}  = H^0 (\mr{Spec} (\mathcal{O} ^{sh} _{Y, y}), \tilde{A}),\quad  \text {for any sheaf \ } A, \quad
\tilde{A} = \mr{Spec} (\mathcal{O}) ^{sh} _{Y, y} \times _Y A,  
\]
where the superscript sh denotes strict henselization. We may therefore replace 
the rings $A$ and $B$ by their strict henselizations. Then $G_{\kappa(\xi)} = I_{x}$ and
$G_{\kappa(\eta)} = I_y$. It is clear now that 
 \[
 p_* ^G (F) _{\bar{y}} = M_{\bar{x}} ^{G} = M_{\bar{x}} ^{I_y/I_x}, 
\]
for the strictly hensel case. The same formula is valid in the original (non strictly) hensel case, but
now the finite group  $\bar{G}_x:=I_y/I_x$ gives the geometric ramification of the map 
and is a subquotient of $G$. 

Summarizing: $p_*^G (F) = (\beta :    N_{\bar{y}}  \to    N_{\bar{\eta}} ^{I _y})$.
where 
\[
 N_{\bar{y}}  = M_{\bar{x}}  ^{\bar{G}_x} , \quad  N_{\bar{\eta}} = M_{\bar{\xi}}, \quad
 \beta = (\alpha ')^{I_y}, \quad \bar{G}_x = I_y/I_x.
\]
In particular   $p_*^G (F) _{\bar{y}} = F_{\bar{x}}^{\bar{G}_x}$.

We can generalize this to a $G$-covering where $A$ is now a product of the finite 
number of henselian discrete valuation rings $A_i$, corresponding to the discrete valuations 
of $\kappa (\xi)$ lying over the valuation ring  $B \subset \kappa (\eta)$. 
Then  $p_*^G (F) _{\bar{y}} = F_{\bar{x}_i}^{\bar{G}_{x_i}}$
where $x_i$ is the closed point of $A_i$. Note that the $A_i$ are permuted transitively 
by $G$. 

Let  $p : X \to Y$ be a finite flat $G$-covering of nonsingular algebraic curves
defined over a field. The constructible sheaves $F$ of interest to us will be such that 
on the complement $U\subset X$ of a finite number of points, $F \mid U$ is lisse. This means that 
$F \mid U$ is equivalent to a representation 
\[
\rho _F : \pi _1 (U, \bar{\xi}) \to \mr {Aut} (F _{\bar{\xi}}), 
\]
where $\bar{\xi}$ is an algebraic closure of the function field $\kappa (\xi)$ of the curve. 
Recall that there is a canonical epimorphism $G_{\kappa (\xi)  }\to \pi _1 (U, \bar{\xi})$, 
so $F _{\bar{\xi}}$ is a $G_{\kappa (\xi)}$-module. 

If $x \in X$ is a closed point, we obtain a morphism 
$f: S:= \mr{Spec}(\mc{O}^h _{X, x})\to X$ and we may consider the 
constructible sheaf $f^* F$ on $S$. This is described as in the previous paragraphs. 
If $s, \sigma$ denote the special and generic point of $S$, then 
$F = (\alpha, M_{\bar{s}}, M_{\bar{\sigma}} )$, where $M_{\bar{s}} = F_{\bar{x}}$, 
since $\kappa (s) = \kappa (x)$. To describe $ M_{\bar{\sigma}} $ we must 
choose an embedding $\kappa (\sigma) \subset \overline{\kappa (\xi)} $. 
Assume $\sigma \in U$, where $F$ is lisse. Then we have an isomorphism
$F_{\bar{\xi }} \cong F_{\bar{\sigma} }= M_{\bar{\sigma}}$ of $G_{\kappa (\sigma)}$-modules. 
If $F$ is a $G$-sheaf on $X$, then for the stalks we have
$p_*  ^G(F)_{\bar{y}} = F _{\bar{x}} ^{\bar{G}_x}$, where $x\in X$ is any point such that 
$p(x) = y$, and $\bar{G}_x $ is the geometric inertia of the point $x$. 

\bibliographystyle{plain}
\bibliography{ref}

\bigskip

\address{Department of Mathematics, Louisiana State University, Baton Rouge, LA 70803, USA}

\email{\href{mailto:hoffman@math.lsu.edu}{hoffman@math.lsu.edu}}
\medskip

\address{Department of Mathematics, Pennsylvania State University, University Park, PA 16802, USA}

\email{\href{mailto:wli@math.psu.edu}{wli@math.psu.edu}}, {\url{http://www.math.psu.edu/wli/}}

\medskip

\address{Department of Mathematics, Louisiana State University, Baton Rouge, LA 70803, USA}

\email{\href{mailto:llong@lsu.edu}{llong@lsu.edu}}, {\url{https://www.math.lsu.edu/~llong/}}

\medskip

\address{Department of Mathematics, Louisiana State University, Baton Rouge, LA 70803, USA}

\email{\href{mailto:ftu@lsu.edu}{ftu@lsu.edu}}, {\url{https://sites.google.com/view/ft-tu/}}

\end{document}